\documentclass[10pt,reqno]{amsart}
\usepackage{mathtools}
\usepackage[all,cmtip]{xy}
\usepackage{amssymb,amsmath}
\usepackage{wrapfig}
\usepackage{amscd}
\usepackage{graphicx}
\usepackage[pdftex,colorlinks,pagebackref]{hyperref}
\usepackage{tikz-cd}

\newcommand\FF{\mathbf{F}}
\newcommand\Z{{\mathbb Z}}
\newcommand\R{{\mathbb R}}
\newcommand\N{{\mathbb N}}

\newcommand\C{{\mathbb C}}
\newcommand\Q{{\mathbb Q}}
\newcommand\HH{{\mathbb H}}
\newcommand\p{\partial}
\newcommand\Aut{{\rm Aut\:}}
\newcommand\Cay{{\rm Cay}}
\newcommand\Sch{{\rm Sch}}
\newcommand\Sub{{\rm Sub}}
\newcommand\Stab{{\rm Stab}\,}
\newcommand\SL{{\rm SL}}
\newcommand\GL{{\rm GL}}
\newcommand\bs{\backslash}
\newcommand\MCG{\mathcal{MCG}}

\newcommand\gr{\rm gr}

\newcommand{\vol}{\mathrm{vol}}
\newcommand{\tors}{\mathrm{tors}}
\newcommand{\Hirsch}{\mathfrak h}

\newtheorem{intro}{Theorem}
\newtheorem{introth}[intro]{Theorem}

\newtheorem{theorem}{Theorem}[section]
\newtheorem{lemma}[theorem]{Lemma}
\newtheorem{corollary}[theorem]{Corollary}
\newtheorem{proposition}[theorem]{Proposition}

\newtheorem{question}[theorem]{Question}
\newtheorem{example}[theorem]{Example}

\theoremstyle{definition}
\newtheorem{definition}[theorem]{Definition}
\newtheorem{definitionintro}{Definition}

\theoremstyle{remark}
\newtheorem*{remarkintro}{Remark}
\newtheorem{remark}[theorem]{Remark}

\newcommand{\acting}{\curvearrowright}
\newcommand\findex{{fi}}

\newcommand\Basedown{\Phi} % base of the stack below
\newcommand\Baseup{\Omega} % base of the stack upstair : the group acts on it
\newcommand\cellup{\omega}% cell in the upstair base of the stack
\newcommand\Totaldown{\Sigma} %total space of the stack below

\newcommand{\Gcell}{\mathbf g}
\newcommand{\Ghom}{g}

\newcommand{\Hcell}{\mathbf h}
\newcommand{\Hhom}{h}

\newcommand{\Homcell}{\mathbf P}
\newcommand{\Homhom}{\rho}

\newcommand{\Gcellfibre}{\mathbf k}
\newcommand{\Ghomfibre}{k}

\newcommand{\Hcellfibre}{\mathbf l}
\newcommand{\Hhomfibre}{l}

\newcommand{\Homcellfibre}{\mathbf \Sigma}
\newcommand{\Homhomfibre}{\sigma}
\newcommand{\indice}{\alpha}

\newcommand\Glemma{G}
\newcommand\Alemma{A}
\newcommand\Glemone{G_1}
\newcommand \Alemone{\Alemma_1}

\newcommand\basezero{B}
\newcommand\baseone{B_1}

\newcommand\stackzero{D}
\newcommand\stackone{D_1}

\newcommand\Gammaprime{\Gamma_1}
\newcommand\Basedownprime{\Basedown_1}

\newcommand\defin[1]{\textbf{#1}\index{#1}}

\newcommand\card[1]{\mathbf{\#}^{(#1)}}

\makeindex

\date{\today}
\theoremstyle{definition}
\numberwithin{equation}{section}

\title{On homology torsion growth}

\author{Miklos Abert} 
\address{M.~A.: Renyi Institute of Mathematics, Realtanoda utca 13-15, 1053 Budapest, Hungary}
\email{abert@renyi.hu}

\author{Nicolas Bergeron} 
\address{N.~B.: ENS / PSL University, D\'epartement de Math\'ematiques et Applications, F-75005, Paris, France}
\email{nicolas.bergeron@ens.fr}
\urladdr{https://sites.google.com/view/nicolasbergeron/accueil}

\author{Mikolaj Fraczyk} 
\address{M.~F.: University of Chicago, Dep. of Mathematics, 
5734 S University Ave, Chicago, IL 60637, USA}
\email{mfraczyk@math.uchicago.edu}

\author{Damien Gaboriau} 
\address{D.~G.: ENS-Lyon, UMR CNRS 5669 / Universit\'e de Lyon, UMPA, 69007 Lyon, France}
\email{damien.gaboriau@ens-lyon.fr}
\urladdr{https://perso.ens-lyon.fr/gaboriau/}


\begin{document}

\begin{abstract}
We prove new vanishing results on the growth of higher torsion homologies for suitable arithmetic lattices, Artin groups  and mapping class groups. The growth is understood along Farber sequences, in particular, along residual chains. For principal congruence subgroups, we also obtain strong asymptotic bounds for the torsion growth.

As a central tool, we introduce a quantitative homotopical method called effective rebuilding. This constructs small classifying spaces of finite index subgroups, at the same time controlling the complexity of the homotopy. The method easily applies to free abelian groups and then extends recursively to a wide class of residually finite groups. 

\end{abstract}

\maketitle

\setcounter{tocdepth}{1} %decreases the depth of ToC.
\tableofcontents

\section{Introduction}

Let $M$ be a finite volume Riemannian manifold or a finite CW-complex. One of the most powerful and well studied invariants of $M$ are its homology groups. These abelian groups decompose to a free and a torsion part. The ranks of the free parts give the rational Betti numbers of $M$. One can also compute the mod $p$ Betti numbers from the full homology group and the Euler characteristic, as the alternating sum of Betti numbers. For a finite sheeted cover of $M$, the Euler characteristic is multiplicative in the index of the cover, but already the rational Betti numbers can behave quite erratically in this respect. To smooth this behavior out, one can consider the $j$-th $\ell^2$ homology, and measure the dimension of the corresponding Hilbert space using von Neumann dimension (Atiyah \cite{Atiyah}). Alternatively, one can consider the \defin{growth} of the $j$-th rational Betti numbers over a suitable sequence of finite sheeted covers of $M$. As shown by the L\"uck approximation theorem \cite{Luck-1994-approx}, these two attempts give the same result, called the $j$-th $\ell^2$ Betti number of $M$. 
This result naturally led to studying the growth of other homological invariants as well, like the growth of the mod $p$ Betti numbers and the growth of the torsion. 

Over the years, the interest of the community has shifted from the study of the homology of spaces $M$ to the homology of their fundamental groups. Our results are also expressed in terms of group homology.

Since the torsion grows at most exponentially in the index of the cover, the right definition of the \defin{$j$-th torsion growth of $M$} is to take the logarithm of the cardinality of the $j$-th torsion group of the covering space, normalized by the index of the cover and consider its limit, for a suitable sequence of covers.

As of now, control of the torsion (and also the mod $p$ Betti numbers) is much weaker than for the rational Betti numbers. In particular, we do not have an analogue of the L{\"u}ck approximation theorem, even in the most natural settings (see paragraph \ref{S:speculations} below where we speculate about a natural analytic invariant that should describe Betti numbers growth in positive characteristics). More than that, strikingly, we do not know a single example of a \textbf{finitely presented} group with positive \textbf{first} homology torsion growth over a decreasing  sequence of  finite index normal subgroups with trivial intersection, while, at the same time, it is a well-accepted conjecture that all sequences of congruence subgroups of arithmetic hyperbolic $3$-manifold group have this property, see \cite{BSV,BrockDunfield,EMS}. 

In this paper, we prove vanishing results on the torsion growth of higher homology groups of a natural class of residually finite groups, using a new homotopical method called effective rebuilding. We  apply this method to non-uniform higher rank lattices, mapping class groups and various Artin groups, among others.
For the lattice case, we also obtain explicit estimates of the convergence for covering maps with respect to principal congruence subgroups.

%%%%%
\subsection{Main results}

We start by stating two theorems in their simplest forms that are good showcases for the more general and technical results of the paper. 
For a group $\Gamma$, its homology groups decompose as $H_j (\Gamma , \Z)=\Z^{b_j(\Gamma)}\oplus H_j (\Gamma , \Z)_{\tors}$ 
where $b_j(\Gamma)$ is the $j$-th Betti number and where $H_j (\Gamma , \Z)_{\tors}$ is the torsion subgroup. Let $|H_j (\Gamma , \Z)_{\tors} |$ denote its cardinality. 

Our first result deals with an arbitrary sequence of finite index subgroups. 

\begin{introth} \label{T00}
Let $\Gamma = \SL_d (\Z )$ ($d \geq 3$) and let $(\Gamma_n)_{n \in \mathbb N}$ be a sequence of pairwise different finite index subgroups of $\Gamma$. Then for every degree $j \leq d-2$, we have:  
\begin{equation} \label{E:11}
\frac{\log |H_j (\Gamma_n , \Z)_{\tors} |}{[\Gamma : \Gamma_n]} \longrightarrow 0, \quad \mbox{as } n \to \infty.
\end{equation}
\end{introth}
There are deep number-theoretical motivations to study torsion in the homology of arithmetic groups. It has long been known to be related to algebraic K-theory and the Vandiver conjecture, see e.g. \cite{Soule-99,Emery} for an account on these relations. 
More recently, it has attracted further attention since, thanks to the deep work of Scholze \cite{Scholze}, one can roughly say that mod $p$ torsion classes in the homology of congruence subgroups of $\SL_d(\Z)$
parametrize field extensions $K/\mathbb{Q}$ whose Galois groups are subgroups of $\mathrm{PGL}_d (\overline{\mathbb{F}}_p)$. 
Theorem~\ref{T00}  confirms a part of a general related conjecture of Bergeron and Venkatesh that postulates that the limit in \eqref{E:11} is zero for all $d \geq 3$ and all $j$ {\bf except} when $(d,j) = (3,2)$ or $(d,j)= (4,4)$, see \cite{BV,EMS,AGMY}. 

When restricting our attention to \defin{principal congruence subgroups}
$$\Gamma (N) = \ker \left( \SL_d (\Z ) \to \SL_d (\Z / N \Z ) \right)$$
our second result gives quantitative upper bounds on the torsion growth. 

\begin{introth} \label{T0}
Let $\Gamma = \SL_d (\Z )$ ($d \geq 3$). Then for all $N \geq 1$ and $j \leq d-2$ we have
\begin{equation} \label{E:12}
\frac{\log |H_j (\Gamma (N) , \Z)_{\tors} |}{[\Gamma:\Gamma(N)]} = O \left( \frac{ \log N}{N^{d-1}} \right), \quad \mbox{as } N \to \infty.
\end{equation}
\end{introth}
We refer to Section \ref{S10} for the more general Theorem~\ref{TorsionSL} about non-uniform lattices in semi-simple Lie groups. 

The first homology case of Theorem \ref{T00}, that is, the case $j=1$, is known. In particular, it follows from \cite[Theorem 4]{AGN}.
Using the less elementary, but more classical machinery of \defin{congruence subgroup property} (CSP) on these groups, they also show the stronger estimate that there exists a constant $c$ depending only on $d$ such that for all finite index subgroups $\Gamma_0$ of $\mathrm{SL}_d (\mathbb{Z})$ we have 
\begin{equation}\label{eq: str estimate degree 1}
|H_1 (\Gamma_0 , \mathbb{Z} )_{\tors}| = |H_1 (\Gamma_0 , \mathbb{Z} )| = 
|\Gamma_0^{\rm ab}| \leq [\mathrm{SL}_d (\mathbb{Z}) : \Gamma_0 ]^c,
\end{equation}
 see \cite[\S 5.1]{AGN}. 

For $d = 3$, we expect that the degree bound $j \leq d-2$ is sharp for both theorems above. 
The general conjecture of \cite{BV} indeed postulates an exponential growth of torsion for $\SL_3 (\Z )$ in degree $2$. As an indication, big torsion groups do show up in the recent computational work of Ash, Gunnells, McConnell and Yasaki \cite{AGMY}. The degree bound $j \leq d-2$ is probably not optimal for higher degree. We shall see that it is a natural stopping point for our approach, though. 

The non-effective version of our methods already gives applications for the growth of mod $p$ Betti numbers for any prime $p$. 
Since we obtain good control over the number of cells of finite covers, it follows that the growth of the $j$-th mod $p$ Betti number is zero
under the same conditions as in Theorem \ref{T00}.
The supplementary control offered by principal congruence subgroups leads, as in Theorem \ref{T0}, to explicit estimates on the growth of the mod $p$ Betti numbers (see Theorem~\ref{TorsionSL}):

\begin{introth} \label{T000}
Let $\Gamma = \SL_d (\Z )$ ($d \geq 3$) and let $(\Gamma_n)_{n \in \mathbb N}$ be a sequence of pairwise distinct finite index subgroups of $\Gamma$. Then, for every field $K$ and for every degree $j \leq d-2$, we have  
\begin{equation}
\label{eq: torsion growth SL d}
\frac{\dim_K H_j (\Gamma_n, K)}{[\Gamma : \Gamma_n]} \longrightarrow 0, \quad \mbox{as } n \to \infty.
\end{equation}
For the principal congruence subgroups $\Gamma (N)$, we have
\begin{equation}
\frac{\dim_K H_j (\Gamma (N), K) }{[\Gamma:\Gamma(N)]} =   O \left( \frac{1}{N^{d - 1}} \right), \quad \mbox{as } N \to \infty. 
\end{equation} 
\end{introth}

Note that  the $j=1$ case of \eqref{eq: torsion growth SL d} again follows from known results. The first mod $p$ Betti number is a lower bound for the minimal number of generators, hence the rank gradient dominates the growth of the first mod $p$ Betti number.
 It is shown for instance in
 \cite[Theorem 4.14(3)]{Carderi-Gaboriau-delaSalle} that for $\SL_d (\Z )$ ($d \geq 3$), the rank gradient vanishes for arbitrary injective sequences
    of finite index subgroups, using the work of Gaboriau on cost \cite{Gab00}. 
 In a somewhat different direction, Fraczyk proved that, for the finite field $\mathbb F_2$, the estimate $\frac{\dim_{ \mathbb F_2} H_1(\Gamma,K)}{{\rm vol}(\Gamma\bs G)}=o(1)$ holds uniformly for all lattices in any simple higher rank real Lie group $G$ \cite{F18}. For non-uniform lattices in higher rank, Lubotzky and Slutsky \cite{LS} use the congruence subgroup property to give explicit and near optimal estimates on the rank and hence on the mod $p$ first Betti numbers, as well, but for uniform lattices the best known result follows from Fraczyk's work.
We wonder whether the methods implemented in our paper can be adapted to deal similarly with arbitrary sequences of lattices in a simple higher rank real Lie group $G$.
For the rank $1$ case, see the nice papers \cite{Sauer-Vol-Hom-growth-2016, Bader-Gelander-Sauer-2020}.

The results of Theorem~\ref{T000}  are new for higher homology groups.
The closest result we are aware of is due to  Calegari and Emerton and concerns  $p$-adic chains like $p^n$-congruence subgroups 
$\Gamma(p^n)$ in $\SL_d(\Z)$, see \cite{CalegariEmerton09,CalegariEmerton11,GGD14}
where the existence of a limit is established (not its value)  and the error term is estimated.
In the small degrees that our methods address, 
our error term is better, but note that their paper also deals with the difficult case around the middle dimension.
Recent works \cite{CE,C}  and conjectures (personal communication) of Calegari and Emerton suggest that in reality the bound on $|H_j (\Gamma (N), \Z)_{\rm tors}|$ (rather than its logarithm) should be polynomial in $N$ in degrees $\leq d-2$.
Let us also mention the papers of Sauer and of Kar, Kropholler and Nikolov \cite{Sauer-Vol-Hom-growth-2016, KKN-2017}, in which they prove the vanishing of the $j$-th torsion growth for a wide class of amenable groups.

\medskip
Although the results above are homological, we build homotopical machinery to obtain them. This approach traces back to \cite{Gab00,KarNikolov,BK}. Since we use a general topological approach, our work also applies to a wide class of groups such as torsion-free nilpotent groups, infinite polycyclic groups, Baumslag-Solitar groups $\mathrm{BS}(1,n)$ and $\mathrm{BS}(n,n)$ for any non-zero integer $n$, residually finite Artin groups, and does not need or assume underlying deep results like the CSP.

In particular, our method applies to mapping class groups of higher genus surfaces. For $g,b\in \mathbb N$ let 
$\indice (g,b)=2g-2$ if $g>0$ and $b=0$, $\indice(g,b)=2g-3+b$ if $g>0$ and $b>0$, and $\indice(g,b)=b-4$ if $g=0$.

\begin{introth}\label{T0000} Let ${S}$ be an orientable surface of genus $g>0$ with $b$ boundary components and let $\Gamma=\mathcal{MCG}({S})$ be its mapping class group. Let $(\Gamma_n)_{n\in\mathbb N}$ be a Farber sequence of finite index subgroups of $\Gamma$. Then for every field coefficient $K$ and every degree $j\leq \indice (g,b)$,  we have 
\begin{equation}
\frac{\dim_K H_j (\Gamma_n, K)}{[\Gamma : \Gamma_n]}  \longrightarrow 0 \quad \mbox{and} \quad \frac{\log |H_j (\Gamma_n , \Z)_{\tors} |}{[\Gamma : \Gamma_n]} \longrightarrow 0, \quad \mbox{as } n \to \infty.
\end{equation} 
\end{introth}See Definition~\ref{def: Farber sequence} for the definition of a Farber sequence. Examples include decreasing sequences of finite index normal subgroups with trivial intersection. 
Observe that for lattices in higher rank simple Lie groups, it follows from the Stuck--Zimmer Theorem \cite{StuckZimmer} via \cite{7samurai} that any injective sequence is indeed a Farber sequence.
Both $\SL_d(\mathbb Z)$ and $\mathcal{MCG}({S})$ are then handled using the same method. 

\subsection{Structure, arguments and further results}

Throughout this article, any group action on a CW-complex is required to respect the CW-complex structure.
Our starting point is the obvious observation that, while the number of cells of a finite index cover  $\overline{\Totaldown}$ of a CW-complex $\Totaldown$ is proportional to the index, it turns out that, for tori, the space $\overline{\Totaldown}$ can be \defin{rebuilt} with a much simpler cell structure.

At the origin  of the theory of ``(co)homology of groups'',
the standard homological invariants of a group $\Gamma$ are obtained as the homological invariants of a (and hence of any) compact space $\Totaldown$ with fundamental group $\Gamma$ and contractible universal cover $\widetilde{\Totaldown}$, whenever such a $\Totaldown$ exists.
In fact, instead of contractibility, the $\indice$-connectedness of $\widetilde{\Totaldown}$ is enough to compute the homological invariants up to degree $\indice$ from the cellular chain complex of $\Totaldown$.

Our strategy is to exploit the existence, for certain groups $\Gamma$ (like those appearing in Theorems \ref{T00} and \ref{T0000}), of such a complex $\Totaldown$ with nicely embedded tori-like sub-complexes so as to capitalize on the ``obvious observation'' above.

\medskip
Let $\Gamma$ be a countable group acting on a CW-complex $\Baseup$ and let $\indice$ be a positive integer. Recall that a topological space $X$ is $\indice$-connected (resp. $\indice$-aspherical) if the homotopy groups $\pi_i(X,x)$ are trivial for $i=0,\ldots, \indice$ (resp. $i=2,\ldots,\indice$).
Suppose that the CW-complex $\Baseup$ is $\indice$-connected and that for every cell $\cellup\subseteq \Baseup$ the stabilizer $\Gamma_\cellup$ acts trivially on the cell $\cellup$. We may be led to consider a finite index subgroup first or to consider a barycentric subdivision so as to ensure this last condition.
A general construction, called the \defin{Borel construction} then associates to the action $\Gamma \acting \Baseup$ an $\indice$-aspherical CW-complex $\Totaldown$ whose fundamental group is $\Gamma$.
The construction in fact naturally leads to a \defin{stack of CW-complexes} $\Pi: \Totaldown\to \Gamma \bs \Baseup$, and, roughly speaking, the fiber over a cell $\Gamma \cellup$ of $\Gamma \bs \Baseup$ is a classifying space for the stabilizer $\Gamma_{\cellup}$ of $\cellup$ under the action $\Gamma\acting \Baseup$. 
\begin{center}
\begin{tikzcd}[column sep=1.2em]
\Gamma_{\cellup}\acting \widetilde{\Pi^{-1}(\Gamma \cellup)}
\arrow[d]&\Gamma\acting \widetilde{\Totaldown} \arrow[r]  \arrow[d] 
&\Gamma\acting\Baseup \arrow[d]
\\
B(\Gamma_\cellup,1)=\Pi^{-1}(\Gamma \cellup)  \arrow[r]& \Totaldown
 \arrow[r, "\Pi"]&    \Gamma\bs\Baseup
\end{tikzcd}
\end{center}
We briefly recall  all these notions in Section~\ref{Sect: Borel construction}  (in particular Proposition~\ref{prop: existence of a stack}) as it will be useful for us. We refer to the excellent book of Geoghegan \cite[\S 6.1]{Geoghegan} for more details. 

At the risk of a spoiler, we indicate right away that we shall use the complex $\Baseup$ to be the {\em rational Tits building} $\Baseup$ for $\Gamma=\SL_d(\Z)$. 
Note that this complex $\Baseup$ is $(d-3)$-connected (see Section~\ref{sect: princ. cong. in semi-simple}), which is the sole reason why we can only treat the homology groups in the range $j \leq d-2$ in Theorems~\ref{T0} and \ref{T00}. Similarly in Theorem~\ref{T0000} for the mapping class group $\MCG(S)$, for which we shall use the {\em curve complex} as $\Baseup$ (see Section~\ref{Sect: Application to mapping class groups}). 
The key aspect of our proof is that the stabilizers $\Gamma_\cellup$ in these actions $\Gamma\acting \Baseup$ contain non-trivial free abelian normal subgroups, so that each fiber is itself a torus bundle. 
However, in its original form, the Borel construction is not yet suitable for us. 
More precisely, in order to build  ``nicer classifying spaces'' for the finite index subgroups that allow 
exploiting the structure of the stabilizers
 out of the total space $\Totaldown$ of the stack $\Pi: \Totaldown\to \Gamma \backslash \Baseup$, we make use of Geoghegan's \defin{Rebuilding Lemma} \cite[Prop. 6.1.4]{Geoghegan} (see Proposition \ref{prop: rebuilding} below). This is not yet enough. In order to get a grip on the torsion part of the homology, 
we make use of a proposition attributed to Gabber (see \cite[Prop. 3, p. 214]{Soule-99}):
\begin{equation}\log |H_j (\Totaldown , \Z)_{\rm tors} | \leq 
(\# \text{ of $j$-cells})
 \times \sup ( \log ||\p_{j+1}|| , 0).
\label{eq: Gabber ineq intro}
\end{equation}
See Section~\ref{sect: proof of Gabber's prop} where, for the convenience of the reader, we give a proof of this inequality following the point of view of \cite{BV}.
This reduces the problem of estimating the torsion to bounding the number of cells and the norms of the boundary maps. Thus, to be able to really make use of Gabber's proposition, we  have to turn the Rebuilding Lemma into an effective statement. This is the content of our Proposition \ref{P2}. In the end, this \defin{Effective Rebuilding Lemma} is a machine that provides an explicit rebuilding of the total space $\Totaldown$ of a stack of complexes $\Totaldown \to  \Gamma \backslash \Baseup$ given a rebuilding of its fibers, and that moreover provides bounds 
on the number of cells and norms of the chain boundary maps.

In various interesting situations, each fiber of $\Pi: \Totaldown\to \Gamma \backslash \Baseup$ is itself a torus bundle. 
The standard CW structure on the torus and an inductive procedure using the effective rebuilding lemma yield ``classifying spaces'' for the finite index subgroups, which will be sufficiently ``small'' to prove our asymptotic theorems.

 One needs to be able to rebuild efficiently not only free abelian groups but also finitely generated torsion-free nilpotent groups. 
 This is the content of the general Theorem~\ref{thm-UnipRewiring} stated below, which we believe to be of independent interest.

First, let us give a precise definition of a rebuilding (of good quality). This is a central notion of this paper.

\medskip
 For a CW-complex $X$ we denote by $X^{(\indice )}$ its $\indice$-skeleton and $\vert X^{(\indice )}\vert$ its number of $\indice $-cells. By convention $X^{(-1)}$ is the empty set.
\begin{definitionintro}[Rebuilding]
\label{def: Rebuilding}
Let $\indice \in \mathbb N$ and let $X$ be a CW-complex with finite $\indice$-skeleton. An $\indice$-\defin{rebuilding} of $X$ consists of the following data $(X,X',\Gcell ,\Hcell ,  \Homcell )$: \
\begin{enumerate}
\item a CW-complex $X'$ with finite $\indice$-skeleton;
\item two cellular maps 
$\Gcell \colon X^{(\indice )}\to X'{}^{(\indice )} \quad \mbox{and} \quad \Hcell \colon X'{}^{(\indice )}\to X^{(\indice )}$
that are homotopy inverse to each other up to dimension $\indice-1$, i.e., $\Hcell \circ \Gcell_{\restriction X^{(\indice-1)}}\sim \mathrm{id}_{\restriction X^{(\indice-1)}}$ within $X^{(\indice)}$ and 
$\Gcell \circ \Hcell_{\restriction X'{}^{(\indice-1)}}\sim \mathrm{id}_{\restriction X'{}^{(\indice-1)}}$ within $X'{}^{(\indice)}$;
\item a cellular homotopy $\Homcell \colon [0,1] \times X^{(\indice-1)}\to X^{(\indice)}$  between the identity and $\Hcell \circ \Gcell$,
i.e., $\Homcell (0, .)=\mathrm{id}_{\restriction X^{(\indice-1)}}$ and $\Homcell (1,.)=\Hcell \circ \Gcell_{\restriction X^{(\indice-1)}}$
\end{enumerate}
\end{definitionintro}

\begin{definitionintro}[Quality of a rebuilding]
\label{def: Rebuilding and quality}
Given real numbers  ${T} , {\kappa}\geq 1$, we say that $(X,X',\Gcell ,\Hcell , \Homcell )$ is an $\indice$-rebuilding of \defin{quality $({T}, {\kappa})$} if 
we have 
\begin{align}
\forall j \leq \indice, \quad |X'{}^{(j)}|  & \leq  {\kappa}{T}^{-1}|X^{(j)}| \tag{cells bound}\\
\forall j \leq \indice , \quad \log \|\Ghom_j \|,\log \| \Hhom_j \|,\log \| \Homhom_{j-1} \|,\log \|\p'_{j} \|  & \leq   {\kappa}(1+\log {T} ) \tag{norms bound}
\end{align}
where $\p'$ is the boundary map on $X'$, we denote by $\Ghom$ and $\Hhom$ the chain maps respectively associated to $\Gcell$ and $\Hcell$, and $\Homhom \colon C_\bullet(X)\to C_{\bullet+1}(X)$ is the chain homotopy induced by $\Homcell$ in the cellular chain complexes: 
\begin{equation}
\label{eq: cellular chain homotopy}
\begin{tikzcd}[column sep=1.2em]
C_{\indice}(X) \arrow[r, "\p_{\indice}"] \arrow[d, "\Ghom_{\indice}" left]
& \cdots \arrow[l, bend left, "\Homhom_{\indice-1}" ] & 
\arrow[r] \cdots &  
C_{1}(X) \arrow[r] \arrow[r, "\p_1"] \arrow[d, "\Ghom_1" left] \arrow[l, bend left, "\Homhom_1" ] &
C_{0}(X) \arrow[d, "\Ghom_0" left]\arrow[l, bend left, "\Homhom_0" ]  \\
C_{\indice}(X') \arrow[r, "\p'_{\indice}"] \arrow[u, xshift=0.35em, "\Hhom_{\indice} " right]
&
\cdots 
& 
\arrow[r] \cdots &  
C_{1}(X') \arrow[r, "\p'_1 "] \arrow[u, xshift=0.35em, "\Hhom_1" right]&
C_{0}(X'), \arrow[u, xshift=0.35em, "\Hhom_0" right]
\end{tikzcd}
\end{equation}
and the norms $\| \cdot \|$ are derived from the canonical $\ell^2$-norms on the cellular chain complexes.
\end{definitionintro}

We will simplify the notation and write $(X,X')$ instead of $(X,X',\Gcell ,\Hcell ,  \Homcell )$ when the explicit cellular maps are not relevant.

The definition above captures an intrinsic tension between ``having few cells'' and ``maintaining tame norms''. 

Given a finite cover $X_1\to X$ (of large degree), our main task is to construct a rebuilding $(X_1,X'_1)$ of sufficiently good quality $(T_1 , \kappa_1)$. In some cases, it is possible to take $\kappa_1 = \kappa (X)$ independently of the cover and $T_1$ linear in the degree, see  Sections \ref{S:unip} and \ref{S:unipext}. In particular, for finitely generated torsion-free nilpotent groups (called \defin{unipotent lattices} in the text), the precise efficient rebuilding we obtain can be stated as follows. 
\begin{introth}\label{thm-UnipRewiring}
Let $\Lambda$ be a finitely generated torsion-free nilpotent group. If $Y_0$ is a compact $K(\Lambda,1)$ space, 
 there exists a constant ${\kappa}_{Y_0}\geq 1$ such that for every finite index subgroup $\Lambda_1\leq \Lambda$, the cover $Y_1=\Lambda_1\bs \tilde Y_0$ admits an $\indice$-rebuilding $(Y_1,Y_1',\Gcell , \Hcell , \Homcell )$ 
of quality $([\Lambda : \Lambda_1], {\kappa}_{Y_0})$ for every $\indice$.
\end{introth}
Recall that a $K(\Lambda,1)$ space is a CW complex with fundamental group isomorphic with $\Lambda$ and whose universal cover is contractible.

\begin{remarkintro} 
Since the number of $j$-cells of $Y_1$ is $\vert Y_1^{(j)}\vert = [\Lambda:\Lambda_1] \vert Y_0^{(j)}\vert$, the number of $j$-cells of the $K(\Lambda_1,1)$ space $Y_1'$ satisfies the bound $\vert (Y_1')^{(j)}\vert\leq {\kappa}_{Y_0} [\Lambda:\Lambda_1]^{-1}\vert Y_1^{(j)}\vert ={\kappa}_{Y_0}\vert Y_0^{(j)}\vert$, independent of $\Lambda_1$.
\end{remarkintro}

\begin{remarkintro} The proof of Theorem \ref{thm-UnipRewiring} (in Section~\ref{S:unip}) yields for various $\Lambda_1$ a rebuilding $Y_1 '$ with only $O(2^{\Hirsch})$ cells, where $\Hirsch$ is the Hirsch length of $\Lambda$. This is, in fact, ``the'' minimal number of cells that a $K(\Lambda , 1)$ CW-complex can have, since $\sum_j \dim_\Q H_j (\Lambda,\mathbb Q)=2^\Hirsch.$ 
\end{remarkintro}

Theorem \ref{thm-UnipRewiring} is used for the proof of Theorem~\ref{T0}; the latter is in fact deduced from a general result: Theorem~\ref{Tmain} proved in Section~\ref{sect: proof of main th}. One key feature of the \defin{principal} congruence subgroup $\Gamma(N)$ of Theorem~\ref{T0} is that it intersects every infinite unipotent subgroup of $\SL_d (\Z)$ along a subgroup of index at least $N$. Note that general sequences of finite index subgroups of $\SL_d (\Z)$ do not have this property, e.g., for each positive integer $N$ the \defin{non-principal} congruence subgroup
$$\Gamma_0 (N) = \left\{ \left( \begin{array}{cc} a & b \\ c^\top & D \end{array} \right) \in \SL_d (\Z ) \; : \; a \in \Z, \ b,c \in \Z^{d-1}, \ D \in \GL_{d-1} (\Z) \mbox{, and } N | c \right\} \subseteq \SL_d (\Z )$$
contains the whole group of upper-triangular unipotent matrices with integer coefficients. To prove Theorem \ref{T00} we need to show that if $(\Gamma_n )$ is an injective sequence of finite index subgroups of $\SL_d (\Z)$, then $\Gamma_n$ intersects most of the infinite unipotent subgroups of $\SL_d (\Z)$ along subgroups of large index. 

In general, to deal with arbitrary Farber sequences of a given residually finite group $\Gamma$ we introduce a property that we believe to be of independent interest. 
We say that $\Gamma$ has the \defin{cheap $\indice$-rebuilding property} (Definition \ref{def-CheapReb}) if it admits a 
$K(\Gamma,1)$ space $X$ with finite $\indice$-skeleton and a constant ${\kappa}_X\geq 1$ that
satisfy the following property. For every Farber sequence $(\Gamma_n)_{n\in\mathbb N}$ for $\Gamma$ and $T\geq 1$, there exists $n_0$ such that for $n\geq n_0$ the cover $X_n:=\Gamma_n\bs \tilde X$ admits an $\indice$-rebuilding $(X_n,Y_n)$ of quality $(T,{\kappa}_X)$. We note that cheap $0$-rebuilding simply means that $\Gamma$ is infinite.

We establish a robust bootstrapping criterion for a group to have the cheap $\indice$-rebuilding property:
\begin{introth}[Theorem~\ref{thm-AdjRebStack}]\label{thm-AdjRebStack-intro}
Let $\Gamma$ be a residually finite group acting on a CW-complex $\Baseup$ in such a way that any element stabilizing a cell fixes it pointwise. Let $\indice \in \mathbb{N}$ and assume that the following conditions hold:
\begin{enumerate}
\item $\Gamma\bs \Baseup$ has finite $\indice$-skeleton; 
\item $\Baseup$ is $(\indice-1)$-connected;
\item The stabilizer of each cell of dimension $j \leq \indice$  has the cheap $(\indice -j)$-rebuilding property.
\end{enumerate}
Then, $\Gamma$ itself has the cheap $\indice$-rebuilding property. 
\end{introth}

As a corollary, we get that the cheap $\indice$-rebuilding property is an invariant of commensurability. Also, we get that it holds for  
all infinite polycyclic groups for all $\indice$. It also passes from a normal subgroup $N\triangleleft \Gamma$ to $\Gamma$ if $N\bs \Gamma$ is of type $F_{\indice}$. See Corollary~\ref{cor-CheapRebGroups}. 
More significantly, we show 
\begin{introth}\label{TCheapReb}
Let $\indice \geq 0$. The following groups have the cheap $\indice$-rebuilding property:
\begin{enumerate}
\item \label{item Arithmetic lattices-intro} Arithmetic lattices of $\mathbb Q$-rank at least $\indice+1$ (Theorem~\ref{thm-LatticesRebuilding}).

\item \label{item Artin groups with flag complex}
Finitely generated residually finite Artin groups satisfying the $K(\pi,1)$ conjecture and whose nerve is $(\indice-1)$-connected (Theorem~\ref{Ex:Artin groups}).

\item \label{item mapping class group-intro} Mapping class groups $\mathcal{MGC}(S)$ where $S$ is a surface of genus $g>0$ with $b>0$ boundary components and $2g-3+b\geq \indice$ or $g>0$, $b=0$ and $2g-2\geq \indice$ (Theorem~\ref{th: MCG has CRP k(g,b)}).

\end{enumerate}
\end{introth}

The ``cells bound'' condition of Definition~\ref{def: Rebuilding and quality} in the cheap $\indice$-rebuilding property is very much related to the notion of slowness introduced in \cite{BK}. This condition already implies the vanishing of the growth of homology over arbitrary fields for a wide class of groups. In particular, a group with the cheap $2$-rebuilding property is finitely presented and economical, in the sense of \cite{KarNikolov}, with respect to any Farber chain of finite index subgroups. For controlling the torsion, however, we need the full force of the cheap $\indice$-rebuilding property:

\begin{introth}[Theorem~\ref{prop-CheapRebTorsion}] \label{Tintrotors}
Let $\Gamma$ be a countable group of type $F_{\indice+1}$ that has the cheap $\indice$-rebuilding property for some non-negative integer $\indice$. Then, for every Farber sequence $(\Gamma_n)_{n \in \mathbb N}$, for every coefficient field $K$ and $0\leq j\leq \indice$ we have 
\[\lim_{n\to \infty} \frac{\dim_K H_j (\Gamma_n , K)}{[\Gamma : \Gamma_n]} = 0 \quad \mbox{and} \quad \lim_{n\to \infty} \frac{\log | H_j(\Gamma_n,\mathbb Z)_{\rm tors}|}{[\Gamma:\Gamma_n]}=0.\]
\end{introth}
In analogy with the L\"uck approximation theorem \cite{Luck-1994-approx} which identifies the first limit with the usual $\ell^2$-Betti numbers $b_j^{(2)} (\Gamma ; K)$ when $K=\Q$, we will loosely write $$b_j^{(2)} (\Gamma ; K) = \limsup_{n\to \infty} \frac{\dim_K H_j (\Gamma_n , K)}{[\Gamma : \Gamma_n]}$$ 
once the sequence $(\Gamma_n )$ is fixed. 
Note that non-abelian finitely generated free groups do not have the cheap $1$-rebuilding property (since $b_1^{(2)} (\FF_r; \Q) =r-1\not= 0$).

According to Theorem \ref{TCheapReb}~\eqref{item Arithmetic lattices-intro} $\SL_d (\Z)$ has the cheap $(d-2)$-rebuilding property. Theorem \ref{Tintrotors}, therefore, implies Theorem \ref{T00}. Similarly, Theorem \ref{TCheapReb}~\eqref{item mapping class group-intro} and Theorem \ref{Tintrotors} imply Theorem \ref{T0000}.
 
Theorem \ref{Tintrotors} is related to \cite[Theorems 7 and 9]{AGN} where Abert, Gelander and Nikolov investigate the first homology torsion growth for
\defin{chain-commuting}\footnote{Note that in that paper the authors named these groups ``right-angled'', which raised concerns in the community 
since it is too close to the well-established name ``right-angled Artin'' which may generate confusion between these two different notions. So, we suggest the new name ``chain-commuting''.}, 
i.e. groups that admit a finite generating list $\{\gamma_1 , \ldots , \gamma_m \}$ of elements of infinite order such that $[\gamma_i , \gamma_{i+1}] = 1$ for $i=1 , \ldots , m-1$.
This is a class of groups featured in \cite{Gab00} where $\SL_d (\Z )$ ($d \geq 3$) being a chain-commuting group is exploited to compute its cost. These groups have the cheap $1$-rebuilding property, see Proposition \ref{P:rag}. We therefore recover that finitely presented chain-commuting groups have vanishing first homology torsion growth along any Farber sequence.  Note that already in degree $j=1$, Theorem~\ref{Tintrotors} also establishes the vanishing of the homology torsion growth for a natural class of groups where \cite{AGN} does not apply (see Example~\ref{ex: non rag semidirect F x Z2}). Moreover, thanks to Theorem \ref{TCheapReb}~\eqref{item Artin groups with flag complex}, it applies to quite general Artin groups  in higher degrees and yields:

\begin{introth} \label{TintroArtin}
Let $\Gamma$ be a finitely generated residually finite Artin group satisfying the $K(\pi,1)$ conjecture and whose nerve is $(\indice-1)$-connected . Then, for every Farber sequence $(\Gamma_n)_{n \in \mathbb N}$, coefficient field $K$ and $0\leq j\leq \indice$ we have 
\[b_j^{(2)} (\Gamma ; K ) =0 \quad \mbox{and} \quad \lim_{n\to \infty} \frac{\log | H_j(\Gamma_n,\mathbb Z)_{\rm tors}|}{[\Gamma:\Gamma_n]}=0.\]
\end{introth}

When $\Gamma$ is a right-angled Artin group (\defin{RAAG}) the first part of the theorem is not new: 
In fact Avramidi, Okun and Schreve \cite{Avramidi} have even computed all the $\ell^2$-Betti numbers $b_j^{(2)} (\Gamma ; K)$. Their result  is that $b_j^{(2)} (\Gamma ; K) = \overline{b}_{j-1} (L ; K)$, the reduced Betti number of the nerve $L$ of $\Gamma$. Thus $b_j^{(2)} (\Gamma ; K )$ is indeed equal to $0$ when $L$ is $(j-1)$-connected. This shows that Theorem \ref{TCheapReb}~\eqref{item Artin groups with flag complex} is optimal for RAAGs. 
Building on their computation of the $\ell^2$-Betti numbers $b_j^{(2)} (\Gamma ; \mathbb{F}_2)$,  Avramidi, Okun and Schreve also prove that a RAAG $\Gamma$ whose nerve is a flag triangulation of $\mathbb{R}P^2$ has exponential homology torsion growth in degree $j=2$, see \cite[Corollary 3]{Avramidi}. Since $\Gamma$ has the cheap $1$-rebuilding property (it is in fact chain-commuting), this shows that Theorem \ref{Tintrotors} is in a certain sense optimal: for $\indice=1$, there are groups that have the cheap $\indice$-rebuilding property and have exponential torsion growth in degree $j={\indice+1}$.

\subsection{Speculations and questions} \label{S:speculations}

\newcommand\Image{{\mathrm {Im} }}%             \Image      Im d'un morphisme
\newcommand{\toph}[3]{h_{\mathrm{top}}^{#1}(#2 \acting #3)} % topological sofic entropy
\newcommand{\tophsimple}[2]{h_{\mathrm{top}}^{#1}(#2)} % topological sofic entropy without the group action specified

One thing that makes rational Betti numbers powerful invariants is that their growth is connected with $\ell^2$-cohomology.
In particular, by the L\"uck approximation theorem, the rational homology growth can be expressed as the corresponding $\ell^2$-Betti number. Thus it is natural to ask, even at a conceptual level, what is the analytic invariant behind the homology growth in positive characteristics.  
In \cite{Gab-Seward-2019}, the last-named author with Brandon Seward drew a connection between finite-field homology problems and sofic entropy problems, such as, for example, whether entropy depends upon the choice of sofic approximation $\Sigma$. 
Let $\Gamma$ act freely, cocompactly on a contractible (or at least $j$-connected) simplicial complex $L$. Consider the coboundary maps with coefficients in a finite field ${K}$:
\begin{equation*}
\begin{matrix}
C^{j-1}(L,{K}) & \overset{\delta^{j}}{\longrightarrow} & C^{j}(L,{K})  & \overset{\delta^{j+1}}{\longrightarrow} & C^{j+1}(L,{K}) .
\end{matrix}
\end{equation*}
In particular, 
$\Image(\delta^{j})=\ker(\delta^{j+1})$ (equivalently $H^{j}(L,{K})=0$).
\\
In analogy with the work of Elek \cite{Elek-2002-entropy-Betti-numbers}, the authors were led to introduce the notion of the \defin{$j$-th sofic entropy Betti number} $\beta^{j,\Sigma}_{{K}} ( \Gamma)$ of the group $\Gamma$ over the field ${K}$ for a sofic approximation $\Sigma$ of $\Gamma$ as a measurement of the violation of a Yuzvinsky addition formula (see \cite[p. 6]{Gab-Seward-2019}):
\begin{equation*}
\beta^{j,\Sigma}_{{K}} ( \Gamma) \cdot \log |{K}|+\tophsimple{\Sigma}{C^{j-1}(L, {K})}  = \tophsimple{\Sigma}{\ker(\delta^j)} + \tophsimple{\Sigma}{\Image(\delta^{j})},
\end{equation*}
where $\tophsimple{\Sigma}{V}$ is the topological sofic entropy of the action of $\Gamma$ on the subshift $V$.
They proved in particular \cite[Cor. 5.4]{Gab-Seward-2019} that:
{\em If $\Gamma$ is residually finite, $\Sigma$ comes from a Farber chain $(\Gamma_n)$,
and the entropy $\tophsimple{\Sigma}{\ker(\delta^j)}$ is achieved as a limit rather than a limit-supremum then
$$\limsup_{n \rightarrow \infty} \frac{\dim_{\mathbb K} H^j(\Gamma_n \backslash L, {\mathbb K})}{|\Gamma : \Gamma_n|} =\beta^{j,\Sigma}_{{\mathbb K}} ( \Gamma).$$}

To conclude this introduction, we finally raise the following:
\begin{question} \label{QIntro}
Let $\Gamma$ be the fundamental group of a finite volume hyperbolic $3$-manifold. Does $\Gamma$ have the cheap $1$-rebuilding property?
\end{question}

It is tempting to believe that the answer is no. 
These groups have zero first $\ell^2$-Betti number, so the Betti number criterion does not apply. 
But it is conjectured that for arithmetic hyperbolic $3$-manifold groups the torsion in degree $1$ grows exponentially along Farber sequences of congruence subgroups, see \cite{BSV,BrockDunfield,EMS}.

\subsection*{Acknowledgement}We would like to thank Ian Biringer, Ross Geoghegan, Gr{\'e}gory Ginot, Athanase Papadopoulos, and Shmuel Weinberger
for helpful conversations.
We especially thank Boris Okun and Kevin Schreve for reading a preliminary version of this paper and making valuable comments.
We also thank Emma Bergeron for drawing our pictures. 
We thank Clara L\"oh and Matthias Uschold for pointing us a gap in the proof of one of our applications, namely Proposition~\ref{P:rag}.
A complete proof of its statement can be found in Uschold's paper \cite{Usc22}.
M.~A. and M.~F. 
are partially supported by ERC Consolidator Grant 648017. M.~A. is further supported by KKP 139502 project.
D.~G. was partially supported by the ANR project GAMME (ANR-14-CE25-0004) and by the LABEX MILYON (ANR-10-LABX-0070) of Universit{\'e} de Lyon, within the program ``Investissements d'Avenir" (ANR-11-IDEX-0007) operated by the French National Research Agency (ANR).
 M.~F. would like to thank IAS Princeton for an excellent working environment during a crucial stage of this project.

\section{The Borel construction and Geoghegan rebuilding} \label{Sect: Borel construction}

Let $\Gamma$ be a countable group acting on a CW-complex $\Baseup$. 

Technically speaking,  the \defin{Borel construction} is a trick that converts an
action of a group $\Gamma$ on a space $\Baseup$ into a \textbf{free action} of $\Gamma$ on a homotopy equivalent
space $\Baseup'$. See for instance \cite[3.G.2]{Hatcher-book-Algeb-topo}.
Namely, take $\Baseup' = \Baseup \times E\Gamma$ with the diagonal action of $\Gamma$, $\gamma(y, z) = (\gamma y, \gamma z)$ where $\Gamma$ acts on $E\Gamma$ (the universal cover of some classifying space $B\Gamma$) as deck transformations. The diagonal action is free, in fact a
covering space action, since this is true for the action in the second coordinate. 
The orbit space of this diagonal action is usually denoted $\Baseup\times_{\Gamma} E\Gamma$. 
We now briefly explain the ``stack" interpretation of Geoghegan \cite[\S 6.1]{Geoghegan} alluded to in the introduction.

\medskip

A cellular map $\Pi : \Totaldown \to \Basedown$ between two CW-complexes is a \defin{stack of CW-complexes} with \defin{base space} $\Basedown$, \defin{total space} $\Totaldown$ and CW-complexes called \defin{fibers} $F_e$ over $e$, if for each $n \geq 1$ (denoting by $E_n$ the set of $n$-cells of $\Basedown$) there is a cellular map 
\begin{equation}f_n : \bigsqcup_{e \in E_n} F_e \times \mathbb{S}^{n-1} \to \Pi^{-1} (\Basedown^{(n-1)})\end{equation}
and a homeomorphism
\begin{equation} \label{k}
k_n : \Pi^{-1} (\Basedown^{(n-1)}) \cup_{f_n} \left( \bigsqcup_{e \in E_n} F_e \times \mathbb{B}^{n} \right) \to \Pi^{-1} (\Basedown^{(n)})\end{equation}
satisfying:
\begin{enumerate}
\item $k_n$ agrees with the inclusion on $\Pi^{-1} (\Basedown^{(n-1)})$,
\item $k_n$ maps each cell onto a cell, and
\item the following diagram commutes up to homotopy relative to $\Pi^{-1} (\Basedown^{(n-1)})$:
\begin{equation} \label{E:diag23}
\xymatrix{\Pi^{-1} (\Basedown^{(n-1)}) \sqcup \left( \bigsqcup_{e \in E_n} F_e \times \mathbb{B}^{n} \right) \ar[d]^{\mathrm{quotient}} \ar[rd]^{(\Pi_{|} , \chi_e \circ \mathrm{pr}_2)} & 
\\
\Pi^{-1} (\Basedown^{(n-1)}) \cup_{f_n} \left( \bigsqcup_{e \in E_n} F_e \times \mathbb{B}^{n} \right) \ar[d]^{k_n} & \Basedown^{(n)} 
\\
\Pi^{-1} (\Basedown^{(n)}) \ar[ru]^{\Pi_{|}} & }\end{equation}
\end{enumerate}
where $(\Pi_{|} , \chi_e \circ \mathrm{pr}_2):\Pi^{-1} (\Basedown^{(n-1)}) \sqcup \left( \bigsqcup_{e \in E_n} F_e \times \mathbb{B}^{n} \right) \to \Basedown^{(n)}$ is the restriction of $\Pi$ on $\Pi^{-1} (\Basedown^{(n-1)})$, while 
on each $F_e \times \mathbb{B}^{n}$, it is the composition of the projection onto $\mathbb{B}^{n}$ followed by the characteristic map $\chi_e$ of the cell $e\in E_n$ in the
CW-complex $\Basedown$. Thus there is a homotopy 
$$\mathfrak{H} : \left( \Pi^{-1} (\Basedown^{(n-1)}) \sqcup \left( \bigsqcup_{e \in E_n} F_e \times \mathbb{B}^{n} \right) \right) \times [0,1] \to \Basedown^{(n)}$$
such that 
$$\mathfrak{H}_{|  \Pi^{-1} (\Basedown^{(n-1)}) \times [0,1]} = \Pi_{|}, \quad \mathfrak{H} (\cdot , 0) = (\Pi_{|} , \chi_e \circ \mathrm{pr}_2), \quad \mbox{and} \quad 
\mathfrak{H} (\cdot , 1) = \Pi_{|} \circ k_n \circ \mathrm{quotient}.$$ 

Note that the strong commutation of diagram \eqref{E:diag23} required in the definition of stacks in  \cite[Section 6.1, p. 147 ]{Geoghegan} (2008 version of the book) is in fact too strong and that \cite[Proposition 6.1.4]{Geoghegan} needs the commutation to occur up to homotopy. Both the flaw and the way to fix it have been pointed out to us by Boris Okun and Kevin Schreve. We are extremely grateful to them for that.
The relevant corrections of Ross Geoghegan's book appear on his webpage \url{math.binghamton.edu/ross/tmgt}.

Following \cite[\S 6.1]{Geoghegan},
the Borel construction for $\Gamma\acting \Baseup$  eventually takes the form of a stack $\Pi : \Totaldown \to \Gamma \backslash \Baseup$ below:

\begin{proposition} \label{prop: existence of a stack}
Let $\Gamma$ be a countable group acting on a simply connected CW-complex $\Baseup$ so that for every cell $\cellup\subseteq \Baseup$ the stabilizer $\Gamma_\cellup$ acts trivially on $\cellup$. Write $\Basedown=\Gamma\backslash \Baseup$. Then there exists a stack of CW-complexes $\Pi : \Totaldown \to \Basedown$ with fiber $F_e$ over $e$ such that
\begin{enumerate}
\item The fundamental group $\pi_1 (\Totaldown)$ is isomorphic to $\Gamma$.
\item For each cell $e$ of $\Basedown$ the fiber $F_e$ is aspherical and $\pi_1( F_e)\cong \Gamma_\cellup$, where $\cellup$ is any cell in $\Baseup$ above $e$.
\end{enumerate}
Moreover, if $\Baseup$ is $n$-connected then the universal cover $\widetilde{\Totaldown}$ of the total space $\Totaldown$ is $n$-connected. 
\end{proposition} 
Recall that a CW-complex space $Y$ is said to be \defin{$n$-connected} if it is connected and all its homotopy groups $\pi_j Y$ are trivial for $0\leq j\leq n$. By convention, a $(-1)$-connected space will just be an arbitrary topological space.  

Note that if $\cellup$ and $\cellup'$ are two cells of $\Baseup$ above a cell $e$ of $\Basedown$, they are  in the same $\Gamma$-orbit and thus the stabilizers $\Gamma_\cellup$ and $\Gamma_{\cellup '}$ are conjugate in $\Gamma$. In particular these stabilizers are isomorphic, so there is no harm in writing $\Gamma_e:=\Gamma_\cellup$.

In practice it is often desirable to replace the fibers $F_e$ by different CW-complexes of the same homotopy types. 
Geoghegan's stack decomposition and his \defin{rebuilding Lemma} \cite[Proposition 6.1.4]{Geoghegan} makes this possible:

\begin{proposition}[Rebuilding Lemma, Geoghegan {\cite[Proposition 6.1.4]{Geoghegan}}] \label{prop: rebuilding}
Let $\Pi:\Totaldown\to \Basedown$ be a stack of CW-complexes.
If for each cell $e$ of $\Basedown$ we are given a CW-complex $F_e '$ of the same homotopy type as $F_e$, then there is a stack of CW-complexes 
$\Pi' : \Totaldown' \to \Basedown$ with fiber $F_e '$ over $e$, and a homotopy equivalence $\Gcell$ making the following diagram commute up to homotopy over each cell:
\begin{equation*}
\xymatrix{ \Totaldown \ar[rd]_{\Pi}  \ar[rr]^\Gcell &   & \Totaldown ' \ar[ld]^{\Pi'} 
\\
& \Basedown  & }
\end{equation*}
\end{proposition} 

\medskip

Proposition \ref{prop: rebuilding} allows one, up to homotopy equivalence, to replace each fiber $F_e$ in Proposition~\ref{prop: existence of a stack} by a prescribed classifying space for $\Gamma_\cellup$. 

Recall that a group $\Gamma$ is of type $F_{n}$ if it admits a classifying space whose $n$-skeleton is finite; equivalently if it admits a
$(n-1)$-aspherical CW-complex $X$ with finite $n$-skeleton and $\pi_1(X) \simeq \Gamma$ (since turning $X$ into an aspherical complex can be made by adding only cells of dimension $\geq n+1$). 

Let $\indice \in \mathbb{N}$. Suppose that each group $\Gamma_e$ is of type $F_{\indice}$. Thanks to the rebuilding lemma we may then assume that:
\begin{center}
Each fiber $F_e$ of the stack $\Pi : \Totaldown \to \Basedown$ has a finite $\indice$-skeleton.
\end{center}
If we suppose furthermore that $\Baseup$ is $(\indice-1)$-connected then the universal cover of the total space $\widetilde{\Totaldown}$ is also $(\indice-1)$-connected. The cellular chain complex $C_q (\widetilde{\Totaldown})$ therefore gives a partial resolution 
$$C_{\indice} (\widetilde{\Totaldown}) \to \cdots \to C_0 (\widetilde{\Totaldown}) \to \Z \to 0$$
of $\Z$ by free $\Z[\Gamma]$-modules of finite rank. 

\begin{remark} 
If $\overline{\Gamma}\leq \Gamma$ is a subgroup of finite index, then in the commuting diagram
\begin{equation} 
\xymatrix{\overline{\Totaldown}=\overline{\Gamma}\backslash\widetilde{\Totaldown} \ar[r]^{\overline{\Pi}} \ar[d]^{\zeta} \ar@{}[rd]|{\circlearrowleft}& \overline\Basedown=\overline{\Gamma}\backslash\Baseup \ar[d]
\\
\Totaldown=\Gamma\backslash\widetilde{\Totaldown}\ar[r]^{\Pi} & \Basedown=\Gamma\backslash\Baseup
}
\end{equation}
the map $\overline{\Pi}$ 
is naturally a stack satisfying (1) and (2) of Proposition \ref{prop: existence of a stack}, the map $\zeta$ is a finite cover. For any $n \in \mathbb{N}$ we thus have
$$\vert \overline{\Totaldown}^{(n)} \vert = [\Gamma:\overline{\Gamma}]  \vert \Totaldown^{(n)} \vert|.$$ 
\end{remark}

Proposition \ref{prop: rebuilding} gives a way to improve the covering space $\overline{\Totaldown}$ by simplifying the fibers of the stack $\overline{\Pi}$. Our next goal is to make this procedure explicit enough to control the boundary maps in (cellular) homology. We first recall some basic facts about the latter.

\section{Cellular homology}
The material in this section covers the basic properties of cellular homology in the context of stacks of CW-complexes. We review it to set up the notation for the following sections.
Let $\Pi : \Totaldown \to \Basedown$ be a stack of CW-complexes with fiber over each cell $e$ of $\Basedown$ denoted by $F_e$. 

Recall that by definition, the $\Z$-module of the degree $a$ cellular chains is  
$$C_a (\Totaldown ) = H_a ( \Totaldown^{(a)} , \Totaldown^{(a-1)}).$$
Now each homeomorphism \eqref{k} induces a map of pairs
$$(F_e^{(k)} \times \mathbb{B}^n  , F_e^{(k)} \times \mathbb{S}^{n-1} \cup F_e^{(k-1)} \times \mathbb{B}^n) \to (\Totaldown^{(n+k)} , \Totaldown^{(n+k-1)} )$$ 
which induces an injective map
\begin{equation} \label{inj1}
H_{n+k} (F_e^{(k)} \times \mathbb{B}^n  , F_e^{(k)} \times \mathbb{S}^{n-1} \cup F_e^{(k-1)} \times \mathbb{B}^n) \to C_{n+k} (\Totaldown).
\end{equation}
Considering the long exact sequence for the pair 
$$(F_e^{(k)} \times \mathbb{B}^n  , F_e^{(k)} \times \mathbb{S}^{n-1} \cup F_e^{(k-1)} \times \mathbb{B}^n),$$
K\"unneth theorem yields a natural isomorphism of $\Z$-modules: 
\begin{equation} \label{isom2}
H_k (F_e^{(k)} , F_e^{(k-1)}) \otimes H_n (\mathbb{B}^n , \mathbb{S}^{n-1} ) \stackrel{\simeq}{\longrightarrow} H_{n+k} (F_e^{(k)} \times \mathbb{B}^n  , F_e^{(k)} \times \mathbb{S}^{n-1} \cup F_e^{(k-1)} \times \mathbb{B}^n).
\end{equation}
By definition, we have 
$$H_k (F_e^{(k)} , F_e^{(k-1)}) = C_k (F_e).$$
The $\Z$-module $H_n (\mathbb{B}^n , \mathbb{S}^{n-1} )$ is free of rank one generated by the relative fundamental class $[\mathbb{B}^n , \mathbb{S}^{n-1}]$. From now on we will identify   
\begin{equation} \label{subspace}
C_k (F_e) \otimes H_n (\mathbb{B}^n , \mathbb{S}^{n-1} )
\end{equation}
with the subspace of $C_{n+k} (\Totaldown )$ spanned by the image of the composition of the maps \eqref{isom2} and \eqref{inj1}. 

The ascending filtration of $\Basedown$ by its $n$-skeleta $\Basedown^{(n)}$ induced via $\Pi : \Totaldown \to \Basedown$ then yields an ascending filtration 
\begin{equation}F^n C_\bullet (\Totaldown) = \bigoplus_{e \in \Basedown^{(n)}} C_{\bullet-\dim(e)} (F_e )  \otimes H_{\dim e} (\mathbb{B}^{\dim e} , \mathbb{S}^{\dim e -1} )
\label{eq: Fn C(A) complex}\end{equation}
on the cellular chain complex $C_\bullet (\Totaldown)$. Recall that the boundary operator $\partial : C_a (\Totaldown) \to C_{a -1} (\Totaldown)$ 
is defined using the homology long exact sequences of the pairs $(\Totaldown^{(a)} , \Totaldown^{(a-1)} )$ and $(\Totaldown^{(a-1)} , \Totaldown^{(a-2)} )$: 
\begin{equation*}
\xymatrix{
C_a (\Totaldown) \ar[d]^{\cong} \ar[rr]^{\partial} & & C_{a-1} (\Totaldown) \ar[d]^{\cong} \\
H_a (\Totaldown^{(a)} , \Totaldown^{(a-1)} ) \ar[r]^{\delta} & H_{a-1} (\Totaldown^{(a-1)} ) \ar[r] & H_{a-1} (\Totaldown^{(a-1)} , \Totaldown^{(a-2)} ).
}
\end{equation*}
In restriction to the image of \eqref{inj1} the composition of maps in the diagram above is the composition of the map 
$$H_{n+k} (F_e^{(k)} \times \mathbb{B}^n  , F_e^{(k)} \times \mathbb{S}^{n-1} \cup F_e^{(k-1)} \times \mathbb{B}^n) \stackrel{\delta = \delta_1 \oplus \delta_2}{\longrightarrow} H_{n+k-1} (F_e^{(k)} \times \mathbb{S}^{n-1} ) \oplus H_{n+k-1} (F_e^{(k-1)} \times \mathbb{B}^n) )$$
by the direct sum of the two natural maps
$$H_{n+k-1} (F_e^{(k)} \times \mathbb{S}^{n-1} ) \longrightarrow H_{n+k-1} (F_e^{(k)} \times \mathbb{S}^{n-1} , F_e^{(k-1)} \times \mathbb{S}^{n-1} )$$
$$H_{n+k-1} (F_e^{(k-1)} \times \mathbb{B}^n ) \longrightarrow H_{n+k-1} (F_e^{(k-1)} \times \mathbb{B}^n , F_e^{(k-1)} \times \mathbb{S}^{n-1}).$$

\begin{definition}[In the fiber, $\partial^{\rm vert}$]
Let
$$\partial^{\rm vert} : C_k (F_e) \otimes H_n (\mathbb{B}^n , \mathbb{S}^{n-1} ) \to C_{k-1} (F_e) \otimes H_n (\mathbb{B}^n , \mathbb{S}^{n-1} )$$
be the map that makes the following diagram commute: 
\begin{footnotesize}
\begin{equation*}
\xymatrix{
C_k (F_e) \otimes H_n (\mathbb{B}^n , \mathbb{S}^{n-1} ) \ar[d]^{\cong} \ar[rr]^{\partial^{\rm vert}} & & C_{k-1} (F_e) \otimes H_n (\mathbb{B}^n , \mathbb{S}^{n-1} )  \\
H_{n+k} (F_e^{(k)} \times \mathbb{B}^n  , F_e^{(k)} \times \mathbb{S}^{n-1} \cup F_e^{(k-1)} \times \mathbb{B}^n) \ar[r]^{\quad \quad \quad \quad \delta_2} & H_{n+k-1} (F_e^{(k-1)} \times \mathbb{B}^n)  \ar[r] &  H_{n+k-1} (F_e^{(k-1)} \times \mathbb{B}^n , F_e^{(k-1)} \times \mathbb{S}^{n-1}) \ar[u]^{\cong}.
}
\end{equation*}
\end{footnotesize}
\end{definition}

Given $e \in E_n$, we denote by 
$$f_e : F_e \times \mathbb{S}^{n-1} \to \Totaldown$$
the map obtained by restricting $f_n$. 
By construction it induces a map of pairs 
$$(F_e^{(k)}  \times \mathbb{S}^{n-1} , F_e^{(k-1)}  \times \mathbb{S}^{n-1} ) \longrightarrow (\Totaldown^{(n+k-1)} , \Totaldown^{(n+k-2)} ).$$

\begin{definition}[In the base, $\partial^{\rm hor}$] \label{def:dhor}
Let
$$\partial^{\rm hor} : C_k (F_e) \otimes H_n (\mathbb{B}^n , \mathbb{S}^{n-1} ) \to  C_{n+k-1} (\Totaldown )$$
be the map that makes the following diagram commute: 
\begin{footnotesize}
\begin{equation*}
\xymatrix{
C_k (F_e) \otimes H_n (\mathbb{B}^n , \mathbb{S}^{n-1} ) \ar[d]^{\cong} \ar[rr]^{\partial^{\rm hor}} & & C_{n+k-1} (\Totaldown )   \\
H_{n+k} (F_e^{(k)} \times \mathbb{B}^n  , F_e^{(k)} \times \mathbb{S}^{n-1} \cup F_e^{(k-1)} \times \mathbb{B}^n) \ar[r]^{\quad \quad \quad \quad \delta_1} & H_{n+k-1} (F_e^{(k)} \times \mathbb{S}^{n-1}) ) \ar[r] &  H_{n+k-1} (F_e^{(k)} \times \mathbb{S}^{n-1} , F_e^{(k-1)} \times \mathbb{S}^{n-1}) \ar[u]^{(f_e)_*}.
}
\end{equation*}
\end{footnotesize}
\end{definition}

By naturality we get the following:

\begin{lemma} \label{L23}
The boundary map $\partial : C_\bullet (\Totaldown) \to C_{\bullet -1} (\Totaldown)$
decomposes as 
\begin{equation}
\partial = \partial^{\rm vert} + \partial^{\rm hor} \label{eq: def partial vert}
\end{equation}
where $\partial^{\rm vert}$ preserves each summand $C_\bullet (F_e) \otimes H_{\dim e} (\mathbb{B}^{\dim e} , \mathbb{S}^{\dim e -1} )$ and acts on it by the boundary operator of the cellular chain complex $C_\bullet (F_e)$, and where 
$$\partial^{\rm hor} : F^n C_\bullet (\Totaldown) \to F^{n-1} C_{\bullet-1} (\Totaldown)$$
maps $c \otimes [\mathbb{B}^n , \mathbb{S}^{n-1}] \in C_\bullet (F_e) \otimes H_{n} (\mathbb{B}^{n} , \mathbb{S}^{n -1} )$ (with $e \in E_n$) to 
$$(-1)^{\dim c} (f_e)_* (c \otimes [\mathbb{S}^{n-1}]).$$
\end{lemma}
\begin{proof} The sign here comes from the isomorphism in the first vertical arrow of the commutative diagram of Definition~\ref{def:dhor}.
\footnote{Analogously, the cellular chain complex of a product $X \times Y$ of two CW-complexes is the tensor product $C_\bullet (X) \otimes C_\bullet (Y)$ equipped with the boundary operator $\partial \otimes \mathrm{id} + (-1)^j \mathrm{id} \otimes \partial$ on $C_i (X) \otimes C_j (Y)$.}
\end{proof}

\section{Effective rebuilding}
Let $\Basedown$ be a finite CW-complex and let $\Totaldown\to \Basedown$ be a stack of CW-complexes with fibers $F_e$.
Now, if for each cell $e$ of $\Basedown$ we are given a CW-complex $F_e '$ of the same homotopy type as $F_e$, Proposition \ref{prop: rebuilding} implies that there exists a stack of CW-complexes $\Pi' : \Totaldown' \to \Basedown$ with fiber $F_e '$ over $e$, and a homotopy equivalence $\Gcell$ making the following diagram commute up to homotopy over each cell:
\begin{equation}\xymatrix{\Totaldown \ar[rd]^{\Pi} \ar[rr]^\Gcell & &  \Totaldown' \ar[ld]^{\Pi '} \\
& \Basedown & }\end{equation}
For each cell $e$ of $\Basedown$ let $\Gcellfibre_e : F_e  \to  F_e '$ be a cellular map that induces a homotopy equivalence. In particular there exists a cellular map $\Hcellfibre_e : F_e ' \to F_e$ and a homotopy 
\begin{equation}
\Homcellfibre_e : [0,1] \times F_e  \to F_e
\end{equation} 
between $\Homcellfibre_e (0,  \cdot ) = \Hcellfibre_e \circ \Gcellfibre_e$ and $\Homcellfibre_e (1 , \cdot )=\mathrm{id}_{F_e}$.  

The maps $\Gcellfibre_e$ and $\Hcellfibre_e$ induce chain maps
$$\Ghomfibre_e : C_\bullet (F_e ) \to C_\bullet (F_e ' ) \quad \mbox{and} \quad \Hhomfibre_e : C_\bullet (F_e ' ) \to C_\bullet (F_e ).$$
The homotopy $\Homcellfibre_e$ induces a map $\Homhomfibre_e \colon C_\bullet (F_e ) \to C_{\bullet+1} (F_e ),$ given by 
$$\Homhomfibre_e (c)=(\Homcellfibre_e)_{*}([I, \p I ]\otimes c), \textrm{ for any cell } c\subseteq F_e .$$
The map $\Homhomfibre_e$ is then a chain homotopy between $1$ and $\Hhomfibre_e\circ \Ghomfibre_e$, i.e.  
$$\Hhomfibre_e\circ \Ghomfibre_e-1= \p \Homhomfibre_e + \Homhomfibre_e \p.$$ 
We denote by 
$$\Ghomfibre : C_\bullet (\Totaldown ) \to C_\bullet (\Totaldown '), \quad  \Hhomfibre : C_\bullet (\Totaldown ' ) \to C_\bullet (\Totaldown )$$
and 
$$\Homhomfibre : C_\bullet (\Totaldown ) \to C_{\bullet+1} (\Totaldown )$$
the (``vertical'') maps induced by $\Ghomfibre_e$, $\Hhomfibre_e$ and $\Homhomfibre_e$ on each subspace \eqref{subspace}.

The three goals of this section are the following.
\begin{enumerate}
\item To give explicit formulas for the chain maps $\Ghom : C_\bullet (\Totaldown) \to C_\bullet (\Totaldown')$ and $\Hhom : C_\bullet (\Totaldown') \to C_\bullet (\Totaldown)$ respectively associated to $\Gcell$ and a homotopy inverse $\Hcell$, in terms of $\Ghomfibre$, $\Hhomfibre$ and $\Homhomfibre$.
\item To describe an explicit chain homotopy $\Homhom :  C_\bullet (\Totaldown) \to C_{\bullet+1} (\Totaldown)$ between $1$ and $\Hhom\circ \Ghom$, in terms of $\Ghomfibre$, $\Hhomfibre$ and $\Homhomfibre$.
\item To give an explicit formula for the boundary operator $\partial'$ on the cellular chain complex $C_\bullet (\Totaldown ')$ in terms of $\Ghomfibre$, $\Hhomfibre$, $\Homhomfibre$, the boundary operator $\partial$ on the cellular chain complex $C_\bullet (\Totaldown )$, and the vertical boundary operator $(\partial ')^{\rm vert}$. 
\end{enumerate}
The precise result is Proposition \ref{P2} below; it is a (homological) effective version of Proposition \ref{prop: rebuilding}. Note that we do not give effective formulas for the homotopy between $\Gcell \circ \Hcell$ and the identity as we will not use it. For the same reason we do not even name the homotopy between $\Ghom \circ \Hhom$ and the identity. 

\medskip

The proof consists in explicating the construction of the homotopy equivalence between $\Totaldown$ and $\Totaldown'$: both the stack of CW-complexes $\Pi' : \Totaldown' \to \Basedown$ and the map $\Gcell :\Totaldown \to \Totaldown'$ are constructed by induction on the dimension of the cells. Suppose that 
$$\Gcell_{n-1} : \Pi^{-1} (\Basedown^{(n-1)}) \to (\Pi')^{-1} (\Basedown^{(n-1)})$$
has been constructed. Then over $n$-cells the stack $\Totaldown'$ is built from the composition of the maps
\begin{equation}f_n ' : \bigsqcup_{e \in E_n } F_e ' \times \mathbb{S}^{n-1} \stackrel{\Hcellfibre_n}{\longrightarrow} \bigsqcup_{e \in E_n } F_e  \times \mathbb{S}^{n-1}  \stackrel{f_n}{\longrightarrow} \Pi^{-1} (\Basedown^{(n-1)}) \stackrel{\Gcell^{(n-1)}}{\longrightarrow}  (\Pi')^{-1} (\Basedown_{n-1}) \end{equation}
where 
$$\Hcellfibre_n = \bigsqcup_{e \in E_n} \Hcellfibre_e \times \mathrm{id}_{\mathbb{S}^{n-1}}.$$

For each $n$ we denote 
$$X_n = \bigsqcup_{e \in E_n } F_e  \times \mathbb{B}^{n}, \quad A_n = \bigsqcup_{e \in E_n } F_e  \times \mathbb{S}^{n-1} \quad \mbox{and} \quad Y_n = \Pi^{-1} (\Basedown^{(n)})$$
and 
$$X_n ' = \bigsqcup_{e \in E_n } F_e '  \times \mathbb{B}^{n}, \quad A_n ' = \bigsqcup_{e \in E_n } F_e '  \times \mathbb{S}^{n-1} \quad \mbox{and} \quad Y_n' = (\Pi')^{-1} (\Basedown^{(n)}).$$
By construction we have:
$$Y_n = Y_{n-1} \cup_{f_n} X_n \quad \mbox{and} \quad Y_n ' = Y_{n-1} ' \cup_{f_n '} X_n '.$$

Now the map $\Gcell$, its homotopy inverse $\Hcell$ and an explicit homotopy
$$\Homcell : [0,1] \times \Totaldown \to \Totaldown, \quad \Homcell (0 , \cdot ) = \mathrm{Id}_{\Totaldown}, \ \Homcell (1 , \cdot ) = \Hcell \circ \Gcell,$$ 
are obtained as the direct limits of maps 
$$\Gcell_{n} : Y_n \to Y_n', \quad \Hcell_n : Y_n \to Y_n ' \quad \mbox{and} \quad \Homcell_n : [0,1] \times Y_n \to Y_n$$
inductively constructed by considering the diagram
\begin{equation} \label{diagn}
\xymatrix{
X_n \ar@<-.5ex>[d]^{\ \Hcellfibre_n} &A_n \ar[l]^{\iota} \ar@<-.5ex>[d]^{\ \Hcellfibre_n} \ar[r]^{f_n} & Y_{n-1} \ar@<-.5ex>[d]^{\ \Hcell_{n-1}} \\
X'_n\ar@<-.5ex>[u]^{\Gcellfibre_n \ } &A'_n \ar[l]^{\iota} \ar@<-.5ex>[u]^{\Gcellfibre_n \ } \ar[r]^{f_n '} & Y_{n-1} \ar@<-.5ex>[u]^{\Gcell_{n-1} \ }.
}
\end{equation}
Here $\iota$ denotes the inclusion map, we have $f_n ' = \Gcell_{n-1}  \circ f_n \circ \Hcellfibre_n$ and the maps $\Hcellfibre_n \circ \Gcellfibre_n$ and $\Hcell_{n-1} \circ \Gcell_{n-1}$ are homotopic to the identity. From these data the proposition proved in the next paragraph gives explicit formulas for the chain maps induced by $\Gcell_n$ and $\Hcell_n$. This provides a (homological) effective version of \cite[Theorem 4.1.8]{Geoghegan}. 

\subsection{Explicit gluing}

Let $X,X',A,A',Y,Y'$ be CW-complexes fitting into the following diagram of cellular maps
\begin{equation*}
\xymatrix{
X \ar@<-.5ex>[d]^{\ \Hcellfibre} &A \ar[l]^{\iota} \ar@<-.5ex>[d]^{\ \Hcellfibre} \ar[r]^{f} & Y \ar@<-.5ex>[d]^{\ \Hcell} \\
X'\ar@<-.5ex>[u]^{\Gcellfibre \ } &A' \ar[l]^{\iota} \ar@<-.5ex>[u]^{\Gcellfibre \ } \ar[r]^{f '} & Y' \ar@<-.5ex>[u]^{\Gcell \ },
}
\end{equation*}
where the $\iota$'s denote inclusions, $f'=\Gcell \circ f \circ \Hcellfibre$ and the maps $\Hcellfibre \circ \Gcellfibre$ and $\Hcell \circ \Gcell$ are homotopic to the identity. Assume furthermore that any cell $e$ of $X$ that intersects $A$ is contained in $A$, similarly for $X'$ and $A'$, and  that $\Gcellfibre^{-1}(A')=A$ and $\Hcellfibre^{-1}(A)=A'$. 

Let $I=[0,1]$. Fix explicit homotopies 
$$\Homcellfibre \colon I\times X\to X \quad \mbox{ and } \quad \Homcell \colon I \times Y\to Y,$$
$$\Homcellfibre  (0,x)=x,\quad \Homcellfibre  (1,x)=\Hcellfibre \circ \Gcellfibre (x),\quad \Homcell (0, y)=y,\quad \Homcell (1,y)=\Hcell \circ \Gcell (y),$$ 
such that $\Homcellfibre (\cdot , A) \subseteq A$.  

The cellular chain complex $C_\bullet (X\sqcup_f Y)$ can be naturally identified with 
$$C_\bullet(X)\oplus C_\bullet(Y)/ (\iota ,-f)C_\bullet(A).$$ 
In the following we use the same letter for the maps and the induced maps on a cellular chain complex except that the chain maps are not bold faced. 

For any cell $e\subseteq X^{(n)}$ let $[e]:=e_*([\mathbb B^n, \mathbb S^{n-1}]).$ The homotopy $\Homcellfibre$ induces a map $\Homhomfibre \colon C_\bullet(X)\to C_{\bullet+1}(X),$ given by 
$$\Homhomfibre ([e])=(\Homcellfibre )_{*}([I, \p I ]\otimes [e]), \textrm{ for any cell } e\subseteq X.$$
The map $\Homhomfibre$ is then a chain homotopy between $1$ and $\Hhomfibre \circ \Ghomfibre$, i.e.  
$$\Hhomfibre \circ \Ghomfibre -1 = \p \Homhomfibre + \Homhomfibre \p.$$ 
The same discussion applies to $Y$; we denote by $\Homhom$ the corresponding chain homotopy.

Decompose the cellular chain complex $C_\bullet (X)$ as
$$C_\bullet(X)=\sum_{e\not\subseteq A}\Z[e]\oplus \iota_*C_\bullet(A)$$ 
and let $1_A\colon C_\bullet (X)\to \iota_*C_\bullet(A)$ be the projection onto the second component. Define analogously $1_{A'}\colon C_\bullet (X')\to \iota_*'C_\bullet (A')$. It follows from our assumptions $\Ghomfibre^{-1}(A')\subseteq A$ and $\Hhomfibre^{-1}(A)\subseteq A'$ 
that 
$$\Ghomfibre \circ 1_A=1_{A'} \circ \Ghomfibre  \quad \mbox{and} \quad \Hhomfibre \circ 1_{A'}=1_A \circ \Hhomfibre.$$ 
In order to lighten the 
notations we will suppress $\iota_*$ and identify $C_\bullet(A)$ with a sub-complex of $C_\bullet(X)$. 

\begin{proposition}[Explicit gluing Lemma] 
 \label{lem-QGluing} 
 1. There exist cellular maps 
$$\tilde \Gcell \colon X \cup_f Y\to X' \cup_{f'}Y' \quad \mbox{and} \quad \tilde \Hcell \colon X'\cup_{f'} Y'\to X\cup_f Y$$ 
such that $\tilde \Gcell |_Y=\Gcell,$ $\tilde \Hcell |_Y=\Hcell$ and $\tilde \Hcell \circ \tilde \Gcell$ is homotopy equivalent to the identity. 

2. The chain maps $\tilde \Ghom : C_\bullet (X \cup_f Y) \to C_\bullet (X' \cup_{f'}Y')$ and $\tilde \Hhom : C_\bullet (X' \cup_{f'}Y') \to C_\bullet (X \cup_f Y)$ respectively associated to $\tilde \Gcell$ and $\tilde \Hcell$ are induced by maps defined on $C_\bullet (X)\oplus C_\bullet(Y)$ and $C_\bullet (X')\oplus C_\bullet(Y')$ by the following formulas 
$$\tilde \Ghom = (\Ghomfibre \circ (1-1_A)+ \Ghom \circ f \circ 1_A - \Ghom \circ f \circ \Homhomfibre \circ (1_A\p -\p 1_A)) \oplus \Ghom $$
and
$$\tilde \Hhom =(\Hhomfibre  \circ (1-1_{A'})+ \Hhom \circ f' \circ 1_{A'} + \Homhom \circ f \circ \Hhomfibre \circ (1_{A'}\p-\p 1_{A'})) \oplus \Hhom,$$ 
where $\Ghom$ and $\Hhom$ are the chain maps respectively induced by $\Gcell$ and $\Hcell$.

3. There exists a homotopy map 
$$\widetilde \Homcell \colon I\times ( X\cup_f Y) \to X\cup_{f} Y$$ 
such that $\widetilde \Homcell (0,z)=z$ and $\widetilde \Homcell (1,z)=\tilde \Hcell \circ \tilde \Gcell (z).$ 

4. The chain homotopy map $\tilde \Homhom (c) =\widetilde \Homcell_*([I , \p I ]\otimes c)$ associated to $\widetilde \Homcell$ is induced by a map defined on $C_\bullet (X)\oplus C_\bullet(Y)$ by the following formula
$$\tilde \Homhom = ( \Homhomfibre \circ (1-1_A)+ \Homhom \circ f \circ 1_A- \Homhom \circ f \circ \Homhomfibre \circ (1_A\p-\p 1_A) ) \oplus \Homhom.$$
\end{proposition}

The proof of Proposition~\ref{lem-QGluing} consists of explicating the construction of the maps $\tilde\Hcell $ and $ \tilde\Gcell$. 
We postpone it until Section \ref{sec-QGluing}.
 The construction is quite technical so the reader may want to skip Section \ref{sec-QGluing} on a first reading.  
It is instructive to check that our formulas work at the level of homology groups.

\subsection{\texorpdfstring{Rebuilding of $\Totaldown$}{Rebuilding of Sigma}}

Starting from 
$$Y_0 = \sqcup_{e \in E_0} F_e, \quad Y_0' = \sqcup_{e \in E_0} F_e', \quad \Gcell_0 = \Gcellfibre_0, \quad \Hcell_0 = \Hcellfibre_0, \quad \mbox{and} \quad \Homcell_0 = \Homcellfibre_0,$$
we inductively apply Proposition \ref{lem-QGluing} to \eqref{diagn} and construct the desired extensions $\Gcell_n$, $\Hcell_n$ and $\Homcell_n$ of $\Gcell_{n-1}$, $\Hcell_{n-1}$ and $\Homcell_{n-1}$.

We identify 
$$C_\bullet (X_n ) = \bigoplus_{e \in E_n} C_\bullet (F_e ) \otimes C_\bullet (\mathbb B^n)$$ 
and let 
$$[\mathbb{B}^n , \mathbb{S}^n] \in H_n (\mathbb{B}^n , \mathbb{S}^n ) = C_n (\mathbb{B}^n)$$ 
be the relative fundamental class that generates the $\Z$-module of rank one $C_n (\mathbb{B}^n)$. Note that given $c \in C_\bullet (F_e)$ we have:
$$1_{A_n} (c \otimes [\mathbb{B}^n , \mathbb{S}^n] )=0 \quad \mbox{and} \quad (1_{A_n}\circ \p -\p \circ 1_{A_n})(c \otimes [\mathbb{B}^n , \mathbb{S}^n] ) = 1_{A_n}\p (c \otimes [\mathbb{B}^n , \mathbb{S}^n] ) = (-1)^{\dim c} c\otimes [\mathbb{S}^{n-1}]$$ 
where the last equality follows from Lemma \ref{L23}. There are analogous formulas for $c' \in C_\bullet (F_e')$ replacing $A_n$ by $A'_n$. 

Denoting by $\Ghom_n$ and $\Hhom_n$ the chain maps respectively associated to $\Gcell_n$ and $\Hcell_n$ and by $\Homhomfibre_n$ and $\Homhom_n$ the chain homotopies respectively associated to $\Homcellfibre_n$ and $\Homcell_{n}$,  Proposition \ref{lem-QGluing} gives 
$$\Ghom_n (c \otimes [\mathbb{B}^n , \mathbb{S}^n])= \Ghomfibre_n (c \otimes [\mathbb{B}^n , \mathbb{S}^n]) - (-1)^{\dim c} \Ghom_{n-1} \circ f \circ  \Homhomfibre_n (c\otimes [\mathbb{S}^{n-1}]),$$
$$\Hhom_n (c' \otimes [\mathbb{B}^n , \mathbb{S}^n]) = \Hhomfibre_n (c' \otimes [\mathbb{B}^n , \mathbb{S}^n])+(-1)^{\dim c} \Homhom_{n-1} \circ f \circ \Hhomfibre_{n-1} (c' \otimes [\mathbb{S}^{n-1}])$$
and
$$\Homhom_n (c \otimes [\mathbb{B}^n , \mathbb{S}^n]) = \Homhomfibre_n (c \otimes [\mathbb{B}^n , \mathbb{S}^n]) - (-1)^{\dim c} \Homhom_{n-1} \circ f \circ \Homhomfibre_n (c\otimes [\mathbb{S}^{n-1}]).$$

We can simplify these formulas by identifying 
$$C_\bullet (Y_n)=(C_\bullet(X_n)\oplus C_\bullet(Y_{n-1}))/ (1,-f_*)C_\bullet(A_n),$$ 
with 
$$\bigoplus_{e \in E_n} C_\bullet(F_e) \otimes H_n (\mathbb{B}^n , \mathbb{S}^{n-1}) \subseteq  F^n C_\bullet(\Totaldown ).$$ 
The map $\Gcell : \Totaldown \to \Totaldown'$ induces chain maps $\Ghom_n : F^n C_\bullet (\Totaldown) \to F^n C_\bullet (\Totaldown' )$, and by Lemma \ref{L23} the chain maps $\Ghom_n$ are inductively defined as follows:
\begin{itemize}
\item The restriction of $\Ghom_n$ to $F^{n-1} C_\bullet (\Totaldown)$ is equal to $\Ghom_{n-1}$.
\item If $e$ is a $n$-cell of $\Basedown$, the restriction of $\Ghom_n$ to $C_\bullet (F_e) \otimes H_n (\mathbb{B}^n , \mathbb{S}^n)$ is given by 
\begin{equation} \label{kn}
\Ghom_n (c \otimes [\mathbb{B}^n , \mathbb{S}^n] ) = \Ghomfibre_e (c) \otimes [\mathbb{B}^n , \mathbb{S}^{n-1}] +  \Ghom_{n-1} \circ \p^{\rm hor} (\Homhomfibre_e (c) \otimes [\mathbb{B}^n , \mathbb{S}^{n-1}] )  \quad \left( c \in C_\bullet (F_e ) \right).
\end{equation}
\end{itemize}
The sign changed because $\dim \sigma_e(c)=\dim c +1$, so $(-1)^{\dim c}f\circ\sigma_n(c\otimes [\mathbb S^{n-1}])=-\partial^{\rm hor}(\sigma_e(c)\otimes [\mathbb B^n,\mathbb S^{n-1}]),$ by Lemma \ref{L23}.
By induction, we get:
\begin{equation} \label{E:k}
\Ghom = \Ghomfibre \circ \left( \sum_{i=0}^{n} \left( \partial^{\rm hor} \circ \Homhomfibre \right)^i \right).\end{equation}
Note that the map $\Homhomfibre$ preserves the filtration
$F^{n} C_\bullet (\Totaldown)$ and that the map $\partial^{\rm hor}$ maps $F^n C_\bullet (\Totaldown)$ to $F^{n-1} C_{\bullet-1} (\Totaldown)$.
It follows that $ \left( \partial^{\rm hor} \circ \Homhomfibre \right)^i$ vanishes on $C_n(\Totaldown)$ for $i\geq n+1$.

We similarly get formulas for $\Hhom$ and $\Homhom$. To sum up we get:

\begin{proposition}[Effective rebuilding Lemma]\label{P2}
We have
\begin{enumerate}

\item The rebuilding lemma yields a CW-stack $\Pi'\colon \Totaldown' \to \Basedown$, cellular maps $\Gcell \colon \Totaldown \to \Totaldown '$, $\Hcell \colon \Totaldown ' \to \Totaldown$ and a homotopy $\Homcell \colon I\times \Totaldown \to \Totaldown$ such that 
$$\Homcell (0,\cdot )={\rm Id}_{\Totaldown} \quad \mbox{and} \quad \Homcell (1,\cdot )=\Hcell \circ \Gcell.$$ 

\item The chain maps  $\Ghom : C_\bullet (\Totaldown) \to C_\bullet (\Totaldown')$, $\Hhom : C_\bullet (\Totaldown') \to C_\bullet (\Totaldown)$ and $\Homhom :  C_\bullet (\Totaldown) \to C_\bullet (\Totaldown)$ respectively associated to $\Gcell$, $\Hcell$ and $\Homcell$ are respectively given by 
\begin{equation}\label{E: effective k}
\Ghom = \Ghomfibre \circ \left( \sum_{i=0}^{\infty} \left( \partial^{\rm hor} \circ \Homhomfibre \right)^i \right) \quad \mbox{ on } C_\bullet (\Totaldown ),
\end{equation}
\begin{equation} 
\Hhom = \left( \sum_{i=0}^{\infty} \left(  \Homhomfibre \circ \partial^{\rm hor}   \right)^i \right) \circ \Hhomfibre \quad \mbox{ on } C_\bullet (\Totaldown' ), \end{equation}
and
\begin{equation}
\Homhom = \Homhomfibre \circ \left( \sum_{i=0}^{\infty} \left(\partial^{\rm hor} \circ \Homhomfibre \right)^i \right) \quad \mbox{ on } C_\bullet (\Totaldown ). \end{equation} 

\item The formula for the boundary operator $\partial'$ on the cellular chain complex $C_\bullet (\Totaldown ')$ of the rebuilt CW-complex $\Totaldown'$ is given on $C_\bullet (\Totaldown ')$ by 
\begin{equation}\label{E:P2} 
\partial ' = (\partial ' )^{\rm vert} + \Ghomfibre \circ \left( \sum_{i=0}^{\infty} \left( \partial^{\rm hor} \circ \Homhomfibre \right)^i \right) \circ  \partial^{\rm hor} \circ \Hhomfibre. \end{equation}
\end{enumerate}
\end{proposition}
\begin{proof} It only remains to prove (3). 
Lemma \ref{L23} applies to the stack $\Pi ' : \Totaldown' \to \Basedown$ so that if $e$ is an $n$-cell of $\Basedown$ and $c \in C_\bullet (F_e)$ we have:
\begin{equation} ((\partial ')^{\rm hor}) (c \otimes e) = (-1)^{\dim c} f_e ' (c \otimes [\mathbb{S}^{n-1}]).\end{equation}
On the other hand it follows from the construction of $\Pi' : \Totaldown' \to \Basedown$ that 
\begin{equation}
f_e ' (c \otimes [\mathbb{S}^{n-1}]) = \Ghom_{n-1} (f_e (\Hhomfibre_e (c) \otimes [\mathbb{S}^{n-1}])).\end{equation} 
We conclude that 
\begin{equation*}
\begin{split}
((\partial ')^{\rm hor}) (c \otimes e) & = (-1)^{\dim c} f_e ' (c \otimes [\mathbb{S}^{n-1}]) \\
& = (-1)^{\dim c} \Ghom_{n-1} (f_e (\Hhomfibre _e (c) \otimes [\mathbb{S}^{n-1}])) \\
& = \Ghom_{n-1} \circ \partial^{\rm hor} \circ \Hhomfibre (c \otimes e),
\end{split}
\end{equation*}
and \eqref{E:P2} follows from \eqref{E: effective k}.
\end{proof}

\section{Proof of Proposition \ref{lem-QGluing}}\label{sec-QGluing}

Let $X$ be a CW-complex and let $A\subseteq X$ be a sub-complex. Then $( \{0 \} \times X ) \cup (I \times A )$ is a strong deformation retract of $I \times X$; see e.g. \cite[1.3.15]{Geoghegan}. We refine this property in the following lemma.

\begin{lemma}\label{lem-pconst}
There exists a cellular map 
$$p\colon I\times X\to \{0\}\times X\cup I\times A$$ 
such that for all $x \in X$, $a \in A$ and $s \in ]0,1]$,
\begin{enumerate}
\item $p(0,x)=(0,x) \in \{0\}\times X$;
\item $p (s,x) = (p^1 (s,x) , p^2 (s,x)) \in [0 ,s] \times X$;
 \item $p(s,a)=(s,a) \in I\times A$
\end{enumerate}
and the following formulas hold for any chain $c\in C_\bullet(X)$: 
\begin{align*}
p_*([I, \p I]\otimes c)=&[I , \p I] \otimes 1_Ac, \\
\delta_*([I, \p I]\otimes c)=&[I , \p I] \otimes (1-1_A)c,
\end{align*}
where 
$$\delta\colon I\times  X\to I\times X, \quad \delta(s,x)=(s-p^1(s,x),p^2(s,x)).$$
\end{lemma}
\begin{proof}
We define the map $p$ cell by cell, starting from $0$-dimensional cells and attaching any available cell of the lowest dimension as pictured on Figure \ref{fig:p}.

\begin{figure}[ht]      
\begin{center}
\includegraphics[width=.3\textwidth]{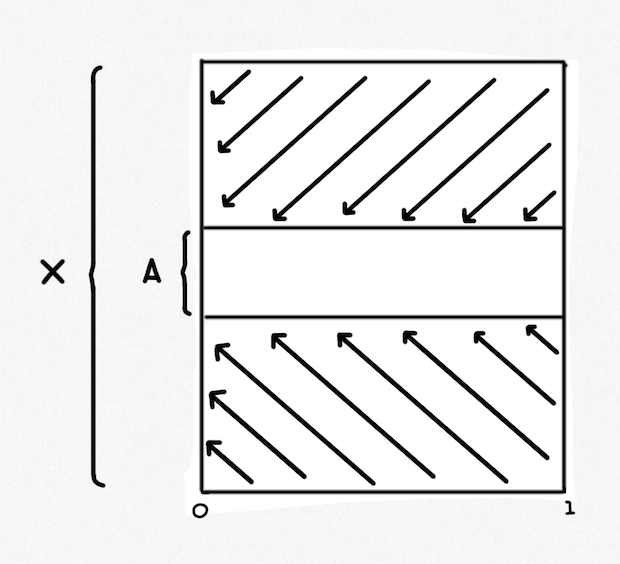}
\end{center}
\caption{The map $p$}\label{fig:p}
\end{figure}

The proof reduces to the following statement. Let $Z$ be a CW-sub-complex of $X$ of dimension at most $n$, and let $e\colon \mathbb B^n\to X$ be an $n$ cell with $e(\mathbb S^{n-1})\subseteq Z^{(n-1)}$.  Suppose we are given a cellular map 
$$p \colon I\times Z\to \{0\}\times Z\cup I\times (A\cap Z)$$ 
satisfying all the desired properties. Put $Z'=\mathbb B^n\sqcup_{e} Z$. We shall construct an extension 
$$p \colon I \times Z'\to \{0\}\times Z'\cup I \times (A\cap Z')$$ 
of $p$, satisfying all the required properties. To do so it is enough to define $p$ on $I\times \mathbb B^n$ so that 
$$\forall s \in I , \ \forall x\in \mathbb S^{n-1},  \quad p(s,x)=p(s,e(x)).$$
Now if $e\subseteq A$ we put 
\begin{equation} \label{E:ponA}
p(s,e(x))=(s,e(x)) \in I \times A.
\end{equation}
Otherwise $p(s,e(x))$ is already defined for all $x\in \mathbb S^{n-1}$ so we put $p(s,x)=p(s,e(x))$ for $x\in \mathbb S^{n-1}$, $p(0,x)=(0,x)$ for $x\in\mathbb B^n$ and extend the map to $I\times \mathbb B^n$ using the homotopy extension property for the pair $\{0\}\times \mathbb B^n\cup I\times \mathbb S^{n-1}$ (cf. \cite[1.3.15]{Geoghegan}). By replacing $p^1(s,x)$ with $\max\{s, p^1(s,x)\}$ we can ensure that the second condition is satisfied. 

\begin{figure}[ht]      
\begin{center}
\includegraphics[width=.25\textwidth]{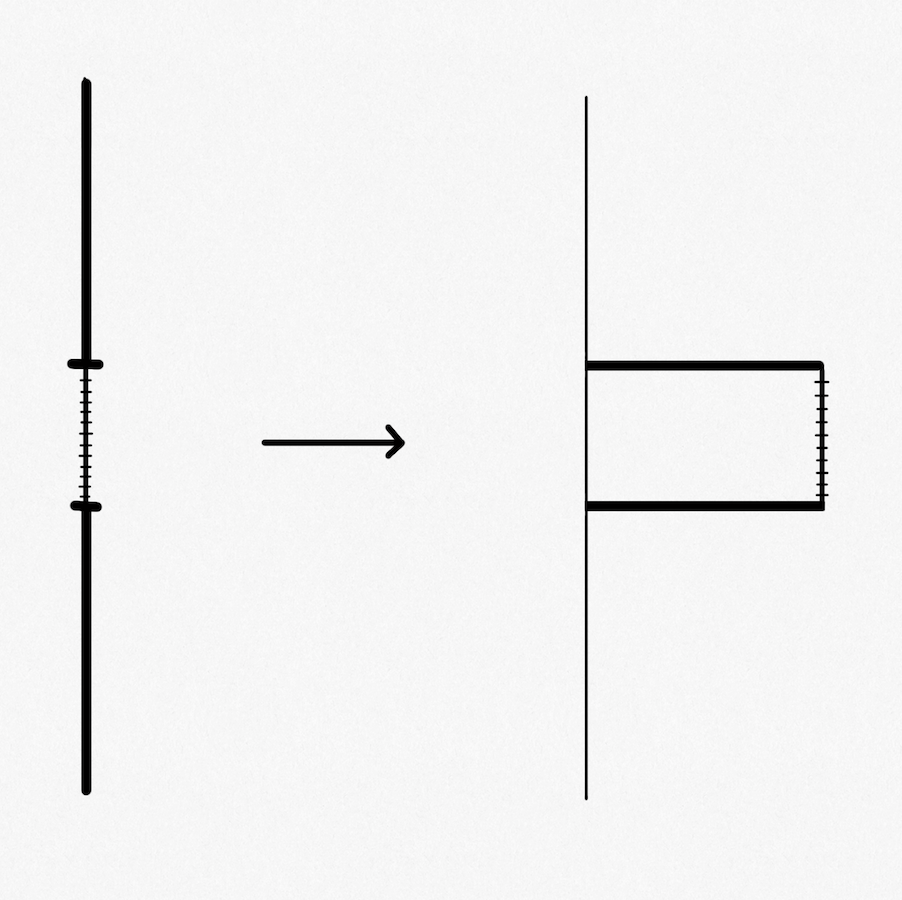}
\end{center}
\caption{The map $p(1 , -)$}\label{fig:p1}
\end{figure}

By construction, the maps $p$ and $\delta$ are cellular and both $p(I\times e)$ and $\delta(I\times e)$ are contained in $I\times e$.  Write $[e]=e_*([\mathbb{B}^n , \mathbb{S}^n ])$ and $[\p e] = e_* ([\mathbb{S}^{n-1}])$.  We check the formulas for $p_*$ and $\delta_*$ on $[I,\p I]\otimes [e]$.
When $e\subseteq A$, the formula for $p_*$ follows from the definition (see \eqref{E:ponA}). If $e\not \subseteq A$ then the $(n+1)$-cell $I\times e$ is mapped by $p$ into 
$$\{0 \} \times (Z')^{(n)} \cup  I \times Z^{(n-1)}$$ 
which is contained in the $n$-skeleton of $I \times Z'$. It follows that $p_*([I , \p I ]\otimes[e])=0$. This proves the formula for $p_*$. 

Using that $p_* \circ \p = \p \circ p_*$, the formula for $p_*$ now implies that 
\begin{equation} \label{E:p*formula}
p_*([1]\otimes c)=[0]\otimes (1-1_A)c+[I , \p I ]\otimes (1_A\p-\p 1_A)c+[1]\otimes 1_Ac
\end{equation} 
and consequently that 
\begin{equation} \label{E:delta*formula}
\delta_*([1]\otimes c)=[1]\otimes (1-1_A)c-[I , \p I ]\otimes (1_A\p-\p1_A)c+[0]\otimes 1_Ac.
\end{equation}

\begin{figure}[ht]      
\begin{center}
\includegraphics[width=.4\textwidth]{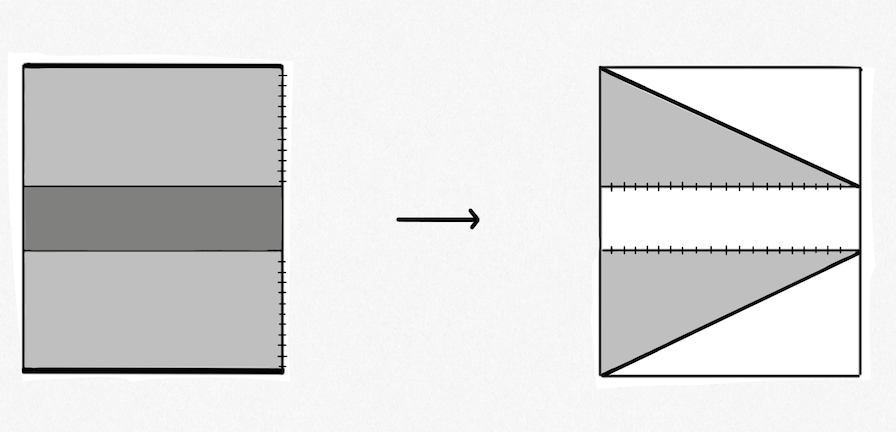}
\end{center}
\caption{The map $\delta$}\label{fig:delta}
\end{figure}

We finally derive the general formula for  $\delta_*([I,\p I]\otimes [e])$ from \eqref{E:delta*formula}: we know that 
$$\delta_*([I,\p I]\otimes [e])\in C_n(I\times Z')=C_0(I)\otimes C_{n+1}(Z')\oplus C_1(I)\otimes C_n(Z')=C_1(I)\otimes C_n(Z'),$$ 
since $Z'$ is $n$-dimensional. As $\delta(I\times e)\subseteq I\times e$ we must have 
$$\delta_*([I,\p I]\otimes [e])=\alpha [I,\p I]\otimes [e]$$ 
for some $\alpha\in \mathbb Z$. To find $\alpha$ we compute $\p \delta_*([I,\p I]\otimes [e])$ using \eqref{E:delta*formula} and, by induction, the formula for $\delta_*$ on $[I , \p I ]\otimes [\p e]$~:
\begin{equation*}
\begin{split}
\p \delta_*([I,\p I]\otimes [e]) & = \delta_* ( \p ([I,\p I]\otimes [e]) ) \\
& = \delta_*([1]\otimes [e])-\delta_*([0]\otimes [e]) - \delta_*([I , \p I ]\otimes [\p e] ) \\ 
& =\p ([I , \p I ]\otimes (1-1_A)[e]).
\end{split}
\end{equation*} 
Hence $\alpha=0$ if $e\subseteq A$ and $\alpha=1$ otherwise. This proves the formula for $\delta_*([I,\p I]\otimes [e]).$

 \end{proof}

\medskip

We now come back to the notations of Proposition \ref{lem-QGluing}.

Let $$p\colon I\times (X\cup_{\iota} A )  \to (\{0\}\times X ) \cup (I\times A)$$ be the map afforded by Lemma \ref{lem-pconst} applied to the pair $(X,A)$.
We write 
$$p_s(x) =p(s,x), \quad ( s\in I, \ x\in X).$$
We have 
$$C_\bullet(\{0\}\times X\cup I \times A)=[0]\otimes C_\bullet(X) + [I , \p I] \otimes C_\bullet(A)+[1]\otimes C_\bullet(A)$$ 
and, according to Lemma \ref{lem-pconst} and Equation \eqref{E:p*formula}, the map $p$ can be chosen so that 
\begin{equation*}
\begin{split}
p_*([I, \p I ] \otimes c) & = [I, \p I] \otimes1_Ac,\\ 
(p_1)_*(c) &= [0]\otimes (1-1_A)c+[1]\otimes 1_Ac+[I , \p I] \otimes (1_A\p-\p1_A)c \\
 & = [0]\otimes c+\p [I , \p I ]\otimes 1_Ac+[I, \p I ] \otimes(1_A\p-\p 1_A)c. 
\end{split}
\end{equation*} 
As in Lemma \ref{lem-pconst} we write $p^1_s\colon I\times X\to I$ and $p^2_s\colon I\times X\to X$ for the coordinates of $p_s$. Replacing $p_s^1(x)$ with $\min\{s, p_s^1(x)\}$  we may assume that $p_s^1(x)\leq s$ for $s\in I$ (actually the construction in Lemma \ref{lem-pconst} already gives such $p$).

Similarly, using Lemma \ref{lem-pconst} we choose a map 
$$q\colon I\times (X' \cup_{\iota} A' ) \to (\{0\}\times X' ) \cup (I\times A' ).$$ 
For the readers' convenience we spell out the relevant properties of $q$. For all $x'\in X'$, $a'\in A'$ and $s\in I$
$$q(0,x')=(0,x'), \quad q(s,a')=(s,a')$$  
and  
\begin{equation*}
\begin{split}
q_*([I , \p I] \otimes c) & = [I , \p I ] \otimes1_{A'} c, \\ 
(q_1)_*(c) & = [0]\otimes (1-1_{A'} )c+[1]\otimes 1_{A'}c+[I, \p I] \otimes (1_{A'}\p-\p1_{A'})c \\
& = [0]\otimes c+\p [I , \p I ]\otimes 1_{A'}c+[I , \p I ]\otimes(1_{A'}\p-\p 1_{A'})c. 
\end{split}
\end{equation*} 
We define the map
$$\tilde \Gcell : X \cup_f Y \to X' \cup_{f'} Y '$$ 
to be equal to $\Gcell$ on $Y$ and for $x \in X$ by the formula 
\begin{equation*}
\tilde \Gcell (x)=\begin{cases} \Gcellfibre \circ p^2_1(x) &\textrm{ if } p_1(x)\in \{0\}\times X,\\
\Gcell \circ f\circ \Homcellfibre (1-p_1^1(x),p_1^2(x)) &\textrm{ if } p_1(x)\in I\times A.\end{cases}
\end{equation*}
To check that $\tilde \Gcell$ is a well defined continuous map it is enough to verify that the partial formulas coincide on 
$$\{x\in X \; : \;  p_1(x)\in \{0\}\times A\}$$ 
and that for all $a \in A$ we have $\tilde \Gcell (a)=\Gcell \circ f(a)$. If $p_1(x)\in \{0\}\times A$ we have 
$$\Gcellfibre \circ p_1^2(x)=f'\circ \Gcellfibre \circ p_1^2(x) \quad \mbox{in} \quad X' \cup_{f'} Y '$$
and
$$f'\circ \Gcellfibre \circ p_1^2(x) =\Gcell \circ f\circ \Hcellfibre \circ \Gcellfibre \circ p_1^2(x)=\Gcell \circ f \circ \Homcellfibre (1,p_1^2(x))=\Gcell \circ f \circ \Homcellfibre (1-p_1^1(x),p_1^2(x)).$$
For $a\in A$ we have $p_1(a)=(1,a)$ so 
$$\tilde \Gcell (a)=\Gcell \circ f\circ \Homcellfibre (0,a)=\Gcell \circ f(a),$$ 
as desired.

Analogously we define the map
$$\tilde \Hcell : X' \cup_{f'} Y' \to X \cup_{f} Y$$ 
to be equal to $\Hcell$ on $Y'$ and for $x \in X'$ by the formula 
\begin{equation*}
\tilde \Hcell (x)= \begin{cases} 
\Hcellfibre \circ q^2_1(x) &\textrm{ if } q_1(x)\in\{0\}\times X',\\
\Homcellfibre (q_1^1(x),f\circ \Hcellfibre \circ q_1^2(x)) & \textrm{ if } q_1(x)\in I\times A' .
\end{cases} 
\end{equation*} 
The verification that $\tilde \Hcell$ is well defined and continuous is completely analogous to what we did for $\tilde \Gcell$.   

The formulas for the chain maps $\tilde \Ghom$ and $\tilde \Hhom$ respectively associated to $\tilde \Gcell$ and $\tilde \Hcell$ now follow from those for $(p_1)_*$ and $(q_1)_*$; more specifically the identity (\ref{E:p*formula}). 

It remains to construct an explicit homotopy between the identity map and $\tilde \Hcell \circ \tilde \Gcell$, which is the main content of the proposition. 
On $X$ we have 
\begin{equation*} 
\tilde \Hcell \circ\tilde \Gcell (x)= \left\{
\begin{array}{ll} 
\Hcellfibre \circ q_1^2\circ \Gcellfibre \circ p_1^2(x) & \mbox{if } p_1(x)\in \{0\}\times X \mbox{ and } q_1\circ \Gcellfibre \circ p_1^2 (x)\in \{0\}\times X', \\
\Homcell ( q_1^1\circ \Gcellfibre \circ p_1^2(x), f\circ \Hcellfibre \circ q_1^2\circ \Gcellfibre \circ p_1^2(x)) & \mbox{if } p_1(x)\in \{0\}\times X \mbox{ and }  q_1\circ \Gcellfibre \circ p_1^2(x)\in I\times A',\\
\Hcell \circ \Gcell \circ f\circ \Homcellfibre (1-p_1^1(x), p_1^2(x)) &  \mbox{if } p_1(x)\in I\times A,
\end{array} \right.
\end{equation*} 
while on $Y$ we have 
$$\tilde \Hcell \circ \tilde \Gcell (y)=\Hcell \circ \Gcell (y), \quad \mbox{for all } y\in Y.$$ 
Note that for $p_1(x)\in I\times A,$ we have $\Gcellfibre \circ p_1^2(x)\in A'$ so that $q_1 ( \Gcellfibre \circ p_1^2(x) ) = (1 , \Gcellfibre \circ p_1^2(x) )$ and therefore
\begin{equation*}
\begin{split}
\Hcell \circ \Gcell \circ f\circ \Homcellfibre (1-p_1^1(x), p_1^2(x)) &= \Homcell (1, f\circ \Homcellfibre (1-p_1^1(x), p_1^2(x)) \\ 
& = \Homcell (q_1^1\circ \Gcellfibre \circ p_1^2(x), f\circ \Homcellfibre (1-p_1^1(x), p_1^2(x)).
\end{split}
\end{equation*} 
This already suggests the rough form of the homotopy between $1$ and $\tilde \Hcell \circ \tilde \Gcell$. We want to construct a map 
$$\tilde \Homcell \colon I\times (X\cup_f Y)\to X\cup_f Y \quad \mbox{with} \quad \tilde \Homcell (0,z)=z \quad \mbox{and} \quad \tilde \Homcell (1,z)=\tilde \Hcell \circ \tilde \Gcell (z).$$ 
We define it piece by piece starting with $[1/2,1]\times (X\cup_f Y)$. Define subsets 
\begin{align*}
\mathcal C_1:=&\{(s,x)\in [0,1/2]\times X |\, p_1(x)\in \{0\}\times X \mbox{ and } q_{2s}\circ \Gcellfibre \circ p_1^2(x)\in \{0\}\times X'\},\\
\mathcal C_2:=&\{(s,x)\in [0,1/2]\times X |\, p_1(x)\in \{0\}\times X \mbox{ and } q_{2s}\circ \Gcellfibre \circ p_1^2(x)\in I\times A'\},\\
\mathcal C_3:=&\{(s,x)\in [0,1/2]\times X |\, p_1(x)\in I\times A\}. 
\end{align*} 

For all $s \in [0,1/2]$ we set 
$$\tilde \Homcell (1/2+s,y)= \Homcell (2s,y) \quad \mbox{if } y \in Y$$ 
and define 
\begin{equation*}
\tilde \Homcell (1/2+s,x)= \left\{ 
\begin{array}{ll} 
\Hcellfibre \circ q_{2s}^2\circ \Gcellfibre \circ p_1^2(x) &  \mbox{if } (s,x)\in \mathcal C_1,\\
\Homcell ( q_{2s}^1\circ \Gcellfibre \circ p_1^2(x), f\circ \Hcellfibre \circ q_{2s}^2\circ \Gcellfibre \circ p_1^2(x)) & \mbox{if } (s,x)\in \mathcal C_2,\\
\Homcell (q_{2s}^1\circ \Gcellfibre \circ p_1^2(x), f\circ \Homcellfibre (1-p_1^1(x), p_1^2(x))) & \mbox{if } (s,x)\in \mathcal C_3, 
\end{array} \right.
\end{equation*} 
for all $x \in X$. 
Let's check that the partial maps agree on the common boundaries.
\begin{itemize}
\item  The common boundary to $\mathcal C_1$ and $\mathcal C_2$ is the set 
$$\{(1/2+s,x) \; : \; p_1(x)\in \{0\}\times X \mbox{ and } q_{2s}\circ \Gcellfibre \circ p_1^2(x)\in \{0\}\times A'\}.$$ 
For $(1/2+s,x)$ therein, we have $\Hcellfibre \circ q^2_{2s}\circ \Gcellfibre \circ p_1^2(x)\in A$ so that
$$\Hcellfibre \circ q_{2s}\circ \Gcellfibre \circ p_1^2(x)=f\circ \Hcellfibre \circ q_{2s}\circ \Gcellfibre \circ p_1^2(x)=\Homcell (q_{2s}^1\circ \Gcellfibre \circ p_1^2(x), f\circ \Hcellfibre \circ q_{2s}\circ \Gcellfibre \circ p_1^2(x)).$$
\item The common boundary of $\mathcal C_2$ and $\mathcal C_3$ is 
$$\{(1/2+s,x) \; : \;  p_1(x)\in \{0\}\times A \mbox{ and } q_{2s}\circ \Gcellfibre \circ p_1^2(x)\in I\times A'\}.$$ 
Note that $p_1(x)\in \{0\}\times A$ forces $\Gcellfibre \circ p_1^2(x)\in A'$ so that $q_{2s}\circ \Gcellfibre \circ p_1^2(x)=(2s, \Gcellfibre \circ p_1^2(x)).$  We then have 
\begin{equation*}
\begin{split}
\Homcell (q^1_{2s}\circ \Gcellfibre \circ p^2_1(x), f\circ \Hcellfibre \circ q^2_{2s}\circ \Gcellfibre \circ p^2_1(x)) &= \Homcell (q^1_{2s}\circ \Gcellfibre \circ p^2_1(x), f\circ \Hcellfibre \circ \Gcellfibre \circ p^2_1(x))\\ 
& = \Homcell (q^1_{2s}\circ \Gcellfibre \circ p^2_1(x), f\circ \Homcellfibre (1-p^1_1(x), p^2_1(x)).
\end{split}
\end{equation*}
\item The common boundary of $\mathcal C_1$ and $\mathcal C_3$ is 
$$\{(1/2,x) \; : \; p_1(x)\in \{0\}\times A\}.$$ 
It is contained in $\mathcal C_1\cap \mathcal C_2\cap \mathcal C_3$ so that this case follows from the two previous ones. 
\item Finally $Y$ intersects non-trivially only $\mathcal C_3$ and their common subset is 
$$\{(1/2+s, a) \; : \;  a\in A\}.$$ 
There we have 
$$q^1_{2s}\circ \Gcellfibre \circ p^2_1(a)=2s, \quad p_1^1(a)=1 \quad \mbox{and} \quad p_1^2(a)=a$$ 
so that  
$$\Homcell (q^1_{2s}\circ \Gcellfibre \circ p^2_1(a), f\circ \Homcellfibre (1-p^1_1(a), p^2_1(a))) = \Homcell (2s,f(a)).$$
\end{itemize}
Now that we have a well defined continuous map $\tilde \Homcell (s,\cdot)$ for $s\in [1/2,1]$, we proceed to define $\tilde \Homcell$ for $s\in [0,1/2]$. First note that 
$\tilde \Homcell (1/2,\cdot)$ can be more simply defined by 
$$\tilde \Homcell (1/2,y)=y, \quad \mbox{if } y\in Y, \quad \mbox{and} \quad \tilde \Homcell (1/2,x)= \Homcellfibre (1-p_1^1(x),p_1^2(x)), \quad \mbox{if } x\in X.$$
The last expression is indeed equal to 
$$\Hcellfibre \circ \Gcellfibre \circ p^2_1(x), \quad \mbox{if } p_1(x)\in \{0\}\times X, \quad \mbox{and} \quad f\circ \Homcellfibre (1-p^1_1(x), p^2_1(x)), \quad \mbox{if } p_1(x)\in I\times A.$$
For $s\in [0,1/2]$ we then set 
\begin{equation*}
\tilde \Homcell (s,y)=y,  \quad \mbox{if } y\in Y, \quad \mbox{and} \quad \tilde \Homcell (s,x)= \Homcellfibre (2s-p_{2s}^1(x),p_{2s}^2(x)), \quad \mbox{if } x\in X.
\end{equation*} 
The total map $\tilde \Homcell \colon I\times (X\cup_f Y)\to X\cup_f Y$ is continuous by construction. It remains to compute 
$$\tilde \Homhom (c) = \tilde \Homcell_*([I , \p I ]\otimes c) \quad \mbox{for all } c\in C_\bullet(X\cup_f Y).$$ 
Let us refine the CW-complex structure of $I$ by taking the first barycentric subdivision of the original one. There are now two $1$-dimensional cells $I_1=(0,1/2)$ and $I_2=(1/2,1)$ such that $[I, \p I ]=[I_1 , \p I_1 ]+[I_2 , \p I_2]$. It follows that
$$\tilde \Homcell_*([I , \p I ]\otimes c)=\tilde \Homcell_*([I_1 , \p I_1 ]\otimes c)+\tilde \Homcell_*([I_2 , \p I_2 ]\otimes c).$$
Given $c\in C_\bullet (Y)$ we have 
$$\tilde \Homcell_*([I_1 , \p I_1 ]\otimes c)=0 \quad \mbox{and} \quad  \tilde \Homcell_*([I_2 , \p I_2 ]\otimes c) = \Homhom (c),$$ 
whence $\tilde \Homcell_*([I , \p I ]\otimes c)=\Homhom (c).$ 

Now let $c\in C_\bullet(X).$ We first compute $\tilde \Homcell_*([I_2, \p I_2 ]\otimes c)$. Recall that it follows from \eqref{E:p*formula} that
$$(p_1)_*(c)=[0]\otimes (1-1_A)c +[I , \p I ]\otimes (1_A\p-\p 1_A)c+ [1]\otimes 1_A c.$$ 
From this and the definition of $\tilde \Homcell$ on $[1/2 , 1] \times (X\cup_f Y)$ we get 
\begin{multline}\label{eq-H1} 
\tilde \Homcell_*([I_2 , \p I_2 ]\otimes c)= \alpha_*\beta_*([I , \p I ]\otimes (1-1_A)c)) \\ + \Homcell_*(\gamma_*([I , \p I] \otimes[I , \p I ]\otimes (1_A\p-\p 1_A)c)) \\ +\Homcell_*([I , \p I ]\otimes f \circ 1_A (c) ),
\end{multline} 
where 
$$\alpha\colon (\{0\}\times X' ) \cup (I\times A' ) \to X\cup_f Y, \quad \beta\colon I\times X\to (\{0\}\times X' ) \cup (I\times A') \quad \mbox{and} \quad  \gamma\colon I\times I \times A\to I\times Y$$ 
are respectively given by 
\begin{equation*} 
\alpha(s,x)= \left\{
\begin{array}{ll} 
\Hcellfibre (x), & \mbox{if } (s,x)\in \{0\}\times X',\\ 
\Homcell (s, f\circ \Hcellfibre (x)), &  \mbox{if }  (s,x)\in I\times A',
\end{array} \right.  \qquad \beta(s,x)=q_s\circ \Gcellfibre (x)
\end{equation*}
and
$$\gamma(s,t,a) =( q_s^1\circ \Gcellfibre (a), f\circ \Homcellfibre (1-t,a))=(s,f\circ \Homcellfibre (1-t,a)).$$
Recall the notations 
$$\Homhomfibre (u)=\Homcellfibre_*([I , \p I ]\otimes u) \quad \mbox{and} \quad \Homhom (u)=\Homcell_*([I , \p I ]\otimes u)$$ 
for respectively $u\in C_\bullet(X)$ and $C_\bullet (Y)$. 

For all $u\in C_\bullet(X)$ we have
$$\beta_*([I , \p I ]\otimes u)= q_*([I , \p I ]\otimes \Ghomfibre (u) )=[I , \p I ]\otimes 1_{A'} \circ \Ghomfibre (u)$$
and
$$\gamma_*([I , \p I ]\otimes [I , \p I ]\otimes u)= -[I , \p I ]\otimes f \circ \Homhomfibre (u).$$
Note that it follows from the formula for $\beta_*$ that 
$$\beta_*([I , \p I ]\otimes (1-1_A)(c))=[I , \p I ]\otimes 1_{A'} \circ \Ghomfibre \circ (1-1_A)(c)=[I , \p I ]\otimes \Ghomfibre \circ 1_A \circ (1-1_A) (c)=0.$$
So we do not need to compute $\alpha_*$ in (\ref{eq-H1}) to conclude that 
$$\tilde \Homcell_* ([I_2 , \p I_2 ]\otimes c)=-\Homhom \circ f \circ \Homhomfibre \circ (1_A\p-\p 1_A) (c)+ \Homhom \circ f \circ 1_A (c).$$

We finally compute $\tilde \Homcell_*([I_1 , \p I_1 ]\otimes c).$ Let 
$$\delta\colon I \times X\to I\times X, \quad \delta(s,x)=(s-p^1_s(x), p_s^2(x)).$$ 
By Lemma \ref{lem-pconst}, we can choose $p$ so that $\delta$ is cellular and 
$$\delta_*([I , \p I ]\otimes c)=[I , \p I]\otimes (1-1_A)c.$$ 
It then follows from the definition of $\tilde \Homcell$ on $[0 , 1/2] \times (X\cup_f Y)$ that
$$\tilde \Homcell_* ([I_1 , \p I_1 ]\otimes c)=\Homcellfibre_* \circ \delta_* ([I , \p I ]\otimes c)=\Homhomfibre \circ (1-1_A)(c).$$ 
We conclude that 
$$\tilde \Homhom = (\Homhomfibre \circ (1-1_A) - \Homhom \circ f \circ \Homhomfibre \circ (1_A\p-\p 1_A) + \Homhom \circ f \circ 1_A) \oplus \Homhom .$$
This finishes the proof of Proposition \ref{lem-QGluing}. \qed

\section{Quality of rebuilding for nilpotent groups (Proof of Theorem \ref{thm-UnipRewiring})} \label{S:unip}

The goal of the section is to prove the following:

\begin{theorem}[Theorem~\ref{thm-UnipRewiring}]\label{thm-UnipRewiring-in-text}
Let $\Lambda$ be a finitely generated torsion-free nilpotent group. If $Y_0$ is a compact $K(\Lambda,1)$ space, 
 there exists a constant ${\kappa}\geq 1$ such that for every finite index subgroup $\Lambda_1\leq \Lambda$, the cover $Y_1=\Lambda_1\bs \tilde Y_0$ admits an $\indice$-rebuilding $(Y_1,Y_1',\Gcell ,\Hcell ,  \Homcell )$
of quality $([\Lambda : \Lambda_1], {\kappa})$ for every $\indice$.
\end{theorem}

\subsection{Generalities on rebuildings and quality}

Let $\indice \in \mathbb N$ and let $X$ be a CW-complex with finite $\indice$-skeleton. Recall (Definition~\ref{def: Rebuilding} of the introduction) that an $\indice$-\defin{rebuilding} of $X$ is a collection $(X,X',\Gcell ,\Hcell ,  \Homcell )$ that consists of a CW-complex with finite $\indice$-skeleton $X'$, two cellular maps  between the $\indice$-skeleta
$$\Gcell \colon X^{(\indice)}\to X'{}^{(\indice )} \quad \mbox{and} \quad \Hcell \colon X'{}^{(\indice )}\to X^{(\indice )}$$
that are homotopy inverse to each other up to dimension $\indice-1$, and a cellular homotopy $\Homcell \colon [0,1] \times X^{(\indice -1)}\to X^{(\indice)}$  between the identity and $\Hcell \circ \Gcell$ on $X^{\indice-1}$.

Recall furthermore (Definition~\ref{def: Rebuilding and quality} of the introduction) that given two real numbers  ${T} , {\kappa}\geq 1$, an $\indice$-rebuilding $(X,X',\Gcell ,\Hcell , \Homcell)$ is of \defin{quality $(T , \kappa )$} if 
\begin{align}
\forall j \leq \indice , \quad |X'{}^{(j)}|  & \leq  {\kappa}{T}^{-1}|X^{(j)}| \tag{cells bound}\\
\forall j \leq \indice , \quad \log \|\Ghom_j \|,\log \| \Hhom_j \|,\log \| \Homhom_{j-1} \|,\log \|\p'_{j} \|  & \leq   {\kappa}(1+\log {T} ) \tag{norms bound}
\end{align}
where $\Ghom,\Hhom$ are the chain maps respectively associated to $\Gcell,\Hcell$ and where $\Homhom \colon C_\bullet(X)\to C_{\bullet+1}(X)$ is the chain homotopy induced by $\Homcell$ in the cellular chain complexes \eqref{eq: cellular chain homotopy}.

In this paragraph we make two general observations regarding rebuildings and their quality. 

First note that if $X_1\to X$ is a finite cover, every $\indice$-rebuilding $(X,X',\Gcell ,\Hcell , \Homcell )$ of $X$ induces (via the lifting property \cite[Prop. 1.33 and 1.34]{Hatcher-book-Algeb-topo}) an $\indice$-rebuilding $(X_1,X_1',\Gcell_1 ,\Hcell_1 , \Homcell_1 )$ with $X_1'=\widetilde{X'}/\pi_1(X_1)$.

\begin{lemma}[Rebuilding induced to finite cover]\label{lem:Induced rebuilding to finite cover}
Let $X$ be a finite CW-complex. 
There is a constant $\delta_X$ such that for every $\indice$-rebuilding $(X,X')$ of quality $({T}, {\kappa})$ and for every finite covering $X_1\to X$, the induced $\indice$-rebuilding $(X_1,X_1' )$ is of quality $({T},{\kappa}\delta_X)$.
\end{lemma}
\begin{proof}
A covering map of CW-complexes induces a trivial covering over each open cell. In particular: 
\begin{enumerate}
\item both sides of the cells bounds are multiplied by the degree of the cover leaving the quality of the cells bounds unchanged, and
\item the degrees of the attaching map of each cell remains bounded along coverings by a constant depending only on $X$. 
\end{enumerate}
It follows from this last observation that the norms (induced by the $\ell^2$-norms) of the boundary map and of the maps $\Ghom$, $\Hhom$ and $\Homhom$ remain bounded by a constant $\delta_X$ along coverings. 
\end{proof}

\begin{lemma}[Composition of rebuildings]\label{lem-composition}
Let $(X,X',\Gcell_1,\Hcell_1,\Homcell_1)$ and $(X',X'', \Gcell_2, \Hcell_2 ,\Homcell_2)$ be two $\indice$-rebuildings of respective quality $(T_1,{\kappa}_1)$ and $(T_2,{\kappa}_2)$, with $T_1, T_2, {\kappa}_1 ,{\kappa}_2\geq 1$. 
Let 
$$\Gcell_3= \Gcell_2\circ \Gcell_1, \quad \Hcell_3= \Hcell_1\circ \Hcell_2, \quad \mbox{and} \quad \Homcell_3(t,x)=\left\{ \begin{array}{ll} \Homcell_1(2t,x) & \mbox{if } 0\leq t\leq 1/2,\\
\Hcell_1\circ \Homcell_2 (2t-1, \Gcell_1(x)) & \mbox{if } 1/2<t\leq 1.\end{array} \right.$$ 
Then $(X,X'',\Gcell_3,\Hcell_3, \Homcell_3)$ is an $\indice$-rebuilding of quality $(T_1 T_2, 4{\kappa}_1{\kappa}_2).$
\end{lemma}
\begin{proof} The fact that $(X,X'',\Gcell_3,\Hcell_3, \Homcell_3)$ is an $\indice$-rebuilding follows from the definition. The quality of the cells bounds being  multiplicative, it remains to check the norms bounds. Since we have
\begin{equation*}
\begin{split}
\log \|\Ghom_3\| & \leq  \log \|\Ghom_1\|+\log\|\Ghom_2\|\leq {\kappa}_1(1+\log T_1 )+{\kappa}_2(1+\log T_2 )\\ 
& \leq {\kappa}_1+{\kappa}_2+{\kappa}_1 \log T_1+{\kappa}_2\log T_2  \\ 
& \leq 2{\kappa}_1{\kappa}_2(1+\log T_1 T_2 ),
\end{split}
\end{equation*}
similarly for $\Hhom_3$, and 
\begin{equation*}
\begin{split}
\log \| \Homhom_3\| &= \log \| \Homhom_1+ \Hhom_1\circ \Homhom_2\circ \Ghom_1\|\leq \log 2+\max\{ \log \|\Homhom_1\|, \log \|\Hhom_1\circ \Homhom_2\circ \Ghom_1\|\}\\
& \leq \log2+  {\kappa}_1(1+\log T_1 )+ {\kappa}_2 (1+\log T_2) + {\kappa}_1(1+\log T_1)\\
& \leq 4{\kappa}_1{\kappa}_2(1+\log T_1 T_2 ),
\end{split}
\end{equation*}
Here the degrees match well, e.g. $(\Homhom_3)_{\indice-1}=(\Homhom_1)_{\indice -1} + (\Hhom_1)_{\indice} \circ (\Homhom_2)_{\indice -1} \circ (\Ghom_1)_{\indice -1}$ in top degree.
The lemma follows.
\end{proof}

\begin{corollary}[Starting with an homotopy equivalent complex]\label{cor: homot equiv. complex}
Let $X$ be a finite CW-complex. 
Let $Y$ be a finite CW-complex that is homotopy equivalent with $X$. There are constants $(T,\kappa)$ such that if a finite cover $Y_1\to Y$ admits an $\indice$-rebuilding $Z_1$ of quality $(T_1,\kappa_1)$, then the corresponding finite cover $X_1\to X$ admits an $\indice$-rebuilding of quality $(T T_1,\kappa \kappa_1)$.
\end{corollary}
\begin{proof} 
The homotopy equivalence between $X$ and $Y$ make of $Y$ an $\indice$-rebuilding of $X$ of a certain quality 
$(T_0,\kappa_0)$.
By Lemma~\ref{lem:Induced rebuilding to finite cover}, 
the  $\indice$-rebuilding induced between the finite covers $X_1$ and $Y_1$ is of quality $(T_0,\delta_X\kappa_0)$.
The composition $\indice$-rebuilding from Lemma~\ref{lem-composition} between $X_1$ and $Z_1$ has quality $(T_0T_1,4 \delta_X \kappa_0 \kappa_1)$.
Set $T=T_0$ and $\kappa=4\delta_X \kappa_0$.
\usetikzlibrary{decorations.pathmorphing}
$$\begin{tikzcd}[column sep=1.2em]
X_1  \arrow[rrrrrr, bend left=40,"{(T T_1,\kappa \kappa_1)}",squiggly%,dashed
] \arrow[rrr,"{(T_0,\delta_X\kappa_0)}",squiggly] %\arrow[rrrrrr,"{(T_0,\delta_X\kappa_0)}"] 
\arrow[d]&&& Y_1 \arrow[d] \arrow[rrr,"{(T_1,\kappa_1)}",squiggly] &&& Z_1&\\
X \arrow[rrr,"{(T_0,\kappa_0)}",squiggly]&&& Y & &
\end{tikzcd}
$$
\end{proof}

\subsection{Unipotent lattices}

Now let $\Gamma$ be a finitely generated, torsion free nilpotent group, equivalently --- by a theorem of Malcev \cite{Malcev1} --- the group $\Gamma$ is a \defin{unipotent lattice}, i.e., it is isomorphic to a lattice in a connected, finite dimensional unipotent Lie group.

We define a central series $L_i(\Gamma)$ by 
$$L_0(\Gamma)=\Gamma \quad \mbox{and} \quad L_{i+1}(\Gamma)=\ker \left( L_i(\Gamma)\to (L_i(\Gamma)/[\Gamma,L_i(\Gamma)])\otimes_{\mathbb Z}\mathbb Q \right).$$ 
By definition, the quotients $L_i(\Gamma)/L_{i+1}(\Gamma)$ are torsion free abelian. We define the graded group 
$$\gr \Gamma=\bigoplus_{i=0}^\infty L_i(\Gamma)/L_{i+1}(\Gamma).$$ 
We have $\gr \Gamma\simeq \mathbb Z^{\Hirsch}$ where $\Hirsch$ is the \defin{Hirsch length} of $\Gamma$. For the purpose of this paper this can be taken as the definition of Hirsch length.   

The goal of this section is to prove Theorem  \ref{thm-UnipRewiring} according to which finite covers of classifying spaces of unipotent lattices admit a rebuilding of quality proportional to the degree of the cover. This is, up to a constant, the best quality one could hope for. 

To prepare for the proof we will need a few simple lemmas. 

\begin{definition}
An automorphism $\sigma$ of an unipotent lattice $\Gamma$ is a \defin{unipotent automorphism} if the induced automorphism of the graded group 
$$\gr \sigma\in \Aut(\gr \Gamma)\simeq \GL_{\Hirsch}(\mathbb Z)$$ 
is unipotent. 
\end{definition}
Equivalently $\sigma\in \Aut(\Gamma)$ is unipotent if and only if the associated semi-direct product 
$$\Gamma \rtimes_\sigma \Z  =  \langle (\gamma , t ) \; | \;  t \gamma t^{-1} =\sigma( \gamma ) \rangle$$ 
is nilpotent. 

\begin{lemma} \label{lem-UnipAuto}
Let $\Gamma$ be a non-trivial unipotent lattice and $\sigma$ be a unipotent automorphism of $\Gamma$. There exists a normal subgroup $\Gamma'\leq \Gamma$ with quotient $\Gamma/ \Gamma'\simeq \mathbb Z$ such that $\Gamma'$ is preserved by $\sigma$ and $\sigma$ acts trivially on $\Gamma/ \Gamma'$.
\end{lemma}
\begin{proof}
The group $L_1(\Gamma)$ is preserved by $\sigma$ and by construction we have $\Gamma/ L_1(\Gamma)\simeq \mathbb Z^{d}$ for some $d>0$. Write $\sigma_0\in \GL(d,\mathbb Z)$ for the automorphism induced on $\Gamma/ L_1(\Gamma)\simeq \mathbb Z^{d}$. Since $\sigma_0$ is unipotent we can find a codimension $1$ subspace $W\subseteq V=\mathbb Q^d$ such that $\sigma_0 W=W$ and the induced action on $V/W$ is trivial. Let $U=\mathbb Z^d\cap W$ and put $\Gamma'=UL_1(\Gamma).$ Then $\sigma(\Gamma')=\sigma_0(U)L_1(\Gamma)=\Gamma'$, the quotient 
$$\Gamma/\Gamma'\simeq \mathbb Z^d/ \mathbb Z^d\cap W\simeq \mathbb Z$$ 
and the induced action of $\sigma$ on this quotient is trivial.
\end{proof}

\begin{lemma} \label{lem-InnerAut}
Let $(X_0,x_0)$ be a pointed CW-complex. Let $G=\pi_1(X_0,x_0)$ and assume that $X_0$ is a classifying space for $G$. Let $\alpha$ and $\beta\colon (X_0,x_0)\to (X_0,x_0)$ be two cellular maps such that the induced morphisms 
$$\alpha_* \quad \mbox{and} \quad \beta_* : G \to G$$
are conjugated by some element $g_0 \in G$~: $\alpha_* = g_0 \beta_* g_0^{-1}$. Then, there exists a cellular homotopy 
$$H\colon I\times X_0\to X_0 \quad \mbox{such that} \quad H(0, \cdot )=\alpha, \quad H(1,\cdot )=\beta$$ 
and the loop 
$$[0,1] \to X_0, \quad s\mapsto H(s,x_0)$$ 
represents $g_0\in G$. 
\end{lemma}
\begin{proof}
We first construct an explicit cellular map $q\colon (X_0,x_0)\to (X_0,x_0)$ that induces the morphism 
$$q_* : G \to G, \quad g \mapsto g_0^{-1} g g_0 $$ 
and a (non-pointed!) homotopy between the identity map and $q$. Start with a map
$$q_0\colon (\{0\}\times X_0) \cup (I\times \{x_0\}) \to X_0, \quad q_0(0, \cdot)= \mathrm{id}_{X_0} \quad \mbox{and} \quad q_0( \cdot ,x_0)=s$$ 
where $s\colon I\to X_0$ is some loop based at $x_0$ and representing $g_0^{-1} \in G$. By Lemma \ref{lem-pconst} there is a map 
$$p\colon I\times X_0\to (\{0\}\times X_0) \cup (I\times \{x_0\}) \quad \mbox{such that} \quad p(0, \cdot )=\{ 0 \} \times \mathrm{id}_{X_0} \quad \mbox{and} \quad p(1,x_0)=(1,x_0).$$ 
Let 
$$q\colon (X_0,x_0)\to (X_0,x_0) ; \quad x \mapsto q_0 (p(1,x)).$$ 
It follows from the construction that $q$ induces the morphism 
$$q_* : G \to G, \quad g \mapsto g_0^{-1} g g_0$$ 
and the map 
$$H_0\colon I \times X_0 \to X_0;  (t,x) \mapsto q_0 (p(t,x))$$
gives a homotopy between the identity map and $q$. 
The maps 
$$q\circ \alpha \quad \mbox{and} \quad  \beta\colon (X_0,x_0)\to (X_0,x_0)$$ 
induce the same endomorphisms of $G$. It therefore follows from e.g. \cite[Proposition 7.1.6]{Geoghegan} that there exists a 
homotopy 
$$H_1\colon I\times X_0\to X_0 \quad \mbox{such that} \quad H_1(0,\cdot )=q \circ \alpha , \quad H_1(1, \cdot )=\beta \quad \mbox{and} \quad H_1( \cdot , x_0)=x_0.$$ 
Therefore, the map 
$$H(t,x)=\begin{cases}
H_0(2t,\alpha(x)) & \textrm{ if } 0\leq t\leq 1/2,\\
H_1(2t-1,x) & \textrm{ if } 1/2\leq t\leq 1,
\end{cases}$$ 
is a homotopy between $\alpha$ and $\beta$ and the loop $s\mapsto H(s,x_0)$ represents $g_0.$
\end{proof}

\begin{lemma}\label{lem-NilpotentAutmorphism}
Let $\Gamma$ be a unipotent lattice of Hirsch length $\Hirsch$ and let $\sigma\colon \Gamma\to \Gamma$ be a unipotent automorphism. There exists a $K(\Gamma,1)$ CW-complex $X_0$ with a distinguished point $x_0$ and a cellular map $\theta \colon X_0\to X_0$ fixing $x_0$ such that for every $\gamma \in \Gamma$, 
$$\theta_* (\gamma)= \sigma (\gamma).$$ 
and for every $m\in \N$ we have 
$$\log \|\theta^m \|\leq  \Hirsch\log |m|+O_\Gamma(1),$$
where $\| \theta^m \|$ denotes the norm  induced by the $\ell^2$-norm of the action of $\theta^m$ on the cellular chain complex associated to $X_0$. 
\end{lemma}
\begin{proof}
We prove the lemma by induction on the Hirsch length of $\Gamma$. The base case $\Hirsch=0$ corresponds to the trivial group and one can take $X_0$ to be the CW-complex that consists of exactly one $0$-cell. 

Suppose now that the lemma is proved up to Hirsch length $\Hirsch-1$ and consider a unipotent lattice $\Gamma$ of Hirsch length $\Hirsch$. By Lemma \ref{lem-UnipAuto} there exists a $\sigma$-invariant normal subgroup $\Gamma'\leq\Gamma$ such that $\Gamma/\Gamma'\simeq \mathbb Z$ and the induced action of $\sigma$ on $\Gamma/\Gamma'$ is trivial. Choose $t_0\in \Gamma$ such that $t_0\Gamma'$ generates  $\Gamma/\Gamma'$.

By the inductive hypothesis there exists a $K(\Gamma',1)$ space $X_0 '$ with a distinguished point $x_0'$ and a cellular automorphism $\theta ' \colon X'_0\to X'_0$ that fixes $x_0'$, induces $\sigma '$ on $\pi_1(X_0',x_0')= \Gamma'$ and satisfies 
$$\log \|(\theta ')^m\|\leq (\Hirsch-1)\log m+O_{\Gamma'}(1)$$ 
for $m\in \mathbb N$. Choose a cellular map $\tau_0 \colon X_0'\to X_0'$ which fixes $x_0'$ and induces the automorphism $\gamma\mapsto t_0^{-1}\gamma t_0$ on the fundamental group. Take $X_0$ to be the quotient 
$$X_0 =(X_0'\times [0,1]) / \sim \quad \mbox{where } \sim \mbox{ is generated by }  (x,1)\sim (\tau_0 (x ),0)  \ (x \in X_0 ')$$ 
and let $x_0$ be the image of $(x_0',0)$ in the quotient. The CW-complex $X_0$ is a $K(\Gamma , 1)$ space and the image of $\{x_0 ' \} \times [0,1]$ in $X_0$ is a loop based at $x_0$ that represents $t_0$ in $\pi_1(X_0)\simeq \Gamma$. We proceed to construct the desired cellular map $\theta \colon X_0\to X_0$. Since $\sigma$ acts trivially on $\Gamma / \Gamma'$ the two maps 
$$\tau_0 \circ \theta' \quad \mbox{and} \quad \theta' \circ \tau_0 : (X_0 , x_0) \to (X_0 , x_0)$$ 
induce two endomorphisms of $\Gamma$ that are conjugated by $t_0^{-1} \sigma (t_0) \in \Gamma '$. Lemma \ref{lem-InnerAut} therefore implies that 
there exists a homotopy 
$$H\colon I\times X_0'\to X_0' \quad \mbox{such that} \quad H(0, \cdot )= \tau_0 \circ \theta' ,  \quad H(1,\cdot )=\theta' \circ \tau_0 $$ 
and the loop 
$$[0, 1 ] \to X_0 , \quad s\mapsto H(s,x_0)$$ 
represents $t_0^{-1} \sigma (t_0)  \in \Gamma'$.
We define $\theta : X_0 \to X_0$ by 
\begin{align*}
\theta (x,t) =\begin{cases} (\theta ' (x ),2t) &\textrm{ if } 0\leq t\leq 1/2\\
(H(2t-1 , x), 0) &\textrm{ if } 1/2<t\leq 1.\end{cases}
\end{align*}
We claim that the endomorphism of $\Gamma$ induced by $\theta$ is $\sigma$. Since we already know that $\theta$ induces $\sigma$ on $\Gamma'$ it is enough to compute its value on $t_0$. By construction, $t_0$ is represented by the loop in $X_0$ image of $\{x_0 ' \} \times [0,1]$. This loop is sent  
by $\theta$ to the concatenation of $s\mapsto (x_0,s)$ with $s\mapsto H(s,x_0)$. The latter represents $t_0^{-1} \sigma (t_0)  \in \Gamma'$ so the concatenation represents $\sigma (t_0 )$ as desired. 

It remains to compute the chain maps and to check the bound on the norm. Recall that as $\mathbb Z$-modules, 
$$C_\bullet(X_0) = (C_\bullet (X_0')\otimes [0] ) \oplus (C_\bullet(X_0')) \otimes [I, \p I ].$$ 
In these coordinates, the chain map induced by $\theta$ is given by the following formulas 
$$\theta ( c\otimes [0])= \theta '(c)\otimes [0] \quad \mbox{and} \quad  \theta (c\otimes [I, \p I ])= \theta ' (c)\otimes [I , \p I ]+ r(c)\otimes [0],$$ 
where $r(c) =H([I , \p I ]\otimes c).$
By induction we have 
$$\theta^m (c\otimes [0])=\theta'^m (c)\otimes [0] \quad \mbox{and} \quad \theta^m (c\otimes [I , \p I ])=\theta'^m(c)\otimes [I , \p I ]+\sum_{i=0}^{m-1}\theta'^{i}(r(c))\otimes [0].$$ 
It follows that 
$$\|\theta^m \| \leq m \sup \{  \|\theta'^m \| ,  \| \theta'^{i} \circ r \|  \; : \; i = 1 , \ldots , m-1 \} + || r ||$$
which, by induction, gives 
$$\log\| \theta^m \| \leq \Hirsch\log m+2\log \|r\|+O_{\Gamma',\sigma'}(1)= \Hirsch\log m+ O_{\Gamma,\sigma}(1).$$
\end{proof}

\subsection{The Proof of Theorem \ref{thm-UnipRewiring}}

\begin{proof}[Proof of Theorem \ref{thm-UnipRewiring}]
We prove the theorem by induction on the Hirsch length $\Hirsch$. The base case $\Hirsch=0$ holds trivially since then $\Gamma$ is the trivial group.  

Suppose now that the lemma is proved up to Hirsch length $\Hirsch-1$ and consider a unipotent lattice $\Gamma$ of Hirsch length $\Hirsch>0$. Thanks to Corollary~\ref{cor: homot equiv. complex} we only need to prove the theorem for a single classifying space of $\Gamma$. Therefore our first step is the construction of a convenient $Y_0$. 

Let $\Lambda$ be a normal subgroup of $\Gamma$ such that $\Gamma/\Lambda\simeq \mathbb Z$. Choose $t_0\in \Gamma$ such that $t_0\Lambda$ generates $\Gamma/\Lambda$. Then $\Gamma$ decomposes as a semidirect product 
$$\Gamma=\langle t_0\rangle \ltimes \Lambda.$$
The automorphism 
$$\Lambda \to \Lambda; \quad \lambda \mapsto t_0 \lambda t_0^{-1}$$ 
is unipotent. By Lemma \ref{lem-NilpotentAutmorphism} there exists a pointed CW-complex $(X_0,x_0)$ which is a $K( \Lambda , 1)$ space and a cellular map $\theta \colon (X_0,x_0)\to (X_0,x_0)$ such that the induced endomorphism $\theta_* : \Lambda \to \Lambda$ is precisely equal to the unipotent automorphism $\lambda \mapsto t_0 \lambda t_0^{-1}$ and 
$$\|\theta^m\|\leq (\Hirsch-1)\log m+ O_{\Gamma', t_0}(1).$$ 
We let $Y_0$ be the quotient  
$$Y_0= (X_0\times [0,1] ) / (x,1)\sim (\theta (x),0)$$
and take the image $y_0$ of $(x_0 , 0)$ as basepoint. Then $\pi_1(Y_0,y_0)$ is isomorphic to $\Gamma$ and $Y_0$ is a $K(\Gamma , 1)$ space. Note for future reference that the projection of $\{ x_0 \} \times [0,1]$ into $Y_0$ is a loop that represents the element $t_0$ in $\Gamma$. 

Now let $\Gamma_1\leq \Gamma$ be a finite index subgroup. Set $\ell=[\Gamma:\Lambda\Gamma_1]$ and let $\Lambda_1 =\Lambda\cap \Gamma_1$. Since $t_0^\ell\in \Lambda \Gamma_1$ we can choose $a\in \Lambda$ such that $t_1=a t_0^\ell \in \Gamma_1.$ The group 
$\Gamma_1$ then decomposes as a semi-direct product $\Gamma_1=\langle t_1\rangle \ltimes \Lambda_1.$ It follows that $[\Gamma:\Gamma_1]=\ell[\Lambda:\Lambda_1].$ Put 
$$Y_1=\Gamma_1\bs \tilde Y_0.$$
This is the complex we want to rebuild. It is naturally a stack over the CW-complex obtained by lifting the standard CW-complex structure on the circle by a $\ell$-fold self-covering map. We rebuild $Y_1$ in two steps. The first rebuilding $(Y_1,Y_1')$ consists of changing the base of the stack to get a stack over the standard CW-complex on the circle. It decreases the number of cells to $|Y_1'|=|Y_1|\ell^{-1}$. The second rebuilding $(Y_1',Y_1'')$ is obtained by applying our effective version of the rebuilding lemma after applying the inductive hypothesis to the fibers of the stack. This second rebuilding will bring the number of cells down to $|Y_1''|=|Y_1|[\Gamma:\Gamma_1]^{-1}=|Y_0|.$
 
Before performing the two rebuildings we give an explicit description of the stack structure of $Y_1$: the map $\theta \colon X_0\to X_0$ lifts to a map $\theta \colon \tilde X_0\to \tilde X_0$ such that for all $\lambda \in \Lambda$ and for all $\tilde x \in \tilde X_0$, we have
$$\theta (\lambda \cdot \tilde x) = \theta_* (\lambda ) \cdot \tilde x.$$
In particular $\theta$ maps each $(t_0^i \Lambda_1 t_0^{-i})$-orbit of $\tilde x$ (with $i \in \mathbb Z$) to the $(t_0^{i+1} \Lambda_1 t_0^{-i-1})$-orbit of $\tilde x$. Writing
$$X_{1,(i)} =(t_0^i \Lambda_1 t_0^{-i}) \bs \tilde X_0 \quad \mbox{and simply} \quad  X_1 =X_{1,(0)},$$
we conclude that $\theta$ induces a well-defined map 
$$\theta \colon X_{1,(i)}\to X_{1,(i+1)}.$$ 
Finally since 
$$a (t_0^\ell \Lambda_1 t_0^{-\ell}) a^{-1} = t_1 \Lambda_1 t_1^{-1} = \Lambda_1 ,$$
the map 
$$\tilde X_0 \to \tilde X_0; \quad \tilde x \mapsto a \cdot \tilde x$$
induces a well-defined map 
$$a  \colon X_{1,(\ell )}\to X_0; \quad (t_0^\ell \Lambda_1 t_0^{-\ell} ) \tilde x \mapsto  \Lambda_1 a\tilde x.$$
It follows that 
\begin{equation} \label{E:Y1}
Y_1\simeq \bigsqcup_{i=0}^{\ell-1} (X_{1,(i)} \times [0,1] ) / \sim,
\end{equation}
where $\sim$ is generated by 
$$(x_{(i)},1)\sim (\theta (x_{(i)}),0), \quad \mbox{for} \quad i\in \{0,\ldots, \ell-2 \}, \quad \mbox{and} \quad (x_{(\ell-1)},1)\sim (a \circ \theta (x_{(\ell-1)}),0).$$ 
Indeed: we have
$$\Lambda_1\bs \tilde Y_0 = \bigsqcup_{i\in \mathbb Z} (X_{1,(i)} \times [0,1] ) / (x_{(i)},1) \sim (\theta (x_{(i)}) , 0)$$ 
and quotienting it further by the action of $\langle t_1\rangle$, which has the effect of identifying $(x_{(i)},t)$ with $(a (x_{(i+\ell)}),t)$, we get \eqref{E:Y1}.  

\subsection*{First rebuilding} 
Let 
$$Y_1'= X_1\times [0,1]/ (x,1)\sim (a \circ \theta^\ell (x),0).$$ 
It is aspherical and $\pi_1(Y_1')\simeq \langle t_1\rangle \ltimes \Lambda_1=\Gamma_1$ so $Y_1'$ is another $K(\Gamma_1 , 1)$ space. Note that $|Y_1'|=\ell^{-1}|Y_1|.$
We define three maps 
$$\Gcell \colon  Y_1\to Y_1', \quad \Hcell \colon Y_1'\to Y_1 \quad \mbox{and} \quad  \Homcell \colon I\times Y_1\to Y_1$$ 
by the following formulas, with $x_{(i)} \in X_{1,(i)}$, $x \in X_1$ and $s,t \in [0,1]$,
\begin{equation*}
\begin{split}
\Gcell (x_{(i)} ,t) & =\left\{ \begin{array}{ll} (x_{(i)},t) & \mbox{if } i=0,\\ (a \circ \theta^{\ell-i} (x_{(i)}),0) & \mbox{if } \quad i=1,\ldots, \ell-1, \end{array} \right. \\
\Hcell (x,t) & =(\theta^{\lfloor \ell t\rfloor}(x), \ell t-\lfloor \ell t\rfloor),
\end{split}
\end{equation*}
and
\begin{equation*}
\Homcell (s,(x_{(i)},t)) = \left\{ \begin{array}{ll} (\theta^{\lfloor \alpha \rfloor}(x_{(i)}), \alpha-\lfloor \alpha \rfloor) & \mbox{if } i+t\leq (\ell-1)s \mbox{ and } \alpha =\frac{\ell(i+t-s(\ell-1))}{\ell-s(\ell-1)},\\
(a \circ \theta^{\ell-i}(x_{(i)}) , 0) & \mbox{if } i+t\geq (\ell-1)s.
\end{array} \right.
\end{equation*}
The maps $\Gcell$ and $\Hcell$ are cellular and $\Homcell$ is a cellular homotopy between the identity and $\Hcell \circ \Gcell$. The composition $\Gcell \circ \Hcell$ is also homotopic to the identity since it induces the identity on $\pi_1(Y_1')$. Therefore $(Y_1,Y_1',\Gcell ,\Hcell ,\Homcell )$ is a rebuilding. 

We now compute the chain maps induced by $\Gcell$, $\Hcell$ and $\Homcell$ as well as the boundary map on $C_\bullet (Y_1')$, which we shall denote by 
$\p_{Y_1'}$. First remark that
$$
C_\bullet(Y_1) = \bigoplus_{i=0}^{\ell-1}\left(C_\bullet(X_{1,(i)} )\otimes [0]\oplus C_\bullet(X_{1,(i)}) \otimes [I]\right) \quad \mbox{and} \quad 
C_\bullet(Y_1')\simeq C_\bullet(X_1)\otimes [0]\oplus C_\bullet(X_1)\otimes [I],  $$ 
as $\Z$-modules. In these coordinates, for $c \in C_\bullet (X_1)$ we have
$$\p_{Y_1'}(c\otimes [0])= \p_{X_1}c\otimes [0] \quad \mbox{and} \quad 
\p_{Y_1'}(c\otimes [I])= \p_{X_1}c\otimes [I]+ (-1)^{\dim c}(-c\otimes [0]+a\circ \theta^\ell(c))\otimes [0]),$$
for $c_{(i)} \in C_\bullet(X_{1,(i)})$ we have
$$\Ghom (c_{(i)} \otimes [0])= \left\{ \begin{array}{ll}  c_{(i)} \otimes [0] &\mbox{if } i=0,\\
a\circ \theta^{\ell-i}(c_{(i)} ) \otimes [0] & \mbox{if } i>0,
\end{array} \right. \quad \mbox{and} \quad 
\Ghom (c_{(i)} \otimes [I])=  \left\{ \begin{array}{ll}  c_{(i)} \otimes [I] & \textrm{ if } i=0,\\
0 & \textrm{ if } i>0,\end{array} \right.$$
for $c \in C_\bullet (X_1)$ we have
$$\Hhom (c\otimes [0])= c\otimes [0] \quad \mbox{and} \quad 
\Hhom (c\otimes [I])= \sum_{i=0}^{\ell-1} \theta^i(c)\otimes [I],$$
and for $c_{(i)} \in C_\bullet(X_{1,(i)})$ we have
$$\Homhom (c_{(i)} \otimes [0])= \left\{ \begin{array}{ll} 
0 & \mbox{if } i=0,\\ 
\sum_{j=i}^{\ell-1} \theta^j(c)\otimes [I], & \mbox{if } i>0,
\end{array}\right. \quad \mbox{and} \quad  
\Homhom (c_{(i)} \otimes [I])= 0.$$
The  map $a \colon C_\bullet (X_{1,  (\ell)} )\to C_\bullet(X_1)$ is unitary and  by Lemma \ref{lem-NilpotentAutmorphism} $\|\theta^\ell\| \leq  O_{X_0,\sigma} (1+ \log \ell ).$ We deduce that 
\begin{equation}\label{eq-NormEstimates}
\|\p_{Y_1'}\|, \quad \|\Ghom \|, \quad \|\Hhom \|, \quad \| \Homhom \| \leq O_{Y_0} (1+ \log \ell ). 
\end{equation} 
This concludes the first step.

\subsection*{Second rebuilding} 
By induction there exists a rebuilding $(X_1,X_1', \Gcellfibre, \Hcellfibre , \Homcellfibre )$ of quality $([\Lambda:\Lambda_1], O_{X_0}(1)).$ 
To shorten notation we shall write $\theta_1 =a \circ \theta^\ell.$ Let 
$$A =X_1\times\{0,1\} \quad \mbox{and} \quad A' =X_1'\times\{0,1\}.$$ 
Define a map $f\colon A\to X_1$ by 
$$f(x,0)=x \quad \mbox{and} \quad f(x,1)=\theta_1(x).$$ 
We then have 
$$Y_1'=X_0\times [0,1]/(x,1)\sim (\theta_1 (x),0)=(X_1\times [0,1])\cup_f X_1.$$
Now consider the following diagram 
\begin{equation*}
\begin{tikzcd}[sep=large]
X_1\times [0,1]\arrow[d,"\Gcellfibre \times \mathrm{id}",swap, shift right] &\arrow[l,hook] A\arrow[d,swap,"\Gcellfibre \times \mathrm{id}",shift right] \arrow[r,"f"] & X_1\arrow[d,swap,"\Gcellfibre", shift right]\\
X_1'\times [0,1]\arrow[u," \Hcellfibre\times \mathrm{id}",swap, shift right]&\arrow[l,hook] A'\arrow[r,"f'"] \arrow[u," \Hcellfibre\times \mathrm{id}",swap, shift right]& X_1'\arrow[u," \Hcellfibre",swap, shift right],
\end{tikzcd}
\end{equation*}
where $f'= \Gcellfibre\circ f\circ ( \Hcellfibre \times \mathrm{id})$. Define 
$$Y_1''=(X_1'\times [0,1])\sqcup_{f'} X_1'.$$ The number of cells in $Y_1''$ is $2|X_1'|=O_{X_0}(1)$ and 
$$C_\bullet(Y_1'') = C_\bullet(X_1)\otimes [I]\oplus C_\bullet(X_1),$$ 
as $\mathbb Z$-modules. The boundary map is given by 
$$\p_{Y_1'}(c\otimes [I])=\p_{X_1'}(c)\otimes [I]+(-1)^{\dim c}(-\Ghomfibre (\theta_1 (\Hhomfibre (c)))+\Ghomfibre ( \Hhomfibre (c)) \quad \mbox{and} \quad \p_{Y_1'}(c)=\p_{X_1'}(c),$$ 
for $c\in C_\bullet(X_1').$

Let 
$$\Gcell ' \colon Y_1'\to Y_1'' \quad \mbox{and} \quad \Hcell ' \colon Y_1''\to Y_1'$$ 
be the cellular maps afforded by Proposition \ref{lem-QGluing} and let 
$$\Homcell ' \colon I\times Y_1'\to Y_1'$$ 
be the homotopy between the identity and $\Hcell ' \circ \Gcell '$, also provided by Proposition \ref{lem-QGluing}. Let 
$$ \Homhomfibre \colon C_\bullet(X_1)\to C_{\bullet+1}(X_1) \quad \mbox{and} \quad \Homhom ' \colon C_\bullet(Y_1')\to C_{\bullet+1}(Y_1')$$ 
be the chain homotopies respectively induced by $\Homcellfibre$ and $\Homcell '$.

By induction we have 
$$\log \|\Ghomfibre \|, \quad \log \|\Hhomfibre \|, \quad \log \|\Homhomfibre \|\leq O_{X_0} ( 1+ \log [\Lambda:\Lambda_1] ) \quad \mbox{and} \quad \log \|f\|\leq O_{Y_0} (1+\log \ell ).$$  
Using the explicit formulas for the chain maps from Lemma \ref{lem-QGluing} we deduce that 
\begin{equation}\label{eq-NormEstimates2}
\log\| \p_{Y_1'} \|, \log\| \Ghom ' \|, \log \| \Hhom ' \|, \log \| \Homhom ' \|\ll_{Y_0} \log\ell+ \log [\Lambda:\Lambda_1]+1=\log [\Gamma:\Gamma_1]+1.
\end{equation}
Let
$$\Gcell ''=\Gcell ' \circ \Gcell, \quad \Hcell ''=\Hcell \circ \Hcell '$$ 
and 
$$\Homcell ''(s,y)= \left\{ 
\begin{array}{ll} 
\Homcell (2s, y) & \mbox{if } 0\leq s\leq 1/2,\\
\Hcell ( \Homcell ' (2t-1,g(y)) & \mbox{if } 1/2\leq s\leq 1,
\end{array} \right. \quad \mbox{for} \quad y\in Y_1.$$ 
Using \eqref{eq-NormEstimates} and \eqref{eq-NormEstimates2} we see that $(Y_1,Y_1'', \Gcell '', \Hcell'', \Homcell'')$ is a rebuilding of quality $([\Gamma:\Gamma_1],O_{Y_0}(1))$. 
This proves Theorem~\ref{thm-UnipRewiring}.
\end{proof}

\section{Quality of rebuilding of  extensions by unipotent lattices} \label{S:unipext}

Consider a countable group $\Glemma$ that contains a finitely generated, torsion-free, nilpotent, \defin{normal} subgroup $\Alemma$. Suppose furthermore that $\Glemma / \Alemma$ is of type $F_{\indice}$ for some integer $\indice \geq 0$. 

Let $\basezero$ be a classifying space (CW-complex) for $\Glemma / \Alemma$ with finite $\indice$-skeleton.
The group $\Glemma$ acts on its universal cover $\widetilde{\basezero}$ with cell-stabilizers all equal to $\Alemma$.

Let $Y$ be a finite classifying space for $\Alemma$. The Borel construction followed by the Proposition~\ref{prop: rebuilding}
yields a stack of CW-complexes $\stackzero \to \basezero$ with base $\basezero$ where each fiber is $Y$.  

The goal of this section is to prove the following effective rebuilding statement.

\begin{proposition} \label{L3}
Let $\Glemone \leq \Glemma$ be a finite index subgroup, let $\stackone \to \baseone$ be the associated stack and let $ \Alemone = \Alemma \cap \Glemone$. Then the total space $\stackone$ is a classifying space for $\Glemone$ and there exists an $\indice$-rebuilding $(\stackone , \stackone ' , \Gcell , \Hcell , \Homcell )$ of quality $([\Alemma : \Alemone ] , O_{Y,D} (1) )$.
\end{proposition}
\begin{proof} 
Each fiber of $\stackzero \to \basezero$ has the same number of cells $n_i \in \mathbb{N}$ in each dimension $i$ (with $n_i=0$ above the dimension of $Y$). 

The CW-complex $\basezero$ has finitely many cells in each dimension $j \leq \indice$; we denote by $m_j$ this number. The total number of cells of dimension $\ell \leq \indice$ of the total complex $\stackzero$ is then finite and equal to 
$$N_\ell = \sum_{i+j=\ell } n_i m_j.$$  

Now consider the stack $\stackone \to \baseone$. The total space $\stackone$ has $[\Glemma : \Glemone ]N_\ell$ cells in each dimension $\ell \leq \indice$. More precisely, each fiber $Y_1$ of $\stackone \to \baseone$ is a classifying space for $ \Alemone = \Alemma \cap \Glemone$ which contains $[\Alemma : \Alemone] n_i$ cells in each dimension $i$ and $\baseone$ contains $[\Glemma / \Alemma : \Glemone /  \Alemone] m_j$ cells in each dimension $j \leq \indice$. Thus:
$$N_\ell [\Glemma : \Glemone ]=\sum_{i+j=\ell} [\Alemma :  \Alemone ] n_i\ [\Glemma / \Alemma : \Glemone  /  \Alemone] m_j.$$

The group $\Alemma$ being finitely generated, torsion-free and nilpotent, Theorem \ref{thm-UnipRewiring} implies that there exists an $\indice$-rebuilding $(Y_1 , Y_1 ' , \Gcellfibre , \Hcellfibre , \Homcellfibre)$ of quality $([\Alemma :  \Alemone] , O_{Y} (1))$. In particular $Y_1'$ is a classifying space for $A_1$ and letting $\stackone ' \to \baseone$ be the rebuilding of the stack $\stackone \to \baseone$ associated to the rebuilding $Y_1'$ of the fibers we conclude that $\stackone ' $ has
$$O_Y \left( \sum_{i+j=\ell} n_i \ [\Glemma / \Alemma : \Glemone  /  \Alemone] m_j \right) = O_Y \left( \frac{N_\ell [\Glemma : \Glemone ]}{[\Alemma :  \Alemone ]} \right)$$
cells in each dimension $\ell \leq \indice$.  

In fact Proposition \ref{P2} provides $(\stackone , \stackone ' , \Gcell , \Hcell , \Homcell)$ and expresses the boundary map $\partial '$ on  $C_n (\stackone ' )$, the chain maps induced by $\Gcell$ and $\Hcell$, and the chain homotopy induced by $\Homcell$ in terms of $\Ghomfibre$, $\Hhomfibre$, $\Homhomfibre$, the (vertical) boundary map on $C_\bullet (Y )$ and the horizontal boundary map on $C_\bullet (\basezero )$. The norms of the two boundary maps are bounded by $O_{Y,D} (1)$ and, since the rebuilding $(Y_1 , Y_1 ' , \Gcellfibre , \Hcellfibre , \Homcellfibre)$ is of quality $([\Alemma :  \Alemone] , O_{Y} (1))$,  the log of the norms of $\Ghomfibre$, $\Hhomfibre$ and $\Homhomfibre$ are bounded by $O_{Y,D} ( 1 + \log [\Alemma :  \Alemone] )$. The formulas \eqref{E:P2} therefore imply that $(\stackone , \stackone ' , \Gcell , \Hcell , \Homcell)$ is of quality $([\Alemma : \Alemone ] , O_{Y,D} (1) )$. Indeed, consider for example the boundary map $\partial ' $ on $C_j (\stackone ' )$ for some $j \leq \indice$. We have: 
\begin{eqnarray*}
\Vert \partial ' \Vert &\leq & \Vert(\partial ' )^{\rm vert} \Vert+ \Vert \Ghomfibre \Vert \ \ \left\Vert \left( \sum_{i=0}^{j} \left(  \partial^{\rm hor} \circ \Homhomfibre  \right)^i \right) \right\Vert\ \  \Vert \partial^{\rm hor} \Vert\ \Vert \Hhomfibre \Vert
\\
&\leq & \Vert(\partial ' )^{\rm vert} \Vert+
P_j([\Alemma :  \Alemone])
\end{eqnarray*}
where $P_j$ is a polynomial of degree $O_{Y,D} (1)$ whose coefficients do not depend on the subgroup $\Alemone$. It follows that 
$$\log \Vert \partial ' \Vert = O_{Y,D} (1 + \log [\Alemma :  \Alemone] ) .$$
The bounds for $\Ghom$, $\Hhom$ and $\Homhom$ are obtained similarly. 
\end{proof}

\section{A general quantitative rebuilding theorem}
\label{sect: proof of main th}

Let $\indice$ be a positive integer and let $\Gamma \acting \Baseup$ be a CW-complex action of a countable group $\Gamma$ that satisfies the following assumptions. 
\begin{description}
\item[(Cond 1)]   The CW-complex $\Baseup$ is $(\indice-1)$-connected.
\item[(Cond 2)]   For every cell $\cellup\subseteq \Baseup$ the stabilizer $\Gamma_\cellup$ acts trivially on $\cellup$.
\label{item: H2}
\item[(Cond 3)]  The quotient CW-complex $\Basedown:=\Gamma \backslash \Baseup$ has finite $\indice$-skeleton $\Basedown^{(\indice)}$.
\item[(Cond 4)]  The group $\Gamma$ is of type $F_{\indice+1}$. 
\item[(Cond 5)]  Each stabilizer group $\Gamma_\cellup$ is of type $F_{\indice}$.
\end{description}
Recall from Proposition \ref{prop: existence of a stack} that, as long as $\Baseup$ is assumed to be simply connected, the Borel construction associates to this action a stack of CW-complexes $\Totaldown \to \Basedown :=\Gamma \backslash \Baseup$ such that $\Totaldown$ has an $(\indice-1)$-connected universal cover and the fundamental group $\pi_1 (\Totaldown)$ isomorphic to $\Gamma$. 

\begin{theorem}[Quantitative rebuilding] \label{Tmain}
Suppose that $\Gamma$ is finitely presented and that for each cell $e$ of $\Basedown^{(\indice)}$ the group $\Gamma_e$ is finitely generated and contains a finitely generated, torsion-free, nilpotent, normal subgroup $Z_e$. 
There exists a constant ${\kappa}$ such that for every finite index normal subgroup $\Gammaprime \leq \Gamma$, there exists a CW-complex $\Totaldown_1^+$ with finite $(\indice+1)$-skeleton such that the following hold.
\begin{enumerate}
\item The CW-complex $\Totaldown_1^+$ has an $\indice$-connected universal cover.
\item The fundamental group $\pi_1 (\Totaldown _1 ^+)$ is isomorphic to $\Gammaprime$. 
\item In each dimension $\leq \indice$, the total number of cells of $\Totaldown_1^+$ is 
\begin{equation}
\leq {\kappa} \sum_{e \in \Basedown^{(\indice)}} \frac{[\Gamma : \Gammaprime]}{[Z_e : \Gammaprime \cap Z_e ]}.\end{equation}
\item In each degree $\leq \indice +1$, the norm of the boundary operator on the chain complex $C_\bullet (\Totaldown_1 ^+ )$ is 
\begin{equation}\leq {\kappa} [\Gamma : \Gammaprime]^{\kappa}.\end{equation}
\end{enumerate}
\end{theorem}

To prove the theorem we will consider the stack of CW-complexes $\Totaldown \to \Basedown :=\Gamma \backslash \Baseup$. However for the fundamental group of the latter to be isomorphic to $\Gamma$, we need $\Omega$ to be simply connected. This follows from the hypothesis for $\indice \geq 2$ but in certain interesting cases, e.g., for $\Gamma=\SL_3(\mathbb Z)$, the natural candidate for $\Omega$ is the Tits complex which fails to be simply connected. We need a way to ``fix" these spaces before we start using them. This is the content of the following:

\begin{lemma}\label{lem-SimplyConn} Let $\Gamma$ be a finitely generated group. Consider a $1$-dimensional co-compact $\Gamma$-CW-complex $\Baseup$ whose cells are fixed pointwise by their stabilizers.
Assume that all vertex stabilizers are finitely generated. Then there exists a simply connected $\Gamma$-CW-complex $\Baseup^+$ 
whose $1$-skeleton is $\Baseup$. Moreover, if $\Gamma$ is finitely presented then $\Baseup^+$ is $\Gamma$-cocompact.
\end{lemma}

\begin{proof}

Up to refining the structure of the CW-complex on $\Baseup$, we may assume that it is a graph.
The point is to kill the fundamental group of $\Baseup$ by gluing some $2$-cells along $\Gamma$-orbits of loops in $\Baseup$.

By \cite[Theorem 9.2, p. 39]{Dicks-Dunwoody-book} the group $\Gamma$ sits in a short exact sequence
\begin{equation}\label{eq: short exact sequ Dicks-Dunwoody}
1\to \pi_1(\Baseup)\to G \overset{\zeta }{\to} \Gamma\to 1
\end{equation}
where $G$ is the fundamental group of the natural quotient graph of groups associated with the action $\Gamma\acting \Baseup$.
The assumption that the quotient space is compact and the vertex stabilizers are finitely generated imply that $G$ is finitely generated.
Thus $G$ is the quotient of a free group $\FF(\tilde{S})$ (where $\tilde{S}$ is finite).
By \eqref{eq: short exact sequ Dicks-Dunwoody}, $\Gamma$ is also such a quotient. 
Thus we have surjective morphisms:
\[\FF(\tilde{S})\overset{\eta }{\to} G\overset{\zeta }{\to} \Gamma.\]
The group $\Gamma$ is the quotient of $\FF(\tilde{S})$ by a normal subgroup whose image by $\eta$ is exactly the kernel of $\zeta $, i.e., it is $\pi_1(\Baseup)$. Pick $(\rho_j)$ a collection of elements of $G$ such that the kernel of $\zeta$ is the normal subgroup generated by the $\rho_j$'s. If $\Gamma$ is finitely presented, it admits a finite presentation 
\[\Gamma=\langle \tilde{S} | r_1, r_2, \cdots, r_p\rangle.\]
It follows that the collection $(\rho_j)$ can be chosen to be the finite collection $\eta (r_1), \eta (r_2), \ldots, \eta (r_p)$.

Consider now the action of $G$ on its Bass-Serre tree $T$, a base vertex $v$ in $T$ and the paths $c_j$ in $T$ from $v$ to $\rho_j (v)$.
They project to loops $\sigma_j$ in $\Baseup=\pi_1(\Baseup)\bs T$. Since these loops normally generate $\pi_1(\Baseup)$, gluing a $2$-cell along the $\sigma_j$'s and extending $\Gamma$-equivariantly produces the required simply connected complex $\Baseup^+$. If $\Gamma$ is finitely presented, this complex is $\Gamma$-cocompact.
\end{proof}

\begin{proof} [Proof of the Quantitative Rebuilding Theorem~ \ref{Tmain}]
We need to work with an $\Baseup$ that is simply connected. If $\indice \geq 2$ this follows from the assumptions. For $\indice=1$ we can use Lemma \ref{lem-SimplyConn} to replace $\Baseup$ by (its $1$-skeleton and then by) a simply connected $2$-dimensional complex $\Baseup^+$ satisfying all the assumptions of the theorem. In any case, the Borel construction followed by the rebuilding Lemma~\ref{prop: rebuilding} 
then associates to the action of $\Gamma$ on $\Omega$ a stack of CW-complexes $\Totaldown \to \Basedown :=\Gamma \backslash \Baseup$ such that the fundamental group of $\Totaldown$ is isomorphic to $\Gamma$ and such that the fiber over each cell $e=\Gamma \cellup$ is a classifying space $F_e$ of $\Gamma_{\cellup}$ with finite $\indice$-skeleton.

Now consider the stack of CW-complexes
\begin{equation} \label{E:stackprime}
\Totaldown_1  \to \Basedownprime
\end{equation}
associated to the finite index subgroup $\Gammaprime \leq \Gamma$, so that 
$$\Totaldown_1 = \Gammaprime \backslash \widetilde{\Totaldown} \quad \mbox{and} \quad \Basedownprime = \Gammaprime \backslash \Baseup.$$
Each cell $e \in \Basedown$ is covered by $\# ( \Gamma_e \backslash \Gamma / \Gammaprime)$ cells in $\Gammaprime \backslash \Baseup$ and, since $\Gammaprime$ is a normal subgroup of $\Gamma$, the fibers of \eqref{E:stackprime} over each of these cells are all isomorphic to the finite cover $D_e$ of $F_e$ associated to $\Gammaprime \cap \Gamma_e \leq \Gamma_e$. 

By hypothesis, as long as $\dim e \leq \indice$, the group $\Gamma_e$ contains a finitely generated, torsion-free, nilpotent, normal subgroup $Z_e$. We may therefore apply Proposition \ref{L3} with $G = \Gamma_e$ and $G_1 = \Gamma_e \cap \Gammaprime$. It follows that there exists a rebuilding $(D_e , D_e ' , \Gcellfibre_e , \Hcellfibre_e , \Homcellfibre_e )$ of quality $( [Z_e : \Gammaprime \cap Z_e] , O_{\Gamma , \Baseup} (1))$ of each fiber $D_e$ of \eqref{E:stackprime}. 

Applying the rebuilding lemma (Proposition \ref{P2}) to \eqref{E:stackprime} we get a stack of CW-complexes 
$$\Pi ' : \Totaldown_1 ' \to \Basedownprime,$$ 
cellular homotopy equivalences $\Gcell_1 : \Totaldown_1 \to \Totaldown_1 '$, $\Hcell_1 :  \Totaldown_1' \to \Totaldown_1 $ and a homotopy $\Homcell_1$ between the identity and $\Hcell_1 \circ \Gcell_1$. Note that over each cell of $\Basedownprime$ covering a cell $e$ of $\Basedown$, the fiber of $\Pi'$ is $D_e '$.

The total space $\Totaldown_1'$ is homotopy equivalent to $\Totaldown_1$ 
and therefore has an $(\indice-1)$-connected universal cover and a fundamental group isomorphic to $\Gammaprime$. Moreover, in each dimension $n \leq \indice$, the total number of cells of $\Totaldown_1 '$ is
$$O \left( \sum_{e \in \Basedown^{(n)}} \# ( \Gamma_e \backslash \Gamma / \Gammaprime) \frac{[\Gamma_e : \Gammaprime \cap \Gamma_e]}{[Z_e : \Gammaprime \cap Z_e]} \right) = O \left(  \sum_{e \in \Basedown^{(\indice)}} \frac{[\Gamma : \Gammaprime]}{[Z_e : \Gammaprime \cap Z_e ]} \right).$$

To control the norm of the boundary operator on the chain complex $C_\bullet (\Totaldown_1 ')$ we apply Proposition \ref{P2}  as in the proof of Proposition \ref{L3}. Equation \eqref{E:P2} implies that the boundary operator $\partial' $ of the (rebuilt) chain complex $C_\bullet (\Totaldown_1 ')$ can be written in terms of the (vertical) boundary operator in the fibers, the (vertical) maps ($\Ghomfibre$, $\Hhomfibre$ and $\Homhomfibre$) and the boundary operators $\partial$ and $\partial^{\rm vert}$ acting on the chain complex $C_\bullet ( \Totaldown_1 ) = C_\bullet ( \Gammaprime \backslash \widetilde{\Totaldown} )$ (before rebuilding).

The norm of the boundary operators $\partial$ and $\partial^{\rm vert}$ are bounded by constants that depend only on the local combinatorial structure of $\Totaldown$, see Lemma \ref{lem:Induced rebuilding to finite cover}.

Now over each cell $e$ of $\Basedown$ of dimension $\leq \indice$, the rebuilding $(D_e , D_e ' , \Gcellfibre_e , \Hcellfibre_e , \Homcellfibre_e )$ is of quality $( [Z_e : \Gammaprime \cap Z_e] , O_{\Gamma , \Baseup} (1))$ and it follows that the norm of the vertical boundary operator $(\partial ')^{\rm vert}$ is  
$O([\Gamma_e : \Gamma_e \cap \Gammaprime ]^{O(1)} )$ and therefore $O([\Gamma : \Gammaprime ]^{O(1)} )$. For the same reason, the norms of the vertical maps $\Ghomfibre_e$, $\Hhomfibre_e$ and $\Homhomfibre_e$ are bounded by $O([\Gamma_e : \Gamma_e \cap \Gammaprime ]^{O(1)} )$, and formula 
\eqref{E:P2} finally implies that the norm of the boundary operator on the chain complex $C_{\leq \indice} (\Totaldown_1 ')$ is a $O([\Gamma : \Gammaprime ]^{O(1)} )$.

The universal cover of $\Totaldown_1'$ is `only' $(\indice -1)$-connected. However, since $\Gamma$ is of type $F_{\indice+1}$, it follows from \cite[Theorem 8.2.1]{Geoghegan} that it is possible to attach \emph{finitely many} $\Gamma$-orbits of $(\indice+1)$-cells to $\widetilde{\Totaldown}$ to make an $\indice$-connected $\Gamma$-CW complex. 
Write the quotient as $\Totaldown^+ = \Totaldown \sqcup_f \Basedown$,
with $\Basedown=\sqcup_{I} \mathbb B^{\indice+1}$ a finite collection of $(\indice+1)$-cells and $f\colon \partial \Basedown=\sqcup_{I} \mathbb S^{\indice}\to \Totaldown^{(\indice)}$ the map that attaches these $(\indice+1)$-cells to $\Totaldown$. Then write 
$$\Gamma_1 \bs \widetilde{\Totaldown^+} = \Totaldown_1 \sqcup_{f_1} \Basedown_1$$ 
where $\Basedown_1 = \sqcup_{I_1} \mathbb B^{\indice+1}$ is the preimage of $\Basedown$ in $\Totaldown_1^+$ and $f_1 \colon \partial \Basedown_1 =\sqcup_{I_1} \mathbb S^{\indice}\to \Totaldown_1^{(\indice)}$ is the lift of $f$. We have a diagram 
\begin{equation*}
\begin{tikzcd}[sep=large]
\Basedown_1 \arrow[d,"\mathrm{id}",swap, shift right] &\arrow[l,hook] \partial \Basedown_1 \arrow[d,swap,"\mathrm{id}",shift right] \arrow[r,"f_1"] & \Totaldown_1 \arrow[d,swap,"g", shift right]\\
\Basedown_1 \arrow[u,"\mathrm{id}",swap, shift right]&\arrow[l,hook] \partial \Basedown_1 \arrow[r,"\varphi"] \arrow[u,"\mathrm{id}",swap, shift right]& \Totaldown_1' \arrow[u,"h",swap, shift right],
\end{tikzcd}
\end{equation*} 
with $\varphi =g\circ f_1.$ It then follows from Proposition \ref{lem-QGluing} that the space 
$$\Totaldown_1^+ = \Totaldown_1 ' \sqcup_{\varphi} \Basedown_1$$ 
is homotopy equivalent to $\Totaldown_1 \sqcup_{f_1} \Basedown_1  .$ 
The map $f_1$, being a lift of $f$, is of norm $\|f_1 \|\leq \|f\|$. Finally the norm of the boundary map on $\Totaldown_1^+$ in degree $\indice+1$ is bounded by $\|\varphi\|\leq \|f_1\|\|g\|.$  
The latter being of norm $O([\Gamma : \Gammaprime ]^{O(1)} )$ we conclude that the norm of the degree $\indice+1$ boundary operator on the chain complex $C_\bullet (\Totaldown_1^+ )$ is $O([\Gamma : \Gammaprime ]^{O(1)} )$. 
\end{proof}

\section{Bounding torsion - A proof of Gabber's proposition~\ref{Prop: Gabber}}
\label{sect: proof of Gabber's prop}

In this section we prove the following useful proposition attributed to Gabber (see \cite[Prop. 3, p. 214]{Soule-99}). Our proof here follows the viewpoint of \cite[Section 2]{BV}.

Let $(C_{\bullet} , \p)$ be the cellular chain complex associated to a finite CW-complex $\Totaldown$. Each $C_{j}$ is a free $\Z$-module of finite rank with a canonical basis associated to the $j$-cells of $\Totaldown$. We equip each finite dimensional vector space $C_j \otimes \R$ with the associated Euclidean norm. For any coefficient field $K$ and for any integer $j \geq 0$, it follows from the definition of the homology groups that
\begin{equation} \label{E:bettiK}
\dim_K H_j (C_\bullet \otimes K) \leq \mathrm{rank} \ C_j .
\end{equation}
To bound the torsion homology we will use the following analogous observation. 

\begin{proposition}[Gabber]
 \label{Prop: Gabber}
For any $j \geq 0$, 
\begin{equation}
\log |H_j (C_\bullet )_{\rm tors} | \leq (\mathrm{rank} \ C_j ) \times \sup ( \log ||\p_{j+1}|| , 0).
\label{eq: Gabber ineq}\end{equation}
Here $|| \p_{j+1} ||$ denotes the operator norm, associated with the euclidian norm on the $C_\bullet$. 
\end{proposition}
\begin{proof}
Given a finite rank free $\Z$-module $A$, such that $A_{\R} = A\otimes\R$ is endowed with a positive definite inner product $(\cdot , \cdot )$ (a \defin{metric} for short), we define $\vol (A)$ to be the volume of $A_{\R} / A$. When considered without further precision, the free $\Z$-module $\Z^a$ ($a \in \mathbb{N}^*$) will denote the standard one where $\R^a= \Z^a \otimes \R$ is endowed with the canonical metric. 

Let $a>0$, $b>0$ be integers and 
\begin{equation}f : \Z^a \to \Z^b\end{equation} 
be a $\Z$-linear map. We set $\det ' (f)$ to be the product of all nonzero singular values of $f$. Recall that the nonzero singular values of $f$ are --- with multiplicity --- the positive square roots of the nonzero eigenvalues of $ff^*$. Note that we have 
$$\det {}' (f) \leq \sup (||f|| , 1)^{\min (a,b)}$$
where $||f||$ denotes the operator norm of $f_\R : \R^a \to \R^b$.

Now recall from \cite[(2.1.1)]{BV} the ``metric rank formula''
\begin{equation} \label{BT1}
\det {}' (f) = \vol (\mathrm{image} \ f ) \ \vol (\mathrm{ker} \ f).
\end{equation}
Here we understand the metrics on $(\mathrm{ker} \ f) \otimes \R$ and $(\mathrm{image} \ f ) \otimes \R$ as those induced from $\R^a$ and $\R^b$. Let $Q = \mathrm{coker} \ f$ be the cokernel of $f$. It is a finitely generated $\Z$-module. Write $Q= Q_{\rm tors} \oplus Q_{\rm free}$ its decomposition into a torsion and a free part.

Writing 
\begin{equation}\R^b = (\mathrm{image} \ f ) \otimes \R \oplus (\mathrm{image} \ f )^{\perp}\end{equation}
and applying \cite[(2.1.1)]{BV} to the orthogonal projection $\Z^b \to (\mathrm{image} \ f )^{\perp}$ we conclude that the quotient $1 / {\vol (\mathrm{image} \ f )}$ is the product of $|Q_{\tors}|^{-1}$ with the `regulator' $\vol (Q_{\rm free})$, where the metric on $Q_{\rm free} \otimes \R$ is obtained by identifying it as a subspace of $(\mathrm{image} \ f )^{\perp}$. In summary,
\begin{equation} \label{BT2}
\frac{1}{\vol (\mathrm{image} \ f )} = \frac{\vol (Q_{\rm free})}{|Q_{\tors}|^{-1}}.
\end{equation}
It follows from \eqref{BT1} and \eqref{BT2} that 
\begin{equation} \label{BT3}
|Q_{\rm tors} | = \det {}' (f) \frac{\vol (Q_{\rm free})}{\vol (\mathrm{ker} \ f)} \leq \det {}' (f) .
\end{equation}
The last inequality follows from two facts:
(1)
$\mathrm{ker} \ f $ being a sub-lattice of $\Z^a$
 we   have $\vol (\mathrm{ker} \ f) \geq 1$, and
(2) $\vol (Q_{\rm free}) \leq 1$ since $Q_{\rm free}$ is spanned by vectors of length at most one.

Proposition \ref{Prop: Gabber} follows from the equation \eqref{BT3} applied to 
$$Q = \mathrm{coker} \left( \p_{j+1} : C_{j+1} \to C_j \right).$$
Indeed, the homology group $H_j (C_\bullet )$ is contained in $Q$ and it follows from \eqref{BT3} that the size of the torsion part of $Q$ is smaller than 
$$\det {}' (\p_{j+1}) \leq \sup(||\p_{j+1} || , 1)^{\mathrm{rank} C_j} .$$
\end{proof}

\section{Farber sequences and cheap rebuilding property} \label{S:Farber}
\subsection{Farber neighborhoods}
\label{Farber neighborhoods}

Let $\Gamma$ be a countable group.
Let $\Sub_\Gamma$ denote the space of subgroups of $\Gamma$ equipped with the topology induced from the topology of pointwise convergence on $\{0,1\}^\Gamma$. 
The subset $\Sub^{\findex}_\Gamma\subseteq \Sub_\Gamma$ of finite index subgroups is equipped with the induced topology. It is countable when $\Gamma$ is finitely generated.
The group $\Gamma$ continuously acts by conjugation on both $\Sub_\Gamma$ and $\Sub^{\findex}_\Gamma$.

We consider the fixed point ratio function defined for finite index subgroups $\Gamma'\leq \Gamma$:
$${\rm fx}_{\Gamma,\gamma}\colon \Sub^{\findex}_\Gamma\to  [0,1], \quad \Gamma' \mapsto \frac{|\{g\Gamma'| \gamma g\Gamma'=g\Gamma'\}|}{[\Gamma:\Gamma']}.$$
Thus ${\rm fx}_{\Gamma,\gamma}(\Gamma')$ is just the proportion of fixed points of $\gamma\in \Gamma$ in the action $\Gamma\acting \Gamma/\Gamma'$.

\begin{definition}
\label{def: Farber sequence}
 A sequence $(\Gamma_n)_{n\in\mathbb N}$  of subgroups of $\Gamma$ is called a \defin{Farber sequence} if it consists of finite index subgroups and for every $\gamma\in \Gamma\setminus \{1\}$ we have 
$\lim_{n\to\infty} {\rm fx}_{\Gamma,\gamma}(\Gamma_n)=0$.
\end{definition}
Though we won't use it, note that if $\Gamma$ is finitely generated and $S\subseteq \Gamma$ is a finite symmetric generating set, the sequence $(\Gamma_n)_{n\in\mathbb N}$ is Farber if and only if the sequence of Schreier graphs $\Sch(\Gamma_n\bs \Gamma, S)$ converges to the Cayley graph $\Cay(\Gamma,S)$ in the Benjamini--Schramm topology \cite{BS}; if and only if the sequence of actions $\Gamma\acting\Gamma/ \Gamma_n$ defines a sofic approximation of $\Gamma$.

Observe that $\Gamma$ admits Farber sequences if and only if it is residually finite.
The notion is designed to encompass non-normal finite index subgroups. 

\begin{definition}[Farber neighborhood]
Let $\Gamma$ be a residually finite group.
An open subset $U$ of $\Sub^{\findex}_\Gamma$ is a \defin{$\Gamma$-Farber neighborhood} if it is $\Gamma$-invariant and every Farber sequence in $\Sub^{\findex}_\Gamma$ eventually belongs to $U$. 
\end{definition}
We can think of these $\Gamma$-Farber neighborhoods as neighborhoods of $\{\mathrm{id}\}$ in $\Sub^{\findex}_\Gamma$, except that $\{\mathrm{id}\}\not\in \Sub^{\findex}_\Gamma$.
\begin{example}
\label{ex: basis of Farber neighborhoods}
Assume $\Gamma$ is residually finite. The sets \[U_{\Gamma,S,\delta}=\left\{\Gamma'\in \Sub^{\findex}_\Gamma \; \left| \; {\rm fx}_{\Gamma,\gamma}(\Gamma')< \delta \textrm{ for }\gamma\in S\right.\right\},\]
where  $S\subseteq \Gamma\setminus\{1\}$ is a finite subset and $\delta>0$, are non-empty and form a basis of $\Gamma$-Farber neighborhoods.
Let  $(\gamma_j)_{j\in \N}$ be an enumeration of $\Gamma$ and 
$S_n=\{\gamma_0, \gamma_1, \cdots, \gamma_n\}$, then the same holds for $U_{\Gamma,S_n,\frac{1}{n}}$.
If $\Gamma_n\in U_{\Gamma,S_n,\frac{1}{n}}$,
 then $(\Gamma_n)_{n\in \N}$ is a Farber sequence.
It follows that if $V\subseteq  \Sub^{\findex}_{\Gamma}$ does not contain any $U_{\Gamma,S,\delta}$, then we can construct a Farber sequence $(\Gamma_n)_{n\in \N}$ as above that does
 not meet $V$. 
\end{example}

\begin{lemma}\label{lem-FarberIntersection}
Let $\Gamma$ be a residually finite  group and let $\Lambda\leq \Gamma$ be an infinite subgroup. For every $\Lambda$-Farber neighborhood  $U\subseteq  \Sub^{\findex}_{\Lambda}$ and $\delta>0$,  there exists a $\Gamma$-Farber neighborhood $V\subseteq  \Sub^{\findex}_{\Gamma}$ such that for any $\Gamma'\in V$ we have 
\[\frac{\left\vert\left\{\gamma\in \Gamma / \Gamma' \colon \gamma \Gamma'\gamma^{-1}\cap \Lambda\in U\right\}\right\vert}{[\Gamma:\Gamma']}\geq 1- \delta.\]
\end{lemma}

In words, the lemma says that the finite index subgroups $(\Gamma')^\gamma\cap \Lambda$ of $\Lambda$ belong to a prescribed $\Lambda$-Farber neighborhood  for a large proportion of the conjugates $(\Gamma')^\gamma$, as soon as $\Gamma'$ belongs to a small enough $\Gamma$-Farber neighborhood.

\begin{proof}
It is enough to prove the statement for $U=U_{\Lambda,S,\delta}$ (from Example~\ref{ex: basis of Farber neighborhoods}) for any finite $S\Subset \Lambda$ and $1>\delta>0$.

Let $V = U_{\Gamma, S , \frac{\delta^2}{\vert S\vert}}$. Then, for $\Gamma'\in V$ and each $\gamma \in S$ simple combinatorics arguments give 
\begin{equation*}
\begin{split}
 \frac{\delta^2}{\vert S\vert}  \geq  \textrm{fx}_{\Gamma, \gamma }(\Gamma') & =\frac{1}{[\Gamma:\Gamma']} \sum_{g \in \Gamma / \Gamma '} 1_{(\Gamma')^g} (\gamma  ) =
\frac{1}{[\Gamma:\Gamma']} \sum_{g \in \Lambda \bs \Gamma / \Gamma' } \sum_{\lambda \in \Lambda / ((\Gamma')^g\cap \Lambda )} 1_{((\Gamma')^g\cap \Lambda)^\lambda} (\gamma) \\
& = \frac{1}{[\Gamma:\Gamma']} \sum_{g \in \Gamma / \Gamma '} \left(\frac{1}{[\Lambda:(\Gamma')^g\cap \Lambda]}\sum_{\lambda \in \Lambda / ((\Gamma')^g\cap \Lambda )} 1_{((\Gamma')^g\cap \Lambda)^\lambda} (\gamma) \right) \\
& = \frac{1}{[\Gamma:\Gamma']}\sum_{g \in \Gamma / \Gamma '}  {\rm fx}_{\Lambda,\gamma } ((\Gamma')^g \cap \Lambda).
\end{split}
\end{equation*}
Hence, 
$$|\{g \in \Gamma / \Gamma' \colon  {\rm fx}_{\Lambda,\gamma }((\Gamma')^g \cap \Lambda)\geq \delta\}|\leq [\Gamma:\Gamma']  \frac{\delta}{\vert S\vert}$$ 
and 
$$|\{g \in \Gamma / \Gamma' \colon \exists \gamma\in S \text{ s. t. } {\rm fx}_{\Lambda,\gamma }((\Gamma')^g \cap \Lambda)\geq \delta\}|\leq [\Gamma:\Gamma'] \delta .$$ 
Thus 
$$\left\vert\{g \in \Gamma / \Gamma' \colon {\rm fx}_{\Lambda,\gamma }((\Gamma')^g \cap \Lambda)\leq \delta \text{ for all } \gamma\in S\}\right\vert\geq [\Gamma:\Gamma'] (1-\delta).$$  
We conclude that 
\[\frac{1}{[\Gamma : \Gamma']}\sum_{g \in \Gamma / \Gamma' }1_{U_{\Lambda,S,\delta}}((\Gamma')^g \cap \Lambda)\geq 1-\delta,\] 
as desired.
\end{proof}

\subsection{The cheap rebuilding property}

\begin{definition}[Cheap $\indice$-Rebuilding Property, Farber sequences]
\label{def-CheapReb -Farb-sequ} 
Let $\Gamma$ be a residually finite group and let $\indice$ be a non-negative integer.
A Farber sequence  $(\Gamma_n)_n$ of $\Gamma$ has the \defin{Cheap $\indice$-Rebuilding Property}  if
there exists a $K(\Gamma , 1)$ space $X$ with finite $\indice$-skeleton and a constant ${\kappa}_X\geq 1$ such that for every real number $T\geq 1$, there is $n_0$ such that when $n\geq n_0$
the finite covers $Y_n\to X$ with $\pi_1(Y_n)=\Gamma_n$ admit an $\indice$-rebuilding $(Y_n,Y_n')$ of quality $(T,{\kappa}_X).$ 
\end{definition}

The group $\Gamma$ itself has the \defin{Cheap $\indice$-Rebuilding Property} if the existence of the complex $X$ and of the constant ${\kappa}_X$ holds in a ``uniform way'' for all Farber sequences. More precisely:
\begin{definition}[Cheap $\indice$-Rebuilding Property, groups]
\label{def-CheapReb} 
Let $\Gamma$ be a countable group and let $\indice$ be a non-negative integer.
The group $\Gamma$ has the \defin{Cheap $\indice$-Rebuilding Property} if it is residually finite and
there exists a $K(\Gamma , 1)$ space $X$ with finite $\indice$-skeleton and a constant ${\kappa}_X\geq 1$ such that for every real number $T\geq 1$, 
there exists a $\Gamma$-Farber neighborhood $U=U(X,T)\subseteq \Sub^{\findex}_\Gamma$ such that
  every finite cover $Y\to X$ with $\pi_1(Y)\in U$ admits an $\indice$-rebuilding $(Y,Y')$ of quality $(T,{\kappa}_X).$ 
\end{definition} 

The simplest group with cheap $\indice$-rebuilding property for every $\indice$ is the infinite cyclic group $\Z$ as we show in Lemma~\ref{lem-ZRebuilding}.
Many other groups have this cheap $\indice$-rebuilding property. See Corollary~\ref{cor-CheapRebGroups} for some examples.

\begin{remark}
\label{rem: cheap 0-reb prop}
It is important to note that finite groups do not have the cheap $\indice$-rebuilding property, for any $\indice$: for each finite group there is a bound on the qualities of its rebuildings.
In fact, a residually finite group has the cheap $0$-rebuilding property if and only if it is infinite.
\end{remark}

\begin{lemma}
\label{lem: exist-> forall in def CPR}
One can replace in both Definitions  ``there exists a $K(\Gamma , 1)$ space with finite $\indice$-skeleton and there is a constant ${\kappa}_X$''
by ``for every $K(\Gamma , 1)$ space with finite $\indice$-skeleton,  there is a constant ${\kappa}_X$''.
\end{lemma}

\begin{proof}
Let $X$ and $X'$ be $k$-aspherical CW-complexes with finite $\indice$-skeleta and $\pi_1(X) \simeq \pi_1(X') \simeq \Gamma$.
Assume that  $X'$ satisfies the  Definition~\ref{def-CheapReb -Farb-sequ}, then by basic obstruction theory,  there exist cellular maps 
$$\Gcell_1 \colon X^{(\indice)}\to X'{}^{(\indice)} \quad \mbox{and} \quad \Hcell_1 \colon X'{}^{(\indice)}\to X^{(\indice)}$$
that are homotopy inverse to each other up to dimension $\indice-1$, 
and a cellular homotopy $\Homcell_1 \colon [0,1] \times X^{(\indice-1)}\to X^{(\indice)}$  between the identity of $X^{(\indice-1)}$ and the restriction of $\Hcell_1 \circ \Gcell_1$ to $X^{(\indice-1)}$.
Observe for the case $\indice=1$ that the isomorphism between the fundamental groups allows us to define $\Gcell_1$ and $\Hcell_1$ up to the $2$-skeleta.

Let $\kappa_1$ be an upper bound for all the norms $ \|(\Hhom_1)_{j}\|$, $\|(\Ghom_1)_{j}\|$ for $j\in \{ 0, \ldots , \indice\}$ and $\|(\Homhom_1)_{j}\|$ for $j=0, 1, \cdots, \indice-1$
 that moreover satisfies
$|X'{}^{(j)}| \leq \kappa_1|X^{(j)}| $ for all  $j\in \{ 0, \ldots , \indice\}$.
Thus $(X,X')$ is an $\indice$-rebuilding of quality $(1,\kappa_1)$.

For every real number $T\geq 1$, 
there exists a $\Gamma$-Farber neighborhood $U=U(X',T)\subseteq \Sub^{\findex}_\Gamma$ such that
  every finite cover $Y'\to X'$ with $\pi_1(Y')\in U$ admits an $\indice$-rebuilding $(Y',Y'')$ of quality $(T,{\kappa}_{X'}).$ 
Let $Y\to X$ be the finite cover associated with $(\Hcell_1)_*(\pi_1(Y'))\leq \pi_1(X)\simeq \Gamma$. 
By Lemma~\ref{lem:Induced rebuilding to finite cover} (Rebuilding induced to finite cover) the $\indice$-rebuilding $(X,X')$ induces a rebuilding $(Y,Y')$ of quality $(1,\kappa_1\delta_X)$. Applying Lemma~\ref{lem-composition} (Composition of rebuildings) to the rebuildings $(Y,Y')$ and $(Y', Y'')$ we get an $\indice$-rebuilding $(Y, Y'')$ of quality $(T, 4\kappa_1\delta_X\kappa_{X'})$.
\end{proof}

\begin{theorem}\label{thm-AdjRebStack}
Let $\Gamma$ be a residually finite  group acting on a CW-complex $\Baseup$ in such a way that any element stabilizing a cell fixes it pointwise. Let $\indice$ be a non-negative integer and assume that the following conditions hold:
\begin{enumerate}
\item $\Gamma\bs \Baseup$ has finite $\indice$-skeleton; 
\item $\Baseup$ is $(\indice-1)$-connected;
\item For each cell $\cellup\in \Baseup$ of dimension $j\leq \indice$ the stabilizer $\Gamma_\cellup$ has the cheap $(\indice-j)$-rebuilding property.
\end{enumerate}
Then $\Gamma$ itself has the cheap $\indice$-rebuilding property. 
\end{theorem}
%Note that in order to apply the theorem for $\alpha\geq 2$, the  finite presentation assumption is not needed and follows in fact from the other hypotheses.

\begin{proof}

If $\indice = 0$, this follows from Remark~\ref{rem: cheap 0-reb prop} and (3): 
the stabilizer of any $0$-cell is infinite, thus so is $\Gamma$.
\\
 From now on we assume $\indice \geq 1$. 
As in the proof of Theorem \ref{Tmain} we need to work with a simply connected $\Baseup$. For $\indice \geq 2$ this follows from the assumptions.  For $\indice=1$, 
we note that the vertex stabilizers ($j=0$) have the $1$-rebuilding property and thus are finitely generated, so we can use Lemma \ref{lem-SimplyConn}
 to replace $\Baseup$ by (its $1$-skeleton and then by) a $1$-connected $2$-dimensional complex $\Baseup^+$ satisfying all the assumptions of the theorem.

Since all the information we need is located in the $\indice$-skeleton (or $2$-skeleton for $\indice=1$), we may assume that $\Baseup$ has dimension at most $\indice$ (or $2$ for $\indice=1$).

Given a contractible CW-complex $E\Gamma$  with a free action of $\Gamma$ (i.e., the universal cover of some classifying space for $\Gamma$), the Borel construction (Section~\ref{Sect: Borel construction}) 
considers the product $\Baseup \times E\Gamma$ with the diagonal action. It is $(\indice-1)$-connected, the fundamental group of the quotient $\Gamma\bs (\Baseup \times E\Gamma)$ is isomorphic to $\Gamma$ and 
 the projection map $\Gamma\bs (\Baseup \times E\Gamma)\to \Gamma \bs \Baseup$ is interpreted (by Proposition~\ref{prop: existence of a stack}) as a stack of CW-complexes  
 with fiber $\simeq \cellup\times (\Gamma_\cellup \bs E\Gamma)\simeq \Gamma_\cellup \bs E\Gamma$ 
over each cell $\Gamma \cellup$ of $\Gamma\bs \Baseup$.

Let $[\Baseup]:=\{\cellup_1,\ldots, \cellup_N\}\subseteq \Baseup$ be a list of distinct representatives 
of the $\Gamma$-orbits of cells of dimension $\leq \indice$; there are finitely many of them, by hypothesis~(1).
Let $\cellup\in [\Baseup]$ be a cell of dimension $j \leq \indice$. 
By hypothesis the stabilizer $\Gamma_\cellup$ has the cheap $(\indice -j)$-rebuilding property and is therefore of type $F_{\indice-j}$. 
Let $X_\cellup$ be a classifying space for $\Gamma_\cellup$ with finite $(\indice-j)$-skeleton. Since $X_\cellup$ and $\Gamma_\cellup \bs E\Gamma$ are homotopy equivalent, the Geogeghan's rebuilding Lemma (Proposition \ref{prop: rebuilding}) yields a stack $\Totaldown\to \Gamma\bs \Baseup$ with fiber $X_\cellup$ over each cell $\Gamma\cellup$ of $ \Gamma\bs \Baseup$. 
Per condition (1), the quotient $\Gamma\bs \Baseup$ has finite $\indice$-skeleton so the whole stack $\Totaldown$ has finite $\indice$-skeleton. At this point we might as well forget how we constructed $\Totaldown$, the only important properties to keep in mind are that it fits into a stack  $\Pi: \Totaldown\to\Gamma\bs\Baseup$, it has an $(\indice-1)$-connected universal cover, it has finite $\indice$-skeleton, that $\pi_1(\Totaldown)$ is isomorphic to $\Gamma$
and each fiber over $\Gamma\cellup$ is a CW-complex $X_\cellup$ with $(\indice-1)$-connected universal cover, finite $\indice$-skeleton,  and $\pi_1(X_\cellup) \simeq \Gamma_\cellup$.

Now let $\Gamma_1 \leq \Gamma$ be a finite index subgroup of $\Gamma$ and let $\Totaldown_1$ be the finite cover of $\Totaldown$ corresponding to $\Gamma_1$. 
Any $\Gamma$-orbit of cells $\Gamma\cellup\subseteq \Baseup$ splits into a family of $\Gamma_1$-orbits indexed by the double cosets:
\[\Gamma\cellup=\bigsqcup_{\gamma\in \Gamma_1\bs\Gamma/\Gamma_{\cellup}}  \Gamma_1\gamma \cellup .\]

The complex $\Totaldown_1$ is naturally the total space of a stack over $\Gamma_1 \bs \Baseup$.
If $\gamma\cellup$ is any cell of $\Baseup$ $(\gamma\in \Gamma$) and $\Gamma_{\gamma\cellup} \cap \Gamma_1=\gamma\Gamma_{\cellup}\gamma^{-1}\cap \Gamma_1 $ is its stabilizer for the $\Gamma_1$-action, then the fiber of this stack over the cell $\Gamma_1\gamma\cellup$ of $\Gamma_1\bs\Baseup$ takes the form $X_{1,\gamma\cellup}=(\gamma\Gamma_{\cellup}\gamma^{-1}\cap \Gamma_1 ) \bs \widetilde X_{\gamma\cellup}
 \simeq (\Gamma_{\cellup}\cap \gamma^{-1}\Gamma_1\gamma )\bs \widetilde X_\cellup$. 
\begin{center}
\begin{tikzcd}[column sep=1.2em]
\widetilde{\Totaldown} \arrow[rrr]
 \arrow[d]& &&\Totaldown_1 =\Gamma_1\bs \widetilde{\Totaldown}\arrow[r] \arrow[d]& \Totaldown=\Gamma\bs \widetilde{\Totaldown} \arrow[d]\\
\Baseup \arrow[rrr]& && \Gamma_1\bs\Baseup \arrow[r]&   \Gamma\bs\Baseup
\end{tikzcd}
\end{center}

We remark that since $\Totaldown_1$ is a cover of $\Totaldown$, we know that the norm (induced from the $\ell^2$-norm) of the boundary map $\partial\colon C_{\bullet}(\Totaldown_1 )\to C_{\bullet-1}(\Totaldown_1 )$ is bounded in degrees $\bullet\leq \indice$
by a constant depending only on $\Totaldown$, not on $\Gamma_1$. 

We will perform some rebuilding of $\Totaldown_1$ and  determine conditions on $\Gamma_1$ under which its quality is good enough.

\subsection*{Step 1. Rebuilding the stack} 

 The cheap-rebuilding property (assumption (3)) of the stabilizers $\Gamma_{\cellup_i}$ applied to $X_{\cellup_i}$ (by Lemma~\ref{lem: exist-> forall in def CPR}, we have the freedom of the space in Definition~\ref{def-CheapReb}) gives constants $c_{X_{\cellup_i}}\geq 1$ for which the following choice is possible:
 \\
Let $T\geq 1$ be a real number.
 We
 choose for each $i=1,\ldots, N$ a $\Gamma_{\cellup_i}$-Farber neighborhood $U_i\subseteq \Sub^{\findex}_{\Gamma_{\cellup_i}}$ and for every finite index subgroup $\Lambda\leq \Gamma_{\cellup_i}$ with $\Lambda\in U_i$ we choose an $(\indice-\dim(\cellup_i))$-rebuilding 
 $R(\cellup_i,T,\Lambda)$ of $\Lambda \bs \widetilde{X}_{\cellup_i}$ 
 of quality $(T, \kappa_{X_{\cellup_i}})$.  

We want to use the effective rebuilding Lemma (Proposition \ref{P2}) to rebuild the stack $\Totaldown_1=\Gamma_1\bs \widetilde{\Totaldown}$. To do that, we need to specify a rebuilding of each fiber $X_{1, \cellup}$, with $\Gamma \cellup \subseteq \Gamma \bs \Baseup.$  The full set of fibers of $\Totaldown_1\to \Gamma_1\bs\Baseup$
consists of the CW-complexes $X_{1 , \gamma\cellup_i}$, with  $i=1,\ldots,N$ and $\gamma\in \Gamma_1 \bs \Gamma/\Gamma_{\cellup_i}$. Note that 
$$X_{1,\gamma{\cellup_i}}=(\Gamma_{\gamma{\cellup_i}} \cap \Gamma_1) \bs \widetilde X_{\gamma{\cellup_i}}
 \simeq (\Gamma_{{\cellup_i}}\cap \gamma^{-1}\Gamma_1\gamma )\bs \widetilde X_{\cellup_i}.$$

We rebuild $X_{1,\gamma{\cellup_i}}$ according to whether $\Gamma_{1,\cellup_i, \gamma}:=(\Gamma_{\cellup_i}\cap \gamma^{-1}\Gamma_1\gamma )\in U_i\subseteq \Sub^{\findex}_{\Gamma_{\cellup_i}}$ or not by using the rebuilding:
\begin{equation}\label{eq: X'1 gamma omega i}
 (X_{1, \gamma\cellup_i},X'_{1, \gamma\cellup_i}, \Gcellfibre_{\gamma\cellup_i} , \Hcellfibre_{\gamma\cellup_i},\Homcellfibre_{\gamma\cellup_i}):=
\begin{cases}
R(\cellup_i,T,\Gamma_{1,\cellup_i, \gamma}) &\text{ if } \Gamma_{1,\cellup_i, \gamma}\in U_i\subseteq \Sub^{\findex}_{\Gamma_{\cellup_i}}\\
\text{ $X'_{1, \gamma\cellup_i}=X_{1, \gamma\cellup_i}$ }
 &\text { if } \Gamma_{1,\cellup_i, \gamma}\not\in U_i
\end{cases}
\end{equation}
i.e., we simply use the trivial rebuilding $(X_{1, \gamma\cellup_i},X_{1, \gamma\cellup_i}, id ,id,0)$ when 
$\Gamma_{1,\cellup_i, \gamma}\not\in U_i$.

By virtue of Proposition \ref{P2} we obtain a global rebuilding $(\Totaldown_1, \Totaldown_1 ' , \Gcell ,  \Hcell , \Homcell )$.
Our goal is now to estimate its quality. 
Recall that the tension in Definition~\ref{def: Rebuilding and quality} of quality (in the introduction) 
 is between ``having few cells'' and ``maintaining tame norms''.
Observe that for the fibers associated with $\Gamma_{1,\cellup_i, \gamma}\in U_i$, 
both are 
 controlled by definition of the $\Gamma_{\cellup_i}$-Farber neighborhood. In particular, the norms of the vertical maps in these fibers are bounded by a polynomial in $T$. 
 As for the fibers associated with $\Gamma_{1,\cellup_i, \gamma}\not\in U_i$, 
the quality is ``very good'' as far as the ``norms bound" is concerned.
The number of cells will be controlled later on in Step 2.

Let us simply denote by $\Ghomfibre \colon C_\bullet(X_{1,\cellup} )\to C_\bullet(X'_{1,\cellup})$ and $\Hhomfibre \colon C_\bullet( X'_{1,\cellup})\to C_\bullet(X_{1,\cellup})$ the maps respectively induced by $\Gcellfibre_\cellup$ and $\Hcellfibre_\cellup$. Similarly, let $\Homhomfibre \colon C_\bullet(X_{1,\cellup} )\to C_{\bullet+1}(X_{1,\cellup})$ be the chain homotopy map induced by $\Homcellfibre_\cellup$. Finally let $\Ghom \colon C_\bullet(\Totaldown_1 ) \to C_\bullet(\Totaldown_1 ' )$ and $\Hhom \colon C_\bullet(\Totaldown_1 ' )\to C_\bullet(\Totaldown_1)$ be the chain maps induced by $\Gcell$ and $\Hcell$ and let $\Homhom \colon C_\bullet (\Totaldown_1 ) \to C_{\bullet +1} (\Totaldown_1  )$ be the chain homotopy map induced by $\Homcell$.

As in the paragraphs preceding and preparing to Proposition \ref{P2}, we consider the decomposition   
\[ C_\bullet(\Totaldown_1 )=\bigoplus_{\Gamma_1 \cellup \in \Gamma_1 \bs \Baseup}  [\cellup] \otimes C_\bullet(X_{1,\cellup} ) \quad \mbox{and} \quad   C_\bullet(\Totaldown_1 ')=\bigoplus_{\Gamma_1 \cellup \in \Gamma_1 \bs \Baseup}  [\cellup] \otimes C_\bullet(X'_{1,\cellup}).\] 
By a slight abuse of notation we will simply write $\heartsuit$ instead of the $1\otimes \heartsuit$ for $\heartsuit=\Ghomfibre$, $\Ghomfibre_\cellup$, $\Hhomfibre$, $\Hhomfibre_\cellup$, $\Homhomfibre$ or $\Homhomfibre_\cellup.$  We write $\partial$, $\partial^{\rm hor}$, $\partial^{\rm vert}$ for the boundary, horizontal boundary and vertical boundary maps in the stack $C_\bullet(\Totaldown_1)$ and $\partial '$, $(\partial ')^{\rm hor}$, $(\partial ')^{\rm vert}$ for the boundary, horizontal boundary and vertical boundary maps on $C_\bullet(\Totaldown_1 ')$. By Proposition \ref{P2} formulae \eqref{E: effective k}--\eqref{E:P2} 
we have 
\begin{equation*}
\begin{split}
\Ghom & = \Ghomfibre \circ \left( \sum_{i=0}^{\infty} \left( \partial^{\rm hor} \circ \Homhomfibre \right)^i \right), \\
\Hhom & = \left( \sum_{i=0}^{\infty} \left(  \Homhomfibre \circ \partial^{\rm hor}   \right)^i \right) \circ \Hhomfibre, \\
\Homhom & = \Homhomfibre \circ \left( \sum_{i=0}^{\infty} \left( \partial^{\rm hor} \circ \Homhomfibre \right)^i \right), \\
\partial ' & = (\partial ' )^{\rm vert} + \Ghomfibre \circ \left( \sum_{i=0}^{\infty} \left( \partial^{\rm hor} \circ \Homhomfibre \right)^i \right) \circ  \partial^{\rm hor} \circ \Hhomfibre.
\end{split}
\end{equation*}
Recall that in each of these expressions the internal sum is in fact finite since the summands vanish for $i$ large enough. 
The norms of $ (\partial ')^{\rm vert}$, $\Ghomfibre$, $\Hhomfibre$, $\Homhomfibre$ and $\partial^{\rm hor}$ are bounded by polynomials in $T$ with coefficients depending on $\Totaldown$ and on the constants $\kappa_{X_{\cellup_i}}$, but independent of $\Gamma_1$. This is thus also the case for the norms of $\Ghom$, $\Hhom$, $\Homhom$ and $\partial'$ the chain maps appearing in the definition of the quality of a rebuilding. Their $\log$ is thus bounded by a constant (independent of $\Gamma_1$) times $(1+\log T)$. This ensures the norms bound condition for every $\Gamma_1$.

\subsection*{Step 2. Counting cells} In remains to ensure a control on the number of cells of $\Totaldown_1'$.
Let  
$\card{d}(X)$ denote the number of $d$-cells of $X$.
We need to exhibit a $\Gamma$-Farber neighborhood $V\subseteq \Sub^{\findex}_\Gamma$ such that if $\Gamma_1$ belongs to $V$ then 
$$\card{j}(\Totaldown'_1) \leq O( \card{j}(\Totaldown_1)/T )\quad \mbox{for each } j \in \{ 0,\ldots , \indice \}.$$ 

Given an arbitrary finite index  $\Gamma_1\leq \Gamma$, we intend to count the number of cells in $\Totaldown_1 '$. For that purpose we introduce $N$ functions $F_i : \Gamma \to \R$, with $i \in \{ 1 , \ldots , N\}$, corresponding to 
the chosen distinct representatives $[\Baseup]:=\{\cellup_1,\ldots, \cellup_N\}\subseteq \Baseup$ of the $\Gamma$-orbits of cells in $\Baseup$, 
 by the following formulas
$$F_i (\gamma) = 
\left\{ 
\begin{array}{ll}
T^{-1} & \text{ if } 
(\Gamma_{\cellup_i}\cap \gamma^{-1}\Gamma_1\gamma )\in U_i\\
1 &  \text{ otherwise},
\end{array} \right.$$
where $U_i$ is the $\Gamma_{\cellup_i}$-Farber neighborhood 
chosen at the beginning of Step 1.
Denoting $q:=\dim(\omega_i)$, recall that to each $\Lambda\in U_i\subseteq \Sub^{\findex}_{\Gamma_{\cellup_i}}$
we have associated a cheap $(\indice-q)$-rebuilding $(X_{1, \gamma\cellup_i},X'_{1, \gamma\cellup_i})$ in equation~\eqref{eq: X'1 gamma omega i} of quality $(T, \kappa_{X_{\cellup_i}})$.  
In particular, the cells bound gives
$\card{\ell}(X'_{1,\gamma\cellup_i})\leq \kappa_{X_{\cellup_i}}\times  F_i (\gamma)\times \card{\ell}(X_{1,\gamma\cellup_i})$, for every $\ell\in \{0, 1, \cdots, \indice -q\}$.

By our choices of fibers rebuildings we thus get:
\begin{equation*} 
\begin{split}
\card{j}(\Totaldown_1 ') & = \sum_{q=0}^{j}\sum_{\substack{\cellup\in \Gamma_1 \bs \Baseup\\ \dim(\cellup)=q}} \card{j-q}(X'_{1,\cellup})= \sum_{\begin{smallmatrix}\cellup\in \Gamma_1 \bs \Baseup\\ \dim(\cellup)\leq j\end{smallmatrix}} \card{j-\dim(\cellup)}(X'_{1,\cellup})\\
&= \sum_{q=0}^{j} \sum_{\substack{\cellup_i\in [\Baseup]\\ \dim(\cellup_i)=q}}
\sum_{\gamma\in \Gamma_1 \bs \Gamma/\Gamma_{\cellup_i}}  \card{j-q}(X'_{1,\gamma\cellup_i})\\
&\leq \sum_{q=0}^{j} \sum_{\substack{\cellup_i\in [\Baseup]\\ \dim(\cellup_i)=q}}
\sum_{\gamma\in \Gamma_1 \bs \Gamma/\Gamma_{\cellup_i}} \kappa_{X_{\cellup_i}}\times F_i (\gamma) \times \card{j-q}(X_{1,\gamma\cellup_i})\\
&\leq {\kappa}^{\Baseup} \sum_{q=0}^{j}  \sum_{\begin{smallmatrix}\cellup_i\in [\Baseup]\\ \dim(\cellup_i)=q\end{smallmatrix}}
 \left(
\underbrace{\sum_
{\substack{\gamma\in \Gamma_1 \bs \Gamma/\Gamma_{\cellup_i}\\{\phantom{F_i(\gamma)=1}}}}
  \frac{\card{j-q}(X_{1,\gamma\cellup_i})}{T}}_{(A)}
+
(1-\frac{1}{T})
\underbrace{\sum_{\substack{\gamma\in \Gamma_1 \bs \Gamma/\Gamma_{\cellup_i}\\ 
F_i(\gamma)=1}
} \card{j-q}(X_{1,\gamma\cellup_i})}_{(B)}
\right)
\end{split}
\end{equation*} 
where ${\kappa}^{\Baseup}:=\max\{ \kappa_{X_{\cellup_i}}: i\in \{1, \cdots, N\}\} $.

In order to bound the part associated with (A), observe that
\[\sum_{q=0}^{j}  \sum_{\begin{smallmatrix}\cellup_i\in [\Baseup]\\ \dim(\cellup_i)=q\end{smallmatrix}}
\sum_{\gamma\in \Gamma_1 \bs \Gamma/\Gamma_{\cellup_i}} \card{j-q}(X_{1,\gamma\cellup_i})
=\card{j}(\Totaldown_1).\]

In order to bound the part associated with (B), observe that\\
\begin{itemize}
\item the covering $\Totaldown_1\to \Totaldown$ induces the covering $X_{1,\gamma\cellup_i}\to X_{\cellup_i}$; and thus 
\[\card{j-q}(X_{1,\gamma\cellup_i})=[\Gamma_{\gamma \cellup_i} : \Gamma_{\gamma \cellup_i}\cap \Gamma_1]
\card{j-q}(X_{\cellup_i});\]
\item $[\Gamma_{\gamma \cellup_i} : \Gamma_{\gamma \cellup_i}\cap \Gamma_1]=[\Gamma_{\cellup_i}:\gamma^{-1}\Gamma_1\gamma\cap \Gamma_{\cellup_i}]$ is exactly the number of $\Gamma_1\bs\Gamma$-classes that are gathered together to form the $\Gamma_1\bs\Gamma/\Gamma_{\cellup_i}$-class of $\gamma$; 
\item the invariance under conjugation of the $\Gamma_{\cellup_i}$-Farber neighborhoods ensures that $F_i(\gamma)=F_i(\gamma \lambda)$ for every $\lambda\in \Gamma_{\cellup_i}$. It follows that
\[
\sum_{\begin{smallmatrix}\gamma\in \Gamma_1 \bs \Gamma/\Gamma_{\cellup_i}\\ F_i(\gamma)=1
\end{smallmatrix}} 
\card{j-q}(X_{1,\gamma\cellup_i})=\sum_{\begin{smallmatrix}\gamma\in \Gamma_1 \bs \Gamma/\Gamma_{\cellup_i}\\ F_i(\gamma)=1
\end{smallmatrix}} 
[\Gamma_{\gamma \cellup_i} : \Gamma_{\gamma \cellup_i}\cap \Gamma_1]
\card{j-q}(X_{\cellup_i})=
\sum_{\begin{smallmatrix}\gamma\in \Gamma_1 \bs \Gamma\\ F_i(\gamma)=1
\end{smallmatrix}} 
\card{j-q}(X_{\cellup_i}).
\]
\end{itemize}

By Lemma \ref{lem-FarberIntersection}, there exists a $\Gamma$-Farber neighborhood $V\subseteq \Sub^{\findex}_\Gamma$ such that for $\Gamma_1 \in V$ and $i \in \{ 1,\ldots, N \}$ we have 
\[\left\vert\left\{\gamma\in \Gamma_1\bs \Gamma \colon F_i(\gamma)=1\right\} 
\right\vert=
\left\vert\left\{\gamma\in \Gamma / \Gamma_1\colon \gamma \Gamma_1\gamma^{-1}\cap \Gamma_{\cellup_i} \not\in U_i\right\}\right\vert\leq T^{-1}[\Gamma:\Gamma_1].\]
Thus
\begin{eqnarray*}
\sum_{\begin{smallmatrix}\gamma\in \Gamma_1 \bs \Gamma/\Gamma_{\cellup_i}\\ F_i(\gamma)=1
\end{smallmatrix}} 
\card{j-q}(X_{1,\gamma\cellup_i})&\leq& T^{-1}[\Gamma:\Gamma_1]\card{j-q}(X_{\cellup_i})
\\
\sum_{q=0}^{j}  \sum_{\begin{smallmatrix}\cellup_i\in [\Baseup]\\ \dim(\cellup_i)=q\end{smallmatrix}}
\sum_{\begin{smallmatrix}\gamma\in \Gamma_1 \bs \Gamma/\Gamma_{\cellup_i}\\ F_i(\gamma)=1
\end{smallmatrix}} 
\card{j-q}(X_{1,\gamma\cellup_i})
&\leq&
\underbrace{\sum_{q=0}^{j}  \sum_{\begin{smallmatrix}\cellup_i\in [\Baseup]\\ \dim(\cellup_i)=q\end{smallmatrix}} T^{-1}[\Gamma:\Gamma_1]\card{j-q}(X_{\cellup_i})
}_{=T^{-1}\card{j}(\Totaldown_1)}.
\end{eqnarray*}
This finishes the proof of Theorem~\ref{thm-AdjRebStack}.
\end{proof}

\subsection{First examples}

\begin{lemma}\label{lem-ZRebuilding}
$\mathbb Z$ has the cheap $\indice$-rebuilding property for any $\indice$. 
\end{lemma}
\begin{proof}
We will verify the property for $X=[0,1]/\{0\sim 1\}$ being the circle with one $0$-cell and one $1$-cell. Given a positive integer $m$ we denote by $X_m=[0,m]/\{0\sim m\}$ the $m$-fold cover of $X$ associated to the subgroup $m \Z$ of $\Z$ and equipped with the cell structure lifted from $X$.

Let $T$ be a positive real number. The subset $U \subseteq \mathrm{Sub}^{\findex}_{\Z}$ that consists of all the finite index subgroups $N \Z$ in $\Z$ with $N \geq 4T$ is a $\Z$-Farber neighborhood. Pick some subgroup $N\Z$ in $U$. We now explain how to construct a rebuilding of quality $(T,O (1))$ of $X_N$.
To do so, first pick some sequence of integers 
$$0=a_0<a_1\ldots< a_{m} = N \quad \mbox{with} \quad T/2\leq a_{i+1}-a_i\leq T.$$ 
For each $t\in [0,N)$ let $\iota(t)$ be the integer in $[0 , m-1]$ determined by $t\in [a_{\iota(t)},a_{\iota(t)+1}).$ 
The cellular maps $\Gcell \colon X_N\to X_m$ and $\Hcell \colon X_m\to X_N$ defined by 
\[ \Gcell (t) = \iota(t) + \min\{ t-a_{\iota(t)},1\} \quad \mbox{and} \quad \Hcell (t)=a_{\lfloor t\rfloor}+(a_{\lfloor t\rfloor+1}-a_{\lfloor t\rfloor})(t-\lfloor t\rfloor)\]
are homotopy inverses to each other. The explicit homotopy $\Homcell \colon X_N\times [0,1]\to X_N$ between $1$ and $\Hcell \circ \Gcell$ is given by
\[ \Homcell (t,s)=\min\{a_{\iota (t)}+(t-a_{\iota (t)})(1+s(a_{\iota(t)+1}-a_{\iota(t)}-1)) , a_{\iota(t)+1}\}.\]
One easily verifies that the $\ell^2$-operator norms of the induced chain maps satisfy 
$$\|\Ghom \|, \ \|\Hhom \|, \ \|\Homhom \|=O( T+1) \text{ and } \|\p\|, \|\p'\|\leq 2.$$
The number of cells in $X_m$ is 
$m\in [N/T,2N/T]$
in dimension $0$ and $1$. Hence, $\mathbb Z$ has the cheap $\indice$-rebuilding property for any $\indice \in\mathbb N$.
\end{proof}

As particular cases of Theorem~\ref{thm-AdjRebStack} we obtain:
\begin{example}[Graphs of groups]\label{ex: graphs of groups}
Let $\Gamma$ be a residually finite group that splits as a finite graph of groups with edge and vertex stabilizers that satisfy the cheap $(\indice-1)$-rebuilding property (and cheap $\indice$-rebuilding property respectively) for every $\indice$. 
Considering its Bass-Serre tree $\Baseup$, we obtain that $\Gamma$ itself has the cheap $\indice$-rebuilding property.
This applies for instance to the residually finite Baumslag-Solitar groups $\mathrm{BS}(1,n)$ and $\mathrm{BS}(n,n)$, for any non-zero integer $n$.
\end{example}

\begin{example}[Groups acting on graphs] \label{Ex:T10.10}
If a residually finite group $\Gamma$ acts co-compactly on a \emph{connected} graph $\mathcal{G}$ such that 
\begin{itemize}
\item vertex stabilizers have the cheap $1$-rebuilding property; 
\item edge stabilizers are infinite, 
\end{itemize}
then $\Gamma$ has the cheap $1$-rebuilding property.
\end{example}

\subsection{Further applications}

Recall that a group $\Gamma$ is called \defin{polycyclic} if there is a sequence of subgroups 
$$\Gamma=A_0\triangleright A_1\triangleright \cdots \triangleright A_n=\{1\}$$ 
such that each quotient group $A_i/A_{i+1}$ is cyclic. As a corollary of Theorem~\ref{thm-AdjRebStack} and Lemma \ref{lem-ZRebuilding} we obtain

\begin{corollary}\label{cor-CheapRebGroups} Let $\Gamma$ be a residually finite countable group. The following holds:
\begin{enumerate}

\item \label{cor-stab under commens}  Let $\Gamma'\leq \Gamma$ be a finite index subgroup. Then $\Gamma$ has the cheap $\indice$-rebuilding property if and only if $\Gamma'$ does. 

\item \label{it: cheap k-reb. normal subgroup and F k+1 quotient} If $\Gamma$ has an infinite normal subgroup $N$ such that $\Gamma/N$ is of type $F_{\indice}$ and $N$ has the cheap $\indice$-rebuilding property, then $\Gamma$ has the cheap $\indice$-rebuilding property.

\item $\Z^m$ has the cheap $\indice$-rebuilding property for every $\indice$.

\item Infinite polycyclic groups have the cheap $\indice$-rebuilding property for every $\indice$.
\end{enumerate}
\end{corollary}
\begin{proof} 1. If $\Gamma$ has the cheap $\indice$-rebuilding property, then the witnesses $(X, {\kappa}, U(T))$ for $\Gamma$ yield witnesses for $\Gamma'$.
Conversely, if $\Gamma'$ has the cheap $\indice$-rebuilding property one can assume, up to passing to a further finite index subgroup, that it is normal in $\Gamma$. Then pick any $(\indice-1)$-connected CW-complex $\Baseup$ with a free action of $\Gamma/\Gamma'$ that is finite in every dimension $\leq \indice$ and apply Theorem~\ref{thm-AdjRebStack}
to the $\Gamma$-action on $\Baseup$ defined through the quotient map $\Gamma\to \Gamma/\Gamma'$.

2. Let $\basezero$ be a classifying space (CW-complex) $\basezero$ for $\Gamma / N$ with finite $\indice$-skeleton.
The group $\Gamma$ acts on its universal cover $\widetilde{\basezero}$ with cell-stabilizers all equal to $N$ and one can apply Theorem~\ref{thm-AdjRebStack} to this action. 

3. and 4. finally follow from Lemma \ref{lem-ZRebuilding} and 2. by induction. 
\end{proof}

\begin{example}
Corollary~\ref{cor-CheapRebGroups}, item~\eqref{it: cheap k-reb. normal subgroup and F k+1 quotient} implies for instance that $\SL_d(\Z)\ltimes \Z^d$ has the cheap $\indice$-rebuilding property for every $\indice$. The same is true for the standard braid groups since they have an infinite cyclic center.
\end{example}

\subsection{Chain-commuting groups}
The following proposition connects the cheap $1$-rebuilding property to the rewirings of chain-commuting (a.k.a. right-angled) groups considered in \cite{AGN}.

\begin{proposition} \label{P:rag}
Let $\Gamma$ be a residually finite group that is \defin{chain-commuting}, i.e., that admits a finite generating list $\{ \gamma_1 , \ldots , \gamma_m \}$ of elements of infinite order such that $[\gamma_i , \gamma_{i+1} ] = 1$ for all $i \in \{1 , \ldots , m-1 \}$. Then $\Gamma$ has the cheap $1$-rebuilding property. 
\end{proposition}
Nota: Our proof below is incomplete and necessitate the additional assumption that the normalizers of the subgroups $H_j$ are finitely generated, which is not granted in general.
A complete proof of this proposition can be found in an article of Matthias Uschold \cite{Usc22}.
\begin{proof}[Proof of Proposition~\ref{P:rag}] According to Example~\ref{Ex:T10.10} it suffices to produce a cocompact action $\Gamma \acting \mathcal{G}$ on a \emph{connected} graph such that 
\begin{itemize}
\item vertex stabilizers have the $1$-rebuilding property, and
\item edge stabilizers are infinite. 
\end{itemize} 

We build the graph $\mathcal{G}$ as follows: for each $i \in \{1 , \ldots , m \}$, let $H_i$ be the (free abelian) infinite subgroup of $\Gamma$ generated by $\gamma_i$. The vertex set of $\mathcal{G}$ consists of all the conjugates of the $H_i$ in $\Gamma$, and two distinct conjugates $gH_ig^{-1}$ and $hH_jh^{-1}$ represent an edge if and only if the subgroup they generate is abelian. 

The conjugation action of $\Gamma$ on the set of its subgroups induces an action of $\Gamma$ on $\mathcal{G}$. Observe that, by hypothesis, the vertices corresponding to $H_1, \ldots , H_m$ belong to the same connected component of $\mathcal{G}$. It follows that for every $g \in \Gamma$ the vertices $gH_1g^{-1}, \ldots , gH_m g^{-1}$ also belong to one connected component of $\mathcal{G}$. Now for any $j \in \{ 1 , \ldots , m \}$ we have 
$$g H_j g^{-1} = (g\gamma_j^{-1}) H_j (g \gamma_j^{-1})^{-1},$$
and, since $\{ \gamma_1 , \ldots , \gamma_m \}$ is a generating set of $\Gamma$, we conclude that the graph $\mathcal{G}$ is connected. 

The stabilizer of each vertex, or each edge, contains a normal subgroup that is a free infinite abelian subgroup (a conjugate of some $H_j$). These stabilizers therefore have the cheap $1$-rebuilding property (by Corollary \ref{cor-CheapRebGroups}(2) and Lemma \ref{lem-ZRebuilding}). 

Finally, the quotient $\Gamma \backslash \mathcal{G}$ is a finite graph (on $m$ vertices).Theorem~\ref{thm-AdjRebStack} applies to the action of $\Gamma$ on $\mathcal{G}$ and we conclude that $\Gamma$ itself has the cheap $1$-rebuilding property.
\end{proof}

The cheap $\indice$-rebuilding property goes beyond chain-commuting groups, even for $\indice =1$. We now give an example.

\begin{example}[Non chain-commuting examples $\Lambda\ltimes \Z^2$]
\label{ex: non rag semidirect F x Z2}
If a free group $\Lambda\leq \SL_2(\Z)$ does not contain any non-trivial unipotent element, 
then any embedding of $\Z^2$ in the canonical semi-direct product $\Lambda\ltimes \Z^2$ lies inside the obvious normal $\Z^2$:
If $(\lambda_1,a_2), (\lambda_2, a_2)\in \Lambda\ltimes \Z^2$ commute (and generate $\Z^2$), then up to considering some power of them, one can assume 
$\lambda_1=\lambda_2$. It follows that $\lambda_1^{-1} a_2^{-1} a_1 \lambda_1=a_2^{-1} a_1$, so that $\lambda_1$ is unipotent, thus trivial.

In order to find such a subgroup $\Lambda$ of $\SL(2,\Z)$, it's enough to consider a finite cover of the modular orbifold $\SL_2(\Z) \backslash \HH^2$  that is a surface $S$ of genus $g\geq 2$ (with cusps). The fundamental group of the compact surface $S'$ obtained by adding one point to each cusp of $S$ surjects onto the free group $\FF_2$. Finally pick a  pull-back  of this free group along the maps $\SL_2(\Z)\geq \pi_1(S) \twoheadrightarrow \pi_1(S')\twoheadrightarrow \FF_2$. It does not contain any non-trivial unipotent element.

Such a semi-direct product $\Gamma=\Lambda\ltimes \Z^2$ is thus not chain-commuting while it has the cheap $\indice$-rebuilding property for all~$\indice$ by
Corollary~\ref{cor-CheapRebGroups}~\eqref{it: cheap k-reb. normal subgroup and F k+1 quotient}.
\end{example}

\subsection{Artin groups}

Let $I$ be a finite set. A \defin{Coxeter matrix} $M=(m_{ij})$ on $I$ is an $I$-by-$I$ symmetric matrix $M=(m_{ij})$ with entries in $\mathbb{N} \cup \{ \infty \}$ such that $m_{ii}=1$ for all $i$ and such that whenever $i \neq j$, $m_{ij}\geq 2$. Associated to $M$ are 
\begin{itemize}
\item a \defin{Coxeter group} $W_I$, or simply $W$, given by the presentation 
$$W = \langle s_i, \  i \in I \; | \; (s_i s_j )^{m_{ij}} = 1, \ (i,j) \in I^2 \rangle$$
with the convention that the relation is ignored if $m_{ij} = \infty$, 
\item a \defin{Artin group} $A_I$, or simply $A$, given by the presentation 
$$A = \langle a_i, \ i \in I \; | \; \underbrace{a_i a_j a_i a_j a_i \cdots}_{m_{ij}} = \underbrace{a_j a_i a_j a_i a_j \cdots}_{m_{ij}}, \ (i,j) \in I^2 \rangle, \quad \mbox{and}$$
\item a simplicial complex, called the \defin{nerve} of $M$, that we define below.
\end{itemize} 
If the corresponding Coxeter group $W$ is finite we say that $A$ is \defin{spherical}. These Artin groups are sometimes also referred to as `Artin groups of finite type' in the literature.

If $J \subseteq I$, let $M_J$ denote the minor of $M$ whose rows and columns are indexed by $J$, and let $W_J$, resp. $A_J$, be the corresponding Coxeter group, resp. Artin group. It is known \cite{Bourbaki} that the natural map 
$W_J \to W_I$, resp. $A_J \to A_I$, is injective and hence $W_J$, resp. $A_J$, can be identified with the subgroup of $W_I$, resp. $A_I$, generated by $\{ s_i \; : \; i \in J \}$, resp. $\{ a_i \; : \; i \in J \}$.

The \defin{nerve} of $M$ is the simplicial complex $L$ whose vertex set is $I$ and a subset $\sigma \subseteq I$ spans a simplex if and only if $W_\sigma$ is finite. 

Charney and Davis \cite[Section 3]{CharneyDavis} have associated to $M$ a simple complex of groups $\mathcal{C}$, in the sense of \cite[Definition 12.11]{BridsonHaefliger}. The construction goes as follows: first consider the partially ordered set (poset) 
$$\mathcal{P} = \{ J \subseteq I \; : \; W_J \mbox{ is finite} \}$$
ordered by inclusion. To any $\sigma \in \mathcal{P}$ one associates the Artin group $A_\sigma$, and for each $\tau \leq \sigma$ the associate homomorphism $A_\tau \to A_\sigma$ is the natural inclusion. The geometric realization of $\mathcal{C}$ is simply connected (in fact contractible), since the empty set is an initial object for the poset $\mathcal{P}$. In particular, the fundamental group of the complex of groups $\mathcal{C}$ is 
$A$ (the direct limit of the $A_\sigma$).
The complex of groups $\mathcal{C}$ is developable, that is, it arises from the action of a group on a simplicial complex. More precisely, let 
$$A \mathcal{P} = \{ g A_J \; : \; g \in A , \ J \in \mathcal{P} \},$$
ordered by inclusion. The stabilizer of $gA_j \in A \mathcal{P}$ is the subgroup $gA_J g^{-1}$; hence, $A$ acts without inversion on $A \mathcal{P}$. Let 
$$X = | A \mathcal{P} '|$$ 
be the geometric realization of the derived poset of $A \mathcal{P}$; it is the simplicial complex, referred to as the  `modified Deligne complex' in \cite{CharneyDavis}, where $k$-simplices correspond to totally ordered chains of elements in $A\mathcal{P}$. The complex of groups $\mathcal{C}$ explicitely arises from the action of $A$ on $X$.

 It is conjectured that $X$ is contractible and this conjecture is equivalent to the celebrated \defin{$K(\pi,1)$ conjecture} for Artin groups. The main theorem of 
\cite{CharneyDavis} is that this conjecture holds when the nerve $L$ is a \defin{flag complex},  i.e., when the Coxeter group $W_J$ associated to every clique $J$ is finite.
This holds for instance if $A$ is either spherical or a right-angled Artin group. The conjecture also holds whenever  $L$ has dimension  at most $1$.  Recently, Paolini and Salvetti proved that the $K(\pi,1)$ conjecture holds for all affine Artin Groups \cite{Paolini-Salvetti-2021}.

\begin{theorem}[Artin groups with $(\indice-1)$-connected nerve]
\label{Ex:Artin groups}
Let $M$ be a Coxeter matrix whose nerve $L$ is $(\indice-1)$-connected. Suppose that the $K(\pi,1)$-conjecture holds for the associated Artin group $A$ and that $A$ is residually finite. Then $A$ has the cheap $\indice$-rebuilding property.
\end{theorem}

\begin{proof} We first consider the case where $A$ is spherical, or equivalently where $L$ is a simplex. In \cite[Main Theorem]{Bestvina}, Bestvina proves that $A$ is then commensurable with $G \times \Z$, where $G=A/\langle \Delta^2 \rangle$ is the quotient of $A$ by some central element $\Delta^2$, and moreover $G$ acts cocompactly on a contractible simplicial complex $X$ and the action is transitive on the vertices with stabilizers $\Z/2\Z$. It follows that $G$ is of type $F_\indice$ for every $\indice$ (see e.g., \cite[Theorem 7.3.1]{Geoghegan}). Applying Corollary~\ref{cor-CheapRebGroups}~\eqref{it: cheap k-reb. normal subgroup and F k+1 quotient} we get that 
 $G \times \Z$ has the cheap $\indice$-rebuilding property for every $\indice$. By Corollary~\ref{cor-CheapRebGroups}~\eqref{cor-stab under commens}, so has $A$. 
 
Consider now the case where $A$ is of general type. We would like to apply our Theorem~\ref{thm-AdjRebStack} to the action of $A$ on the modified Deligne complex $X$; but the stabilizers of the vertices of the form $A_J$ for $J=\emptyset$ are trivial.
We thus modify $X$ as follows:
Let $\Baseup \subseteq X$ be the sub-complex associated to the sub-poset $\mathcal{P}_0$ of $\mathcal{P}$ consisting of the non-empty subsets $J \subseteq I$. The $A$-action on $X$ preserves $\Baseup$ and the cell stabilizers are now isomorphic to the non trivial spherical Artin groups $A_J$, $J\not=\emptyset$. The complement $X \setminus\Baseup$ is the disjoint union of the open stars $St(v)$ of the vertices $v$ of $X$  corresponding to the elements $g A_\emptyset \in A \mathcal{P}$ ($g \in A$). 
Each such star $St(v)$  is the cone over the link of $v$ in $X$ and this link is isomorphic to the simplicial complex $|\mathcal{P}_0'|$ associated to the poset of non-empty spherical subsets of $I$, i.e., the nerve $L$ of $M$. By hypothesis the latter is $(\indice-1)$-connected and $X$ is contractible, it therefore follows from Mayer-Vietoris and Hurewicz Theorems (and van Kampen Theorem in degree $1$) that $\Baseup$ is $(\indice-1)$-connected.   Then a direct application of Theorem~\ref{thm-AdjRebStack} shows that the Artin group $A$ has the cheap $\indice$-rebuilding property as soon as it is residually finite.
\end{proof}

\begin{example}
The direct product $\Gamma=\FF_2\times \FF_2\times \cdots \times \FF_2$ of $k$ copies of the free group on $2$ generators has the cheap $(k-1)$-rebuilding property.
This follows from Theorem~\ref{Ex:Artin groups} when $k\geq 2$.
The nerve of these right-angled Artin groups has the homotopy type of a $(k-1)$-dimensional sphere.

\end{example}

\begin{remark}
In fact, to obtain Theorem \ref{Ex:Artin groups}, instead of the $K(\pi,1)$ conjecture, it's enough to assume that the modified Deligne complex $X$ is $(\indice-1)$-connected: this is the condition in the proof for $\Baseup$ to be $(\indice-1)$-connected.

\end{remark}

\subsection{Homology growth}

\begin{theorem}\label{prop-CheapRebTorsion}
Let $\Gamma$ be a finitely presented residually finite group of type $F_{\indice+1}$ that has the cheap $\indice$-rebuilding property for some non-negative integer $\indice$. Then, for every Farber sequence $(\Gamma_n)_{n \in \mathbb N}$, coefficient field $K$ and $0\leq j \leq \indice$ we have 
\[ \lim_{n\to \infty} \frac{\dim_K H_j (\Gamma_n , K)}{[\Gamma : \Gamma_n]} =0 \quad \mbox{and} \quad \lim_{n\to \infty} \frac{\log | H_j (\Gamma_n,\mathbb Z)_{\rm tors}|}{[\Gamma:\Gamma_n]}=0.\]
\end{theorem}
\begin{proof}
Let $X$ be a $K(\Gamma , 1)$ CW-complex with finite $\indice$-skeleton. Let $T\geq 1$. Then there exists a $\Gamma$-Farber neighborhood $U\subseteq \Sub^{\findex}_\Gamma$ such that for any finite index subgroup $\Gamma_1 \in U$ there exists an $\indice$-rebuilding $(X_1 = \Gamma_1 \bs \tilde X, X_1 ')$ of quality $(T,O(1)).$ In degree $\leq \indice$, the complex 
$X'_1$ has $O([\Gamma:\Gamma_1 ]/T)$ cells of this dimension and the norms of the boundary maps are $O(T+1).$ It follows from \eqref{E:bettiK} that for all $j \leq \indice$ we have 
$$\dim_K H_j (\Gamma_n , K) = O ( [\Gamma: \Gamma_1 ]T^{-1} ).$$
Letting $T\to\infty$ this proves the first part of the theorem. By Proposition \ref{Prop: Gabber} we similarly get 
\[ \log | H_j (\Gamma_1,\mathbb Z)_{\rm tors}| \leq O( [\Gamma:
\Gamma_1 ]T^{-1}\log(T+1)), \] 
for $j=0,\ldots, \indice-1$. 
Letting $T\to\infty$ this proves the second part of the theorem for all $0\leq j \leq \indice-1$. 

To handle the remaining case $j=\indice$, we take advantage of the fact that in Proposition \ref{Prop: Gabber} the only information needed from dimension $\indice +1$ is the norm of the boundary maps $\p_{\indice +1}$ (but not the number of $(\indice+1)$-cells) and that the norm of the boundary maps is bounded when taking coverings. We proceed exactly as in the end of the proof of Theorem \ref{Tmain}: by adding finitely many  $(\indice +1)$-cells to $X$, we can make it $\indice$-aspherical (see 
\cite[Theorem 8.2.1]{Geoghegan}). Write $\Basedown=\sqcup_{I} \mathbb B^{\indice +1}$ for this collection of $(\indice +1)$-cells and let $f\colon \partial \Basedown=\sqcup_{I} \mathbb S^{\indice}\to X^{(\indice)}$ be the map that attaches the $(\indice +1)$-cells to $X$. Then $X^+:=X\sqcup_f \Basedown$ is a finite $\indice$-aspherical 
CW-complex with fundamental group $\Gamma$.

Let $T,U$ be as above and let $\Gamma_1\in U$ be a finite index subgroup. Let $X_1 =\Gamma_1 \bs \tilde X$ and let $(X_1 ,X_1' ,\Gcell , \Hcell ,\Homcell )$ be an $\indice$-rebuilding of quality $(T,O(1))$ given by the rebuilding property of $\Gamma$.
We also consider  $X_1^+ =\Gamma_1 \bs \widetilde{X^+}$. The CW-complex $X_1^+$ can be written as  $X_1^+ =X_1 \sqcup_{f_1} \Basedown_1$ where $\Basedown_1 = \sqcup_{I_1} \mathbb B^{\indice+1}$ is the preimage of $\Basedown$ in $X_1^+$ and $f_1 \colon \partial \Basedown_1 =\sqcup_{I_1} \mathbb S^{\indice}\to X_1^{(\indice)}$ is the lift of $f$.  We have a diagram 
\begin{equation*}
\begin{tikzcd}[sep=large]
\Basedown_1 \arrow[d,"\mathrm{id}",swap, shift right] &\arrow[l,hook] \partial \Basedown_1 \arrow[d,swap,"\mathrm{id}",shift right] \arrow[r,"f_1"] & X_1 \arrow[d,swap,"\Gcell", shift right]\\
\Basedown_1 \arrow[u,"\mathrm{id}",swap, shift right]&\arrow[l,hook] \partial \Basedown_1 \arrow[r,"\varphi"] \arrow[u,"\mathrm{id}",swap, shift right] & X_1' \arrow[u,"\Hcell",swap, shift right],
\end{tikzcd}
\end{equation*} 
with $\varphi =\Gcell \circ f_1$. By Proposition \ref{lem-QGluing} the space $X_1 ' \sqcup_{\varphi} \Basedown_1$ is homotopy equivalent to 
$$X_1 \sqcup_{f_1} \Basedown_1 = X_1^+ .$$ 
Therefore, $H_{\indice}(\Gamma_1 ,\mathbb Z)=H_{\indice}(X_1' \sqcup_{\varphi} \Basedown_1,\mathbb Z ).$  Since $f_1$ is a lift of $f$, its norm satisfies $\|f_1 \|\leq \|f\|=O(1)$. The norm of the boundary map on $X_1' \sqcup_{\varphi} \Basedown_1$ in degree $\indice+1$ is bounded by $\|\varphi\|\leq \|f_1 \|\|\Ghom \|=O(T+1)$.  By Proposition \ref{Prop: Gabber} we have 
\[ \log | H_{\indice}(X_1' \sqcup_{\varphi} \Basedown_1 ,\mathbb Z)_{\rm tors}|\leq \|\partial_{\indice+1}\|_{(X_1' \sqcup_{\varphi} \Basedown_1 )}
\times \card{\indice}(X_1' \sqcup_{\varphi} \Basedown_1 )\leq O( \log(T+1) [\Gamma:\Gamma_1 ]T^{-1}).\]
Since any Farber sequence eventually falls into $U$, we get the theorem by letting $T\to \infty$. 
\end{proof}

We end this section with a question.
\begin{question}\label{quest: F k+1 amenable have cheap k-reb. prop.?}
Does every residually finite amenable group of type $F_{\indice}$ have the cheap $\indice$-rebuilding property? 
\end{question}
In this context, Kar, Kropholler and Nikolov \cite[Th. 1]{KKN-2017} prove the vanishing of the $j$-th torsion growth for $j\leq \indice$.
The arguments of \cite{Sauer-Vol-Hom-growth-2016}, while stated in a restricted framework also give this vanishing.

\section{Lattices in semi-simple Lie groups (Proof of Theorem \ref{T0})} \label{S10}
\label{sect: princ. cong. in semi-simple}

Let $\mathbf G$ be an affine algebraic group defined over $\Q$. The \defin{radical} of $\mathbf G$ is the greatest connected normal solvable subgroup. The group $\mathbf G$ is called \defin{semi-simple} if its radical is trivial. In the following, we assume that $\mathbf G$ is connected and semi-simple. 

An algebraic group $\mathbf T$ over $\Q$ is an \defin{algebraic torus} if the group of its complex points $\mathbf{T} (\C)$ is isomorphic to a product of $\GL_1 (\C )$. If $\mathbf T$ is isomorphic to a product of $\GL_1$ themselves over $\Q$, then $\mathbf T$ is said to be \defin{split} over $\Q$. 
The maximal dimension of a torus in $\mathbf G$ that is split over $\Q$ is called the \defin{rational rank} of $\mathbf G$. 

A Zariski-closed subgroup $\mathbf P$ of $\mathbf G$ is called \defin{parabolic} if $\mathbf P$ contains a connected solvable subgroup, i.e., a \defin{Borel subgroup} of $\mathbf G$. If $\mathbf P$ is defined over $\Q$, then $\mathbf P$ is called a rational parabolic subgroup.  The \defin{unipotent radical} of $\mathbf P$ is defined to be its largest normal subgroup consisting entirely of unipotent elements.

A $\Z$-structure of $\mathbf G$ is given by a faithful $\Q$-embedding of $\mathbf G$ in some $\GL_n$. We define $\Gamma = \mathbf{G}_\Z$ to be the intersection of $\mathbf G$ with $\GL_n (\Z)$ and the level $N$ \defin{principal congruence subgroup} $\Gamma (N)$ to be the intersection of $\mathbf G$ with the kernel of $\GL_n (\Z) \to \GL_n (\Z / N \Z)$. Here $N$ is some positive integer.

\begin{theorem}\label{TorsionSL}
Let $\mathbf G$ be a connected semi-simple affine algebraic group defined over $\Q$ equipped with a $\Z$-structure. Let $r\geq 2$ be the rational rank of $\mathbf G$. Then there exists a constant ${\kappa}={\kappa}(\mathbf{G})$ such that for every \defin{principal} congruence subgroup $\Gamma(N) \leq \mathbf{G}_\Z$, every $j\leq r-1$ and every coefficient field $K$, we have 
\begin{equation}
\dim_K  H_j (\Gamma (N), K) \leq  {\kappa}  N^{(1- \delta) \dim \mathbf{G}} \quad \mbox{and} \quad \log |H_j (\Gamma(N),\Z)_{\rm tors}| \leq {\kappa} N^{(1-\delta ) \dim \mathbf{G}} \log N ,
\end{equation}
where 
$$\delta = \mathrm{min} \left\{ \frac{\dim \mathbf{U}}{\dim \mathbf{G}} \; : \; \mathbf{U} \mbox{ is the unipotent radical of a parabolic subgroup of } \mathbf{G} \right\} .$$    
\end{theorem}
 The lattices $\Gamma(N)$ are non-uniform if and only if the rational rank $r$ of $\mathbf G$ is positive.  
 In other words, Theorem~\ref{TorsionSL} is empty for uniform lattices.
Our theorem vacuously holds for $r=0$ and $r=1$.

\begin{remark} 
\label{rem: case SL(d), computation of delta}
In case $\mathbf{G} = \SL_d$ we have $r=d-1$ and $\dim \mathbf{G} = d^2-1$. Any parabolic subgroup of $\mathbf{G}$ is conjugated to a subgroup of block upper triangular matrices, and the unipotent radical of such is the subgroup where the block diagonal elements are the identity. It follows that
$$\delta = \frac{d-1}{d^2-1} = \frac{1}{d+1}.$$
Since there exists a universal constant $C$ such that \[1>\frac{[\SL_d (\Z):\Gamma(N)]}{N^{d^2-1}} >C>0,\]
Theorem \ref{TorsionSL} implies Theorem \ref{T0}, as well as the congruence case of Theorem \ref{T000}, of the introduction.
\end{remark}

\begin{proof}[Proof of Theorem~\ref{TorsionSL}]
We first recall the spherical Tits building $\Delta_\Q (\mathbf{G})$ associated with $\mathbf G$ over $\Q$ (see \cite{Tits1,Tits2}); it is a simplicial set whose non-degenerate simplices are in bijection with proper rational parabolic subgroups of $\mathbf G$. Each proper maximal rational parabolic subgroup corresponds to a vertex of $\Delta_\Q (\mathbf{G})$ and $k+1$ proper parabolic subgroups $\mathbf{P}_0 , \ldots , \mathbf{P}_k$ are the vertices of a $k$-simplex if and only if their intersection $\mathbf{P}_0 \cap \ldots \cap \mathbf{P}_k$ is a rational parabolic subgroup, and this simplex corresponds to the parabolic subgroup $\mathbf{P}_0 \cap \ldots \cap \mathbf{P}_k$.

If the rational rank $r$ of $\mathbf G$ is equal to $1$, then $\Delta_\Q (\mathbf{G})$ is a countable collection of points. Otherwise $\Delta_\Q (\mathbf{G})$ is a spherical building (see \cite[Theorem 5.2]{Tits2}). For any maximal $\Q$-split torus $\mathbf T$, all the rational parabolic subgroups containing $\mathbf T$ form an apartment in this building and each of these sub-complexes give a simplicial triangulation of the sphere of dimension $\indice=r-1$ so that $\Delta_\Q (\mathbf{G})$ has the homotopy type of a bouquet of $\indice$-spheres (Solomon--Tits theorem \cite{Solomon-1969}); in particular it is $(\indice-1)$-connected. We refer to \cite[\S V.5]{BrownBuildings} for this and more about spherical buildings. 

The rational points $\mathbf{G} (\Q)$ of $\mathbf G$ act on the set of rational parabolic subgroups by conjugation and hence on $\Delta_\Q (\mathbf{G})$: for any $g \in \mathbf{G} (\Q)$ and any rational parabolic $\mathbf{P}$, the simplex associated to $\mathbf P$ is mapped to the simplex associated to $g\mathbf{P} g^{-1}$. Let $\Gamma \subseteq \mathbf{G} (\Q)$ be an arithmetic subgroup of $\mathbf{G} (\Q)$. Then by reduction theory (see e.g. \cite[Theorem 4.15]{PlatonovRapinchuk}) there are only finitely many $\Gamma$-conjugacy classes or rational parabolic subgroups. Therefore, the quotient $\Gamma \backslash \Delta_\Q (\mathbf{G})$ is a finite simplicial complex. 

From now on we let $\Gamma = \mathbf{G}_\Z$, we fix some principal congruence subgroup $\Gamma (N)$ and let $\Baseup$ denote the rational Tits building $\Delta_\Q (\mathbf{G})$. Up to replacing $\Baseup$ by its barycentric subdivision we may furthermore assume that for every cell $\sigma \subseteq \Baseup$ the stabilizer $\Gamma_\sigma$ acts trivially on $\sigma$. 

The stabilizer $\Gamma_\sigma$ of a simplex $\sigma \subseteq \Baseup$ associated to a rational parabolic subgroup $\mathbf P$ contains the intersection $\mathbf{P}_\Z$ of $\mathbf P$ with $\GL_d (\Z)$. In particular $\Gamma_\sigma$ contains the $\Z$-points $\mathbf{U} (\Z)$ of the unipotent radical  $\mathbf{U}$ of $\mathbf P$ as a normal subgroup. The group $\mathbf{U} (\Z)$ is finitely generated, torsion-free, nilpotent and its intersection with the stabilizer $\Gamma(N)_\sigma$ is equal to 
$$\mathrm{ker} (\mathbf{U} (\Z) \to \mathbf{U} (\Z /N \Z))$$
which is of index at least $N^{ \delta \dim \mathbf{G}}$ in $\mathbf{U} (\Z)$. On the other hand it is well known that
$$[\Gamma : \Gamma(N) ] \leq O( N^{\dim \mathbf{G}}),$$
see e.g. \cite[bottom of page 3134]{BelolipetskyLubotsky}. 
It therefore follows from Theorem \ref{Tmain} (with $\Gamma_1=\Gamma(N)$)
 that there exists 
 a constant $c=c(\mathbf{G})$ that depends on $\mathbf{G}$ but not on $\Gamma(N)$ and a finite CW-complex $\Totaldown_N^+$
 such that the following properties hold. 
\begin{enumerate}
\item The CW-complex $\Totaldown_N^+$ has an $\indice$-connected universal cover.
\item In each dimension $\leq \indice$, the total number of cells of $\Totaldown_N^+$ is bounded by 
\begin{equation}\label{eq-NbOfCelles}
c  N^{(1- \delta) \dim \mathbf{G}} .
\end{equation}
\item In each degree $\leq \indice +1$ the norm of the boundary operator on the chain complex $C_\bullet (\Totaldown_N^+ )$ is bounded by 
\begin{equation}
c N^{c} .
\end{equation}
\end{enumerate}
We conclude by applying \eqref{E:bettiK} and Proposition \ref{Prop: Gabber}. Let ${\kappa}= c^2 \log c$. Then, in each degree $j \leq \indice$ we have both
$$\dim_K H_j (\Gamma (N), K)  \leq {\kappa} N^{(1- \delta) \dim \mathbf{G}} \quad \mbox{and} \quad \log |H_j (\Gamma(N) )_{\rm tors} | \leq {\kappa} N^{(1- \delta) \dim \mathbf{G}} \log N .$$
\end{proof}

We used the fact that the stabilizers of cells in the action of $\Gamma$ on $\Baseup$ have infinite normal unipotent subgroups. Hence, by Lemma \ref{cor-CheapRebGroups} they all have the cheap $\indice$-rebuilding property, for every $\indice$. Upon applying Theorem~\ref{thm-AdjRebStack} to the action of $\Gamma$ on the rational Tits building $\Baseup$ we get
\begin{theorem}\label{thm-LatticesRebuilding}
Let $\mathbf G$ be a connected semi-simple affine algebraic group defined over $\Q$ equipped with a $\Z$-structure. Let $r$ be the rational rank of $\mathbf G$. Then $\mathbf{G}_\Z$ has the cheap $(r-1)$-rebuilding property. 
\end{theorem}
Theorem \ref{prop-CheapRebTorsion} now yields.
\begin{theorem}\label{thm-LatticesFarberTors}
Let $\mathbf G$ be a connected semi-simple affine algebraic group defined over $\Q$ equipped with a $\Z$-structure. Let $r$ be the rational rank of $\mathbf G$. Then for any Farber sequence $(\Gamma_n)_{n\in \mathbb N}$ in $\mathbf{G}_\Z $ and for any coefficient field $K$ we have 
\[ \lim_{n\to\infty} \frac{\dim_K H_j (\Gamma_n , K)}{[\Gamma : \Gamma_n]} = 0 \quad \mbox{and} \quad \lim_{n\to\infty} \frac{\log | H_j(\Gamma_n,\mathbb Z)_{\rm tors}|}{[\mathbf{G}_\Z :\Gamma_n]}=0,\] 
for $j=0,\ldots, r-1$.
\end{theorem}
This implies Theorem \ref{T00} and the first part of Theorem \ref{T000} of the introduction. 

\section{Application to mapping class groups (Proof of Theorem \ref{T0000})} \label{Sect: Application to mapping class groups}

Let ${S}$ be a closed orientable surface of genus $g$ with $b$ connected boundary components. We assume that 
$$\chi ({S} ) = 2-2g - b <0  \quad \mbox{and} \quad b \geq 4 \mbox{ if } g=0.$$  
We write $\MCG({S})$ for the mapping class group of ${S}$. Let us recall the construction of the curve complex $\mathcal C({S})$. 

The curve complex $\mathcal C({S})$ is a combinatorial cell complex whose $k$-cells consist of collections of 
$k+1$ simple closed curves on ${S}$ which are disjoint, essential, and pairwise non-homotopic, 
considered up to homotopy. In particular, a $0$-dimensional cell corresponds to the homotopy class of a simple closed curve and top dimensional cells correspond to maximal families of pairwise non-homotopic essential closed curves in ${S}$. 
It is equipped with a co-compact $\MCG({S})$-action.

Under our hypotheses ${S}$ admits a hyperbolic structure and the size of any maximal family of pairwise non-homotopic essential closed curves in ${S}$ is finite, equal to $3g-3+b$. So the dimension of $\mathcal C({S})$ is $3g-4+b$.

Now let $\mathcal F$ be a finite collection of $m$ closed, disjoint, essential, pairwise non-homotopic simple closed curves on ${S}$. The stabilizer of the homotopy class of $\mathcal F$ in $\MCG({S})$, denoted $\Stab \mathcal F$, fits into a short exact sequence: 
\[\begin{tikzcd}
1 \arrow[r]& \mathbb Z^m \arrow[r]& \Stab \mathcal F\arrow[r]& \MCG'({S} \setminus \mathcal F)\arrow[r]& 1, 
\end{tikzcd}\] where $\MCG'({S}\setminus \mathcal F)$ is the subgroup of those mapping classes in $\MCG({S}\setminus \mathcal F)$ that
permute the boundary in such a way that it can still be glued back to obtain $\mathcal{F}$ in ${S}$
 ; see e.g. \cite[Equation (1)]{Minsky}. 
 The group $\MCG'({S}\setminus \mathcal F)$ has finite index in $\MCG({S}\setminus \mathcal F)$, which is of class $F_k$ for every $k\geq 0$
 by \cite[Theorem 5.4.A]{Ivanov}. 
Hence the stabilizers $ \Stab \mathcal F$ are of class $F_k$  for every $k\geq 0$.

In order to ensure the assumption that the cell-stabilizers act freely on them, we consider instead the action of $\Gamma= \MCG({S})$ on the barycentric subdivision $\Baseup$ of $\mathcal C({S})$. Now the stabilizer of any cell $\sigma$ of $\Baseup$ is a finite index subgroup of some $\Stab \mathcal F$ as above.
Therefore $\Gamma_\sigma$ is of class $F_\indice$  for every $\indice \geq 0$ and it has a finite rank normal free abelian subgroup,
  so by Corollary~\ref{cor-CheapRebGroups} it has the cheap $\indice$-rebuilding property for every $\indice$. 
  
 The homotopy type of $\mathcal C({S})$ (and therefore that of $\Baseup$) was identified by Harer \cite[Thm 3.5]{Harer}; it is homotopy equivalent to a bouquet of spheres of dimension $\indice (g,b)$ where 
$$\indice (g,b):=\left\{ \begin{array}{ll}
2g-2 & \mbox{ if } g\geq 2 \mbox{ and } b=0; \\
2g-3+b & \mbox{ if } g \geq 1 \mbox{ and } b \geq 1; \\ 
b-4 &  \mbox{ if } g=0  \mbox{ and } b\geq 4.
\end{array} \right.$$
Thus, the space $\Baseup$ is $(\indice(g,b)-1)$-connected.
By Theorem~\ref{thm-AdjRebStack}, we obtain:
\begin{theorem}\label{th: MCG has CRP k(g,b)}
The mapping class group $\MCG({S})$ has the cheap $\indice(g,b)$-rebuilding property. 
\end{theorem}
By \cite[Theorem 5.4.A]{Ivanov}, $\MCG({S})$ is of class $F_\indice$ for every $\indice \geq 0$.
Applying Proposition \ref{prop-CheapRebTorsion} finishes the proof of Theorem \ref{T0000}. 
\hfill $\square$

\medskip
Curiously we had a hard time trying to apply our method to $Out(\mathbf{F}_n)$, thus our question:
\begin{question}
What is the range of $\indice$ for which $Out(\mathbf{F}_n)$ has the cheap $\indice$-rebuilding property?
\end{question}

\bibliographystyle{alpha}

\newcommand{\etalchar}[1]{$^{#1}$}

\end{document}